\documentclass[11pt, twoside]{article}

\usepackage[english]{babel}

\usepackage{amssymb}
\usepackage{amsmath}
\usepackage{amsthm}
\usepackage{mathrsfs}
\usepackage{titletoc}
\usepackage{color}
\usepackage{algorithm}
\usepackage{algorithmic}
\usepackage{graphicx}
\usepackage{subfigure}
\usepackage{booktabs}
\usepackage{tabularx}
\usepackage{txfonts}
\usepackage{indentfirst}
\usepackage{anysize}

\usepackage{latexsym}

\usepackage[colorlinks=true,
  linkcolor=blue,
  citecolor=red,
  urlcolor=magenta,
  ]{hyperref}

\allowdisplaybreaks

\newtheorem{theorem}{Theorem}[section]
\newtheorem{lemma}[theorem]{Lemma}

\newtheorem{proposition}[theorem]{Proposition}
\theoremstyle{definition}
\newtheorem{remark}[theorem]{Remark}
\newtheorem{example}[theorem]{Example}
\newtheorem{definition}[theorem]{Definition}
\renewcommand{\appendix}{\par
   \setcounter{section}{0}%
   \setcounter{subsection}{0}%
   \setcounter{subsubsection}{0}%
   \gdef\thesection{\@Alph\c@section}%
   \gdef\thesubsection{\@Alph\c@section.\@arabic\c@subsection}%
   \gdef\theHsection{\@Alph\c@section.}%
   \gdef\theHsubsection{\@Alph\c@section.\@arabic\c@subsection}%
   \csname appendixmore\endcsname
 }

\numberwithin{equation}{section}

\begin{document}

\arraycolsep=1pt

\title{\bf\Large A Novel Two-Parameter Penalty: Relaxation Degree Analysis
and Sparse Signal Recovery\footnotetext{\hspace{-0.35cm} 2020
{\it Mathematics Subject Classification}.
Primary 94A12; Secondary 90C26, 90C90, 94A15.
\endgraf {\it Key words and phrases.}
sparse signal recovery, restricted isometry property, relaxation degree,
two-parameter nonconvex penalty, iteratively re-weighted least square.}}
\author{Ziwei Li, Wengu Chen, Huanmin Ge and Dachun
Yang\footnote{Corresponding author,
E-mail: \texttt{dcyang@bnu.edu.cn}/{\color{red}\today}/Final version.}
}
\date{}
\maketitle

\vspace{-0.8cm}

\begin{center}
\begin{minipage}{13cm}
{\small{\textbf{Abstract}}\quad
In this article, we introduce a  nonconvex two-parameter penalty function $P_{a,p}$,
parameterized by $a\in(0,\infty)$ and $p\in(0,1]$,
and the relaxation degree RD$_P$ for a separable nonconvex penalty function $P$.
Based on $P_{a,p}$, we further  propose the $P_{a,p}$ minimization framework for sparse signal recovery.
This framework generalizes the TL1 minimization model established by S. Zhang and J. Xin
(corresponding to the special case $p=1$) and provides a unified family of
nonconvex penalty functions for sparse signal recovery.
Using the sparse convex-combination technique,
we  establish both exact and stable sparse signal recovery
under the restricted isometry property (RIP).
To efficiently solve the resulting nonconvex optimization problem,
we apply a modified iteratively re-weighted least squares
method and the difference of convex functions algorithm (DCA)
to develop the IRLSTLp algorithm for unconstrained
$P_{a,p}$ minimization and prove some convergence results.
Finally, numerical experiments compare recovery and runtime against single- and
double-loop baselines, examine sensitivity to $(a,p)$, and assess the empirical
correlation between RD$_{P_{a,p}}$ and recovery within the proposed family.
The novelty of these results lies in three aspects:
(i) The concept of the relaxation degree  can quantitatively  measure how closely
a separable nonconvex penalty function approaches $\ell_0$.
(ii) The two-parameter penalty function $P_{a,p}$
offers more flexibility and stronger sparsity-promotion
capability of the $P_{a,p}$ minimization,
compared with both the $\ell_p$ and the TL1 minimization.
(iii) The obtained RIP upper bound for signal recovery via $P_{a,p}$ minimization
is asymptotically optimal (in $a$) as it matches the sharp RIP
bound obtained by R. Zhang and S. Li
when $a\to \infty$ and, especially, when $p=1$, matches the well-known sharp bound
$\delta_{2s}<\frac{\sqrt{2}}{2}$.
}
\end{minipage}
\end{center}

\vspace{0.2cm}

\tableofcontents

\section{Introduction}

Compressed sensing (CS), an interdisciplinary field that bridges physics, medical imaging,
electrical engineering, and computer science, has attracted significant attentions since around 2004.
It relies on a fundamental principle that many signals in the real world are sparse,
making it possible to reconstruct them using far fewer measurements than required by traditional sampling theory.

Let us consider a sparse signal $x\in\mathbb{R}^N$, a sensing matrix $A\in \mathbb{R}^{M\times N}$, and an observation $y\in\mathbb{R}^M$
with $M\ll N$ such that $y=Ax$ or $y=Ax+ \xi$ within some noisy measurements.
A vector $x\in\mathbb{R}^N$ is said to be \emph{$s$-sparse} for some natural number $s$ if
$|\mathrm{supp}(x)| \le s\ll N.$
We would like to search for the sparsest vector, which is equivalently to solve
the following  $\ell_0$ minimization problem:
\begin{equation}\label{eq-Pl0}
\min_{x\in\mathbb{R}^N} \|x\|_0 \quad\text{subject to}\ \ Ax=y
\end{equation}
or the  $\ell_0$ minimization problem within noisy measurements:
\begin{equation}\label{eq-Pl0epsilon}
\min_{x\in\mathbb{R}^N} \|x\|_0 \quad\text{subject to}\ \ \left\|Ax-y \right\|_2\le \epsilon\
\text{with}\ \epsilon\in(0,\infty);
\end{equation}
here and hereafter, for any given $p\in[0,\infty]$ and any $x\in\mathbb{R}^N$,
\begin{align}\label{lp}
\|x\|_{\ell_p}:=\|x\|_{p}:=
\begin{cases}
\left|{\mathop\mathrm{\,supp\,}}(x)\right|\quad& \text{if}\ p=0,\\
\displaystyle\left(\sum_{i=1}^N |x_i|^p\right)^{\frac1p}\quad& \text{if}\ p\in(0,\infty),\\
\displaystyle\max_{1\le i \le N} |x_i|\quad& \text{if}\ p=\infty.
\end{cases}
\end{align}

Since solving such an $\ell_0$ minimization problem is NP-hard in general
(see, for instance, \cite[Section 2.3]{FR13}),
a widely adopted strategy to circumvent this difficulty is
the convex relaxation via $\ell_1$ minimization, also known as Basis Pursuit.
Donoho \cite{D06} and Cand\`{e}s  et al. \cite{CRT06a, CRT06b, CT05}
independently  made seminal contributions to this field,
establishing foundational theory for exact and stable signal recovery.
Their work has paved the way for a broad class of relaxation-based methods;
see, for example, \cite{CLW18,HLSVS}.

Nonconvex relaxation-based methods have been further developed to enhance the sparsity of solutions.
One extensively studied approach is the
$\ell_p$ relaxation with $p\in(0,1)$ (see, for example,
\cite{HLZLT25,S12,WDZ15,WC13,XW13,ZLWX14}),
where the case $p=\frac{1}{2}$ is often regarded as the representative
of the $\ell_{p}$ penalty class and has been specifically examined
in works such as \cite{XZWCL10,ZLWX14}.
Further extensions include  weighted $\ell_p$
(\cite{CL19,GCN20,GCN21b}),
transformed $\ell_1$ (TL1)
(\cite{LF09,ZX17,ZX18}),  and
$\ell_1-\ell_2$ (\cite{LYHX15,YLHX15,ZW25}).
Empirical evidence suggests that these nonconvex relaxation
methods generally outperform standard $\ell_1$
relaxation although their effectiveness depends on the design
of suitable numerical algorithms.
These relaxation methods are also applied to low-rank matrix
recovery and low-rank tensor recovery;
relevant studies can be found in
\cite{FR13,HL25,YY23}.

A commonly used framework for sparse recovery is
the restricted isometry property (RIP) of sensing matrices, introduced by
Cand\`{e}s   et al.  in \cite{CT05} (see Definition \ref{def-RIP}). Since its emergence,
various sufficient conditions for RIP have been progressively
refined in the literature; see
\cite{CWX10,CZ13,CS08,LG26,ZL18}.
By utilizing convex combinations of sparse vectors for $\ell_1$ penalty, Cai and Zhang \cite{CZ14} established
the sharp RIP upper bound $\delta_{ts}<\sqrt{\frac{t-1}{t}}$ when $t\ge \frac{4}{3}$,
in the sense that, for any $\varepsilon>0$, the condition
$\delta_{ts}<\sqrt{\frac{t-1}{t}}+\varepsilon$
cannot guarantee the exact $s$-sparse signal recovery.
Zhang and  Li \cite{ZL19} extended this work to
$\ell_p$ penalty with $p\in(0,1]$, obtaining
the optimal RIP upper bound $\delta_{2s}< \delta(p)$,
where $\delta(p):=\frac{\eta}{2-p-\eta}$ and $\eta$ is
the unique positive solution of the equation
$$\frac{p}{2}\eta^{\frac{2}{p}}+\eta-1+\frac{p}{2}=0.$$

Since both $\ell_p$
  and TL1 belong to the family of single-parameter penalty functions, their intrinsic structural constraints necessarily limit their modeling flexibility and adaptability across diverse signal recovery scenarios. To address these limitations, we propose a novel \emph{two-parameter penalty function $P_{a,p}$}, governed by parameters $a\in(0,\infty)$ and $p\in(0,1]$,
  defined by setting, for any $x:=(x_1,\ldots,x_N)\in\mathbb{R}^N$,
$$P_{a,p}(x):=\sum_{i=1}^N\frac{(a+1)|x_i|^p}{a+|x_i|^p}.$$
 This formulation is essentially distinct from both $\ell_p$
and TL1. Rather than simply extending a single-parameter structure, $P_{a,p}$
multiplicatively couples a rational damping mechanism (inherited from TL1)
with a fractional-power nonlinearity (inherited from $\ell_p$).
This coupling produces a penalty landscape that is richer and more
expressive than either predecessor could achieve alone.

{\color{red}
Moreover, to compare the high-dimensional geometric landscapes of  penalties,
we introduce the concept of the relaxation degree $\mathrm{RD}_P$
for a separable penalty function $P$,
which is precisely the Euclidean distance from
the origin to the intersection point of the unit level set
$\{x\in \mathbb{R}^N: P(x)=1\}$ with the diagonal line
$\{x:=(x_1,\ldots,x_N)\in\mathbb{R}^N: x_1=\cdots=x_N \}$ of
$\mathbb{R}^N$. It is especially useful when two-dimensional level lines are visually
indistinguishable, as occurs for the proposed $P_{a,p}$ function and the
$\ell_a^p$ pseudo-norm introduced in \cite{YY23}.}

We then establish the exact and the stable recovery results for
$P_{a,p}$ minimization under the suitable RIP condition of sensing matrices,
in which we use a normalization step to overcome the absence
of the scaling property of the $P_{a,p}$ function.
We also  employ a combination of both a modified IRLS method and the DC method
to provide an overall algorithm for the $P_{a,p}$ minimization and
show some convergence results. Finally, we
conduct some numerical experiments to show the robustness of the IRLSTLp and
the flexibility of the $P_{a,p}$ minimization framework.

The novelty of this article lies in the following three aspects:

\begin{enumerate}
  \item[(i)] {\color{red}
The proposed relaxation degree RD$_P$ provides a specific, dimension-dependent feature
of the unit level set. We prove its compatibility with the Robin Hood criterion
for strictly concave separable penalties and use numerical tests to assess its
empirical correlation with recovery performance within the $P_{a,p}$ family.}

  \item[(ii)] The two-parameter  penalty function $P_{a,p}$ contains two adjustable parameters
  $a\in(0,\infty)$ and $p\in(0,1]$, which separately control saturation and
  fractional-power curvature. This construction contains TL1 as the case $p=1$
  and approaches the $\ell_p$ penalty as $a\to\infty$.
  \item[(iii)]
  The obtained RIP upper bound for signal recovery via $P_{a,p}$ minimization
is \emph{asymptotically optimal} (in $a$) as it can reduce to the sharp RIP bound obtained by Zhang and Li \cite{ZL19}
when $a\to \infty$ and, especially, when $p=1$, can recover the well-known sharp bound
$\delta_{2s}<\frac{\sqrt{2}}{2}$.
\end{enumerate}

The remainder of this article is organized as follows.

In  Section \ref{sec-theory}, we propose the concept of the relaxation degree RD$_P$ of a
separable penalty function $P$, introduce a two-parameter nonconvex  penalty function $P_{a,p}$,
propose the $P_{a,p}$ minimization framework,
and establish the exact and the stable $P_{a,p}$ sparse recovery
by using the sparse convex-combination technique
developed by Cai and Zhang in \cite{CZ14} and, independently, G. Xu and Z. Xu in
\cite{xx13} (for $\ell_1$) and Zhang and  Li in \cite{ZL19} for $\ell_p$ with $p\in (0,1]$.
Section \ref{sec-algorithm} is dedicated to the related algorithm.
A modified iteratively reweighted least squares (IRLS) method is employed for the outer loop,
while the difference of convex functions (DC) method is adopted
to solve the resulting sub-problem in the inner loop.
Some related convergence results of both the outer-loop modified IRLS algorithm
and the inner-loop DCA procedure are also given.
In Section \ref{sec-test}, we present the results of our numerical experiments.
To be precise, the performance of the proposed IRLSTLp
algorithm is evaluated with different
parameters $a$ and $p$.  Some experiments are also conducted
to analyze the utility of the relaxation degree.
Two representative classes of matrices, Gaussian matrices and
over-sampled discrete cosine transform (DCT) matrices, are used to test
the algorithm's performance under varying degrees of matrix coherence.
Additionally, comparisons are conducted with DCA of TL1 \cite{ZX18},
DCA of $\ell_1-\ell_2$ \cite{YLHX15}, and IRLS of $\ell_{p}$ \cite{LXY13}.
Finally, we conclude this article in Section \ref{s5}.

We end this introduction by making some notational conventions.
Throughout this article, we let $\mathbb{N}$ be the set
of all positive integers and  $M,N\in\mathbb{N}$ be two given natural numbers
which are always used to denote the dimensions of vectors or matrices,
and let $a\in(0,\infty)$ and $p\in[0,1]$ be two
given parameters which are related to penalty functions.
We use $\mathbf{0}$ to denote the zero vector of
the Euclidean space.
For any given $x\in\mathbb{R}^N$, $A\in\mathbb{R}^{M\times N}$, and any given subset $T$ of $\{1,\ldots,N\}$,
we always denote by $x_T$ the vector in $\mathbb{R}^N$
which coincides with $x$ on the components with the index in $T$ and is zero on the components  with the index
outside $T$,  and also denote by $A_T$ the matrix in $\mathbb{R}^{M\times N}$
which coincides with $A$ on the columns with the index in $T$ and is zero outside.
The notation $\mathrm{Ker}\, A$ is always used
to denote the \emph{kernel  space} (also called the \emph{null space}) of $A$.
Let $r(x)$  denote the \emph{decreasing rearrangement} of $x$
and, for any $i\in\{1,\ldots,N\}$, $r(x)_{i}$ denote the $i$-th
largest component  in magnitude.
We further define
\begin{equation}\label{eq-1.4}
\sigma_j(x)_{1}:=\sum_{i=j+1}^N r(x)_i,  \quad \forall j\in\{0,\ldots,N-1\}.
\end{equation}
Also, $\lim_{a\to 0^+}$ means that there exists $c_0\in (0,1)$
such that $a\in (0, c_0)$ and $a\to 0$.
For any set $T$, we use $T^\complement$ to denote its complement.
Finally, in all proofs we
consistently retain the notation introduced
in the original theorem (or related statement).

\section{$P_{a,p}$ Minimization Framework}\label{sec-theory}

This section consists of four subsections. In Subsection \ref{sebsec-degree},
we introduce the concept of the relaxation degree RD$_P$ of a
separable penalty function $P$ to quantitatively measure how closely the penalty function $P$
approximates $\ell_0$ (Definition \ref{rd}). Then, in Subsection \ref{sebsec-TLp},
we introduce the $P_{a,p}$ penalty function in \eqref{2.4x}
and compare its relaxation degree with some known penalty functions.
Moreover, we give some sufficient conditions based on RIP
(Theorems \ref{thm-exact}, \ref{thm-stable}, and \ref{thm-stable,ep})
to guarantee the $P_{a,p}$ sparse recovery in Subsection \ref{subsec-main},
while their proofs are presented in Subsection \ref{subsec-proof}.

\subsection{Relaxation Degree RD$_P$}\label{sebsec-degree}

We begin with a concept of permutation invariant functions.
A function $P$ is said to be \emph{permutation invariant} if, for any  permutation
$\sigma:\ \{1,\ldots,N\}\mapsto\{1,\ldots,N\}$ and any
$x\in\mathbb{R}^N$,
$$
P\left(x_1,\ldots,x_N\right)=P\left(x_{\sigma(1)},\ldots,x_{\sigma(N)}\right).
$$
For a  given permutation invariant penalty function $P$ and a minimization problem
\begin{equation}\label{Ae}
\min_x P(x)\quad\text{subject to}\ \ Ax=Ae,
\end{equation}
where $e$ is a true signal,
the exact recovery occurs  when the constraint hyperplane $Ax=Ae$
intersects with the $P$ sphere (level line) containing $e$ only at $e$; see Figure \ref{fig-exact}(a).
When $e$ is sparse, this may happen for many choices of $A$.
If we only consider the region near $e$, as is presented in Figure \ref{fig-exact}(b),
more $A$ can be chosen;
furthermore, let us compare two penalty functions with their level lines $\mathrm{I}$ and $\mathrm{II}$.
The level line $\mathrm{I}$ is more attracted by the axes, which implies that,
for a given sensing matrix $A$, the penalty with
level line I is more likely to achieve the exact recovery.

These observations inspire us to compare the approximation degree for the level lines
of different penalty functions to approach coordinate axes,
via the ratio of the distance from the origin to the level line containing some given sparse vector $e$ along the diagonal direction
and the distance from the origin to $e$.

\begin{figure}[ht]
  \centering
  \begin{tabular}{cc}
  \includegraphics[width=4.5cm]{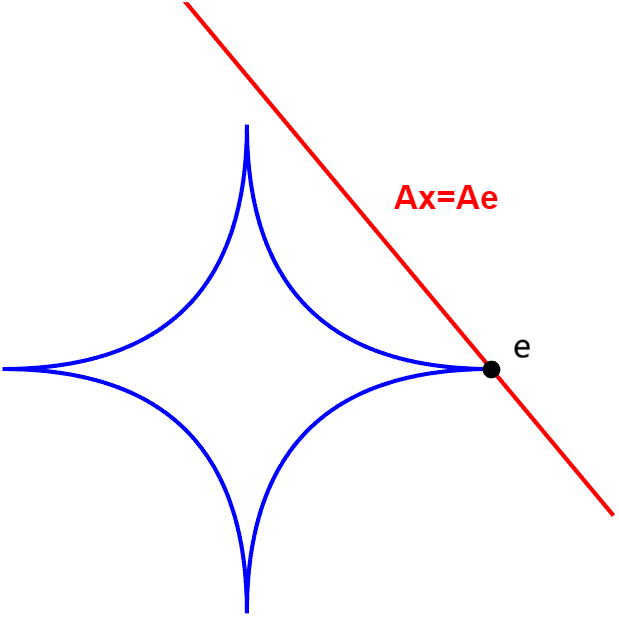}&
  \includegraphics[width=6cm]{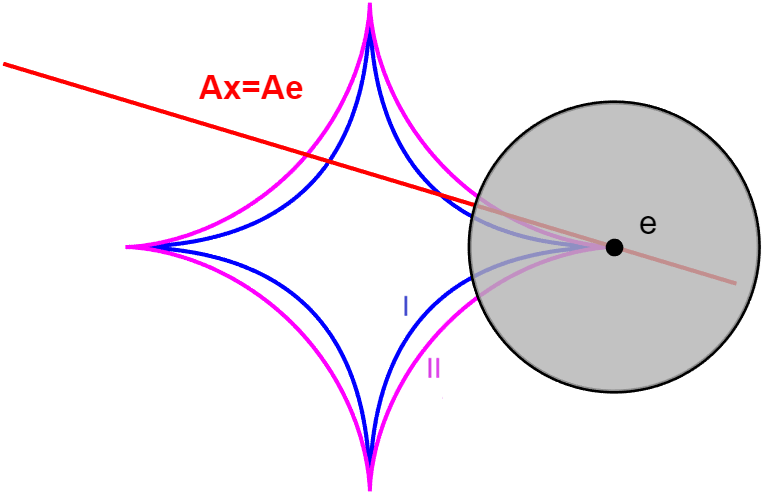}\\
  (a) & (b)
  \end{tabular}
  \caption{Exact recovery for sparse $e$}\label{fig-exact}
\end{figure}

In what follows, we always let
\begin{equation}\label{sp}
P(x):=\rho(|x_1|)+\cdots+\rho(|x_N|),\
\forall\, x:=(x_1,\ldots,x_N)\in\mathbb{R}^N,
\end{equation}
be a \emph{separable function} with $\rho$
being continuous, concave, and {\color{red}strictly}
increasing in $[0,\infty)$, $\rho(0)=0$, and $\rho(1)=1$.
Obviously, such $P$ is permutation invariant. {\color{red}
Observe that the function $\rho$ can be unbounded or can have a horizontal asymptote.
For example, $\rho(t):=t^p$ is unbounded, while the proposed function in
\eqref{eq-def-rho} satisfies $\lim_{t\to\infty}\rho_{a,p}(t)=a+1$. }

For this class of functions, we introduce a  concept of the relaxation degree.

\begin{definition}\label{rd}
Let $P$ be a separable penalty function as in \eqref{sp}
and $x_{\mathrm{diag}}$ be the intersection point
of the hyper-surface $P(x)=1$ with
the diagonal line
$\pi:= \{x\in\mathbb{R}^N:\ x_1=\cdots=x_N \}.$
Then the \emph{relaxation degree} RD$_P$ of $P$ is
defined as $\mathrm{RD}_P:=\|x_{\mathrm{diag}}\|_2.$
\end{definition}

\begin{remark}
\begin{enumerate}
\item[\rm(i)]
By the separability of $P$,  the intersection points of a given
level line with each coordinate axis own the same distance from the origin.
Let $\{\vec{e}_j\}_{j=1}^N$ be the standard orthonormal basis of $\mathbb{R}^N$.
In \eqref{Ae}, by choosing $e:=\vec{e}_j$ for some $j\in\{1,\ldots,N\}$,
$\mathrm{RD}_P$ is exactly the ratio of the distance from the origin to the level line $P(x)=P(e)$ along the diagonal direction
and the distance from the origin to $e$.
\item[\rm(ii)]
{\color{red}The unit level set of $\|\cdot\|_0$ is supported
on the coordinate axes away from the origin. A smaller RD$_P$ therefore records
a smaller diagonal intercept, which
motivates the term ``relaxation degree''. }
\end{enumerate}
\end{remark}

{\color{red} Hurley and Rickard \cite{HR09} compare 15 sparsity measures through six criteria:
Robin Hood, Scaling, Rising Tide, Cloning, Bill Gates, and Babies.
The following proposition shows that the function $-P$
with $P$ in \eqref{sp} and with $\rho$ therein strictly concave
satisfies the Robin Hood inequality, i.e., \eqref{eq-RobinHood-penalty}.
Moreover, from the strict concavity of $\rho$, we deduce a range of RD$_P$.
In what follows, we use the notation $\mathbb{R}_+^N:=[0,\infty)^N$.

\begin{proposition}\label{prop-RobinHoodRD}
Assume that $N\ge2$ and that $P$ is in \eqref{sp} with $\rho$ strictly concave.
Then, for any $x:=(x_1,\ldots,x_N)\in\mathbb{R}_+^N$,
any $x_i,x_j$ with $x_i>x_j$,
and any $\alpha\in(0,\frac{x_i-x_j}{2})$,
\begin{align}\label{eq-RobinHood-penalty}
\begin{cases}
-P(x)>-P\left(x_1,\ldots,x_i-\alpha,\ldots,x_j+\alpha,\ldots,x_N\right) &\mathrm{if}\ i<j,\\
-P(x)>-P\left(x_1,\ldots,x_j+\alpha,\ldots,x_i-\alpha,\ldots,x_N\right) &\mathrm{if}\ j<i.
\end{cases}\end{align}
Moreover,
\begin{equation}\label{eq-RD-strict-bound}
\mathrm{RD}_{P}\in\left(0,\frac{1}{\sqrt{N}}\right).
\end{equation}
\end{proposition}

\begin{proof}
Since $P$ is separable, it suffices to prove that,
for any $v\in[0,\infty)$, $u\in(v,\infty)$, and
$\alpha\in(0,\frac{u-v}{2})$,
\begin{equation*}
\rho(u-\alpha)+\rho(v+\alpha)>\rho(u)+\rho(v).
\end{equation*}
Indeed, let $\theta:=\frac{\alpha}{u-v}\in(0,\frac{1}{2})$. Note that
$u-\alpha=(1-\theta)u+\theta v$ and
$v+\alpha=\theta u+(1-\theta)v$.
Applying the strict concavity to these
two convex combinations, we obtain
$$\rho(u-\alpha) + \rho(v+\alpha) > [(1-\theta) + \theta]\rho(u) + [\theta + (1-\theta)]\rho(v)=\rho(u)+\rho(v),$$
which implies
\eqref{eq-RobinHood-penalty}.
Next, let $x_{\mathrm{diag}}:=(t_N,\ldots,t_N)$. The normalization gives
$\rho(t_N)=\frac{1}{N}$. By the strict concavity, $\rho(0)=0$, and $\rho(1)=1$,
we have $t<\rho(t)$ when $t\in(0,1)$. From this and $t_N\in(0,1)$, it follows that
$t_N\in(0,\frac{1}{N})$. Thus,
$\mathrm{RD}_P=\sqrt{N}t_N\in(0,\frac{1}{\sqrt{N}})$, which completes the proof.
\end{proof}

Equation \eqref{eq-RobinHood-penalty} says that transferring magnitude from a
larger coefficient to a smaller one with their sum preserved increases $P$ and hence decreases $-P$,
which coincides with a decrease of the sparsity measure in the sense of the Robin Hood criterion.
Thus, strict concavity favors the more unequal and hence
sparser signals. Estimate
\eqref{eq-RD-strict-bound} is the associated consequence on the normalized level set.}

{\color{red}\begin{remark}
We further examine whether the function $-P$ with $P$ as in \eqref{sp}
satisfies the other five criteria in \cite{HR09}.
To this end, let $\rho$ be the same as in \eqref{sp} and, for $m\in\mathbb{N}$,
define $S_m(x):=-\sum_{i=1}^m \rho(x_i),\ \forall\,x:=(x_1,\ldots,x_m)
\in\mathbb{R}_+^m$. Obviously,
$S_N=-P$.
\begin{enumerate}
\item[\rm(i)] Scaling requires $S_N(\lambda x)=S_N(x)$ for any $x\in \mathbb{R}_+^N$ and $\lambda>0$.
This fails because $P(\lambda x)>P(x)$ whenever $x\ne\mathbf{0}$ and $\lambda>1$.
\item[\rm(ii)] Rising Tide requires $S_N(x+(t,\ldots,t))<S_N(x)$
for any $x\in \mathbb{R}_+^N$ and $t>0$
when the components of $x$ are not all equal.
This holds because $\rho$ is strictly increasing.
\item[\rm(iii)] Cloning requires invariance under concatenating
identical copies, i.e., for any $x\in \mathbb{R}_+^N$ and   $K\in\mathbb{N}$ with $K\ge2$,
$S_{KN}(x,\ldots,x)=S_N(x)$.
This fails because, for any $x\in \mathbb{R}_+^N\setminus\{\mathbf{0}\}$,
$KP(x)\ne P(x)$.
\item[\rm(iv)] Bill Gates requires that, for any $x\in \mathbb{R}_+^N$ and $i\in\{1,\ldots,N\}$, there exists $\beta_i>0$
such that, for any $\alpha>0$,  $S_N(x+(\beta_i +\alpha)e_i)>S_N(x+\beta_i e_i)$,
where $\{e_i\}_{i=1}^N$ is the standard orthogonal basis of $\mathbb{R}^N$.
This fails by the strict increase of $\rho$.
\item[\rm(v)] Babies  requires $S_{N+1}((x,0))>S_N(x)$ for any $x\in \mathbb{R}_+^N\setminus\{\mathbf{0}\}$.
This also fails because  $\rho(0)=0$.
\end{enumerate}
\end{remark}}

Finally, we clarify how RD$_P$ measures the
approximation degree for the given penalty function $P$
to approach $\ell_0$ through an example.

\begin{example}
Let $p\in(0,1]$ and
\begin{align}\label{pp}
P_p(x):=\|x\|^p_{\ell_p}
\end{align}
for any $x\in\mathbb R^N$,
where $\|\cdot\|_{\ell_p}$ is as in \eqref{lp}.
It is easy to prove $\mathrm{RD}_{P_p}=N^{\frac 12-\frac 1p}$.
We find that $\mathrm{RD}_{P_p}$ decreases to $0$ as $p$ decreases to $0$,
which coincides with the known fact that $\ell_p$ penalty can acquire the sparsity
promotion as $p$ decreases. Besides, it is also known that $\ell_p$ penalty
when $p>1$ cannot recover a sparse solution
and, naturally, $\ell_1$ relaxation is called the \emph{tightest convex relaxation}.
{\color{red}Here,
$\mathrm{RD}_{P_1}=\frac{1}{\sqrt{N}}$ and
$\mathrm{RD}_{P_p}\in(0,\frac{1}{\sqrt{N}}]$.
This also implies that $\frac{1}{\sqrt{N}}$ serves as the \emph{endpoint} attained by the convex $\ell_1$
penalty within this $\ell_p$ family.}
\end{example}

{\color{red}
It is also important to emphasize that $\mathrm{RD}_P$
can quantitatively  measure how closely
a separable nonconvex penalty function approaches $\ell_0$.
However, it cannot independently determine the recovery performance.
To illustrate this, we provide the following counterexample,
showing that two penalty functions
with the same relaxation degrees may exhibit completely different recovery behaviors.}

{\color{red}\begin{example}\label{ex-counter}
Let $N=3$ and consider two separable penalty functions $P_1$ and $P_2$, generated respectively by
$$
\rho_1(t) := \sqrt{t} \quad \text{and} \quad \rho_2(t) := \frac{4t}{1+3t},\ \ \forall\,t \in [0,\infty).
$$
Both $\rho_1$ and $\rho_2$ are continuous, strictly increasing, and strictly concave on $[0,\infty)$, satisfying $\rho_1(0)=\rho_2(0)=0$ and $\rho_1(1)=\rho_2(1)=1$. Moreover,
the intersection points of their unit level sets with the diagonal line $x_1=x_2=x_3$ are both given by $(\frac19,\frac19,\frac19)$.
Consequently, their corresponding penalty functions $P_{1}$ and $P_{2}$ have the same relaxation degree, namely $\mathrm{RD}_{P_{1}} = \frac{\sqrt{3}}{9}= \mathrm{RD}_{P_2}.$
Now, let us consider a sensing matrix
$A:= \begin{pmatrix} 1 & -9 & 0 \\ 1 & 0 & -9 \end{pmatrix}$
with $\mathrm{Ker}\, A = \mathrm{span}\{(9, 1, 1)^T\}$.
Suppose that the true sparse signal to be recovered is $x_0 := (-9, 0, 0)^T$.
 Any feasible solution $x(\tau)$ satisfying $Ax = Ax_0$ can be parameterized as
$$
x(\tau) = (-9, 0, 0)^T+ \tau(9, 1, 1)^T = (9\tau-9, \tau, \tau)^T, \quad \forall\,\tau \in \mathbb{R}.
$$
On the one hand, for $P_1$ we have $P_{1}(x_0)  = 3$ and
$P_{1}(x(1))= 2.$ Thus, $P_{1}(x(1)) < P_{1}(x_0) $, which implies that the $P_1$ minimization fails to exactly recover $x_0$.
On the other hand, for $\tau \in \mathbb{R}$ we have
$$
F_2(\tau) := P_{2}(x(\tau)) = \rho_2(9|\tau-1|) + 2\rho_2(|\tau|).
$$
By calculating its derivative in $\tau \in [0,1]$,
we find that $F_2$ has a local maximum in $(0,1)$,
is decreasing in $(-\infty,0]$, and is increasing in $[1,\infty)$.
Combining these, $F_2(0) =\frac 97$, and $F_2(1) = 2$,
we conclude that $\tau=0$ is the unique global minimum of $F_2$.
Consequently, the $P_2$-minimization exactly recovers $x_0$.
This demonstrates that  RD$_P$ cannot independently
determine the recovery performance.
\end{example}}

\subsection{Two-Parameter Penalty $P_{a,p}$}\label{sebsec-TLp}

In this article,
we consider the following nonnegative function
\begin{equation}\label{eq-def-rho}
\rho_{a,p}(t):=\frac{(a+1)|t|^p}{a+|t|^p},\quad\forall\,t\in\mathbb{R}
\end{equation}
with two parameters $a\in(0,\infty)$ and $p\in(0,1]$.

We first give some properties of $\rho_{a,p}$.

\begin{lemma}\label{lem-TLp,prop}
Let $a\in(0,\infty)$ and $p\in(0,1]$. The function $\rho_{a,p}$
in \eqref{eq-def-rho} has the following properties:
\begin{enumerate}
\item[\rm(i)] For any $t\in\mathbb{R}$, $\rho_{a,p}(t)=\rho_{a,1}(|t|^p).$
\item[\rm(ii)] $\rho_{a,p}$ is strictly increasing and concave on $[0,\infty)$
with
$$
\rho_{a,p}(0)=0,\ \ \rho_{a,p}(1)=1,\ \ \text{and}\ \ \lim_{t\to\infty}\rho_{a,p}(t)=a+1.
$$
\item[\rm(iii)] For any $t\in\mathbb{R}$,
$\rho_{a,p}(t)\le \frac{a+1}{a}|t|^p;$
$|t|^p\le\rho_{a,p}(t)\le1$ if and only if $|t|\le 1$.
\item[\rm(iv)] $\rho_{a,p}'$ is continuous on $(0,\infty)$; moreover,
$$
\lim_{t\to0}\rho_{a,p}'(t)=
\begin{cases}
\infty\quad&\text{when}\  p\in(0,1), \\
\displaystyle\frac{a+1}{a}\quad&\text{when}\  p=1
\end{cases}
\ \ \text{and}\ \
\lim_{t\to\infty}\rho_{a,p}'(t)=0.
$$
\item[\rm(v)] For any $t\in\mathbb{R}$ and $c\in\mathbb{R}$,
\begin{align*}
\rho_{a,p}(ct)
\begin{cases}
\ge |c|^p \rho_{a,p}(t)\quad&\text{when}\  |c|\le 1, \\
\le |c|^p \rho_{a,p}(t)\quad&\text{when}\  |c|> 1.
\end{cases}
\end{align*}
\item[\rm(vi)] For any $t_1,t_2\in\mathbb{R}$,
\begin{align*}
|\rho_{a,p}(t_1)-\rho_{a,p}(t_2)|&\le
\rho_{a,p}(t_1+t_2)\le\rho_{a,p}(|t_1|+|t_2|)\le \rho_{a,p}(t_1)+\rho_{a,p}(t_2)\\
&\le2 \rho_{a,p}\left(\frac{|t_1|+|t_2|}{2}\right).
\end{align*}
\end{enumerate}
\end{lemma}

\begin{proof}
(i), (ii), and (iii) are obvious, (iv) is easy to be verified by calculating
$$
\rho_{a,p}'(t) =\frac{a(a+1)p|t|^{p-2}t}{(a+|t|^p)^2},
$$
 and
(v) can be  inferred from the observation
$$
\rho_{a,p}(ct)=|c|^p\rho_{a,p}(t)\frac{a+|t|^p}{a+|c|^p|t|^p}.
$$
Finally, we prove (vi).
Applying  the increasing property  of $\rho_{a,1}$ on $[0,\infty)$
and the fact that, for any $p\in(0,1]$, $x\ge0$, and $y\ge0$,
\begin{equation}\label{eq-elementary}
(x+y)^p\le x^p +y^p,
\end{equation}
we find that, for  any $t_1,t_2\in\mathbb{R}$,
\begin{align*}
\rho_{a,p}(t_1)+\rho_{a,p}(t_2)
&=(a+1)\frac{|t_1|^p +|t_2|^p +2a^{-1}|t_1|^p|t_2|^p}{a+|t_1|^p +|t_2|^p +a^{-1}|t_1|^p|t_2|^p}\\
&\ge\rho_{a,1}(|t_1|^p +|t_2|^p +a^{-1}|t_1|^p|t_2|^p)
\ge\rho_{a,1}(|t_1|^p +|t_2|^p)\\
&\ge \rho_{a,1}\left((|t_1| +|t_2|)^p\right)
=\rho_{a,p}(|t_1|+|t_2|).
\end{align*}
From this, the fact that $\rho_{a,p}$ is even, and
the increasing property of $\rho_{a,p}$ on $[0,\infty)$,
it then follows that
$$
\rho_{a,p}(t_1)+\rho_{a,p}(t_2) \ge \rho_{a,p}(|t_1|+|t_2|) \ge \rho_{a,p}(t_1+t_2).
$$
By this with $t_1$ and $t_2$ therein replaced, respectively, by $t_1+t_2$ and $-t_2$
and also by $\rho_{a,p}(t)=\rho_{a,p}(-t)$ for any $t\in\mathbb{R}$, we further obtain
$$
\rho_{a,p}(t_1)-\rho_{a,p}(t_2)\le\rho_{a,p}(t_1+t_2).
$$
Thus,
$$
\left|\rho_{a,p}(t_1)-\rho_{a,p}(t_2)\right|\le\rho_{a,p}(t_1+t_2).
$$
The remaining desired inequality follows from  the concavity of $\rho_{a,p}$.
This completes the proof of Lemma \ref{lem-TLp,prop}.
\end{proof}

Furthermore, let us define a separable two-parameter  penalty function
$P_{a,p}:\ \mathbb{R}^N\rightarrow [0,\infty)$ by setting,
for any $x:=(x_1,\ldots,x_N)\in\mathbb{R}^N$,
\begin{equation}\label{2.4x}
P_{a,p}(x):=\sum_{i=1}^N \rho_{a,p}(x_i)=\sum_{i=1}^N\frac{(a+1)|x_i|^p}{a+|x_i|^p}.
\end{equation}
One can verify that, for any $x\in\mathbb{R}^N$,
$$
\lim_{a\rightarrow0^+}P_{a,p}(x)=\|x\|_0 \ \ \text{and}
\ \
 \lim_{a\rightarrow\infty}P_{a,p}(x)=\|x\|_{p}^p.
$$
That is to say, by selecting the value of $a$, the $P_{a,p}$ function can approximate either $\|\cdot\|_0$ or $\|\cdot\|_{p}^p$.

\begin{proposition}\label{prop-RD}
Let $a\in(0,\infty)$, $p\in(0,1]$, and $P_{a,p}$ be as in \eqref{2.4x}.
Then
$$
\mathrm{RD}_{P_{a,p}}= \left[\frac{a}{(a+1)N-1}\right]^{\frac1p} \sqrt{N};
$$
moreover, $\mathrm{RD}_{P_{a,p}}$ is monotonously increasing on either $p$ or $a$.
\end{proposition}

Proposition \ref{prop-RD} follows from a straight calculation
and we omit the details.

Using Proposition \ref{prop-RD}, we find that
\begin{equation}\label{RDPap1}
\lim_{a\to\infty}\mathrm{RD}_{P_{a,p}}=N^{\frac 12-\frac 1p}
=\mathrm{RD}_{P_p},\ \lim_{a\to 0^+}\mathrm{RD}_{P_{a,p}}=0,
\end{equation}
and
\begin{equation}\label{RDPap2}
\lim_{p\to 0^+}\mathrm{RD}_{P_{a,p}}=0,
\end{equation}
where $P_p$ is as in \eqref{pp}.
These clearly indicate that, as $a\to\infty$, the $P_{a,p}$ penalty function
approaches $P_p$ and, as $a\to 0^+$ or $p\to 0^+$, the $P_{a,p}$ penalty function
approaches $\ell_0$.

\begin{figure}[h]
  \centering
  \begin{tabular}{cccc}
  \includegraphics[width=3.3cm]{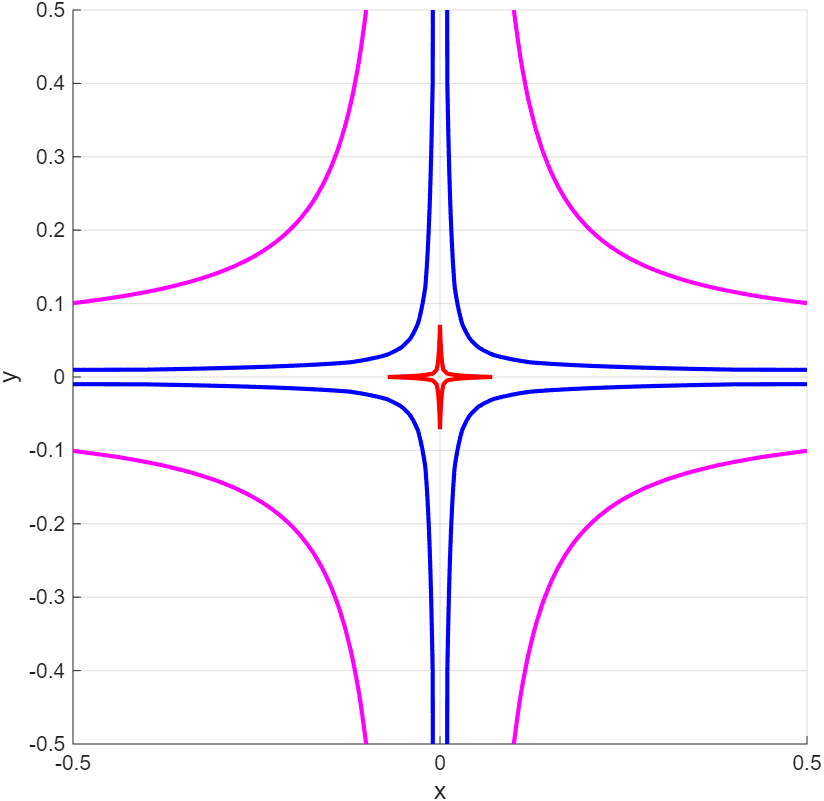}&
  \includegraphics[width=3.3cm]{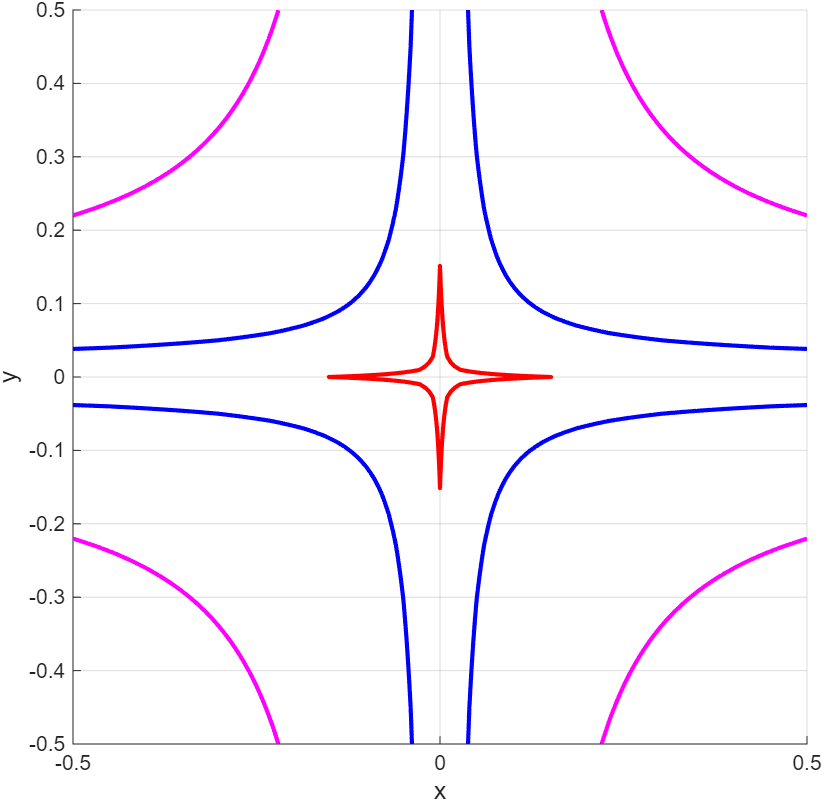}&
  \includegraphics[width=3.3cm]{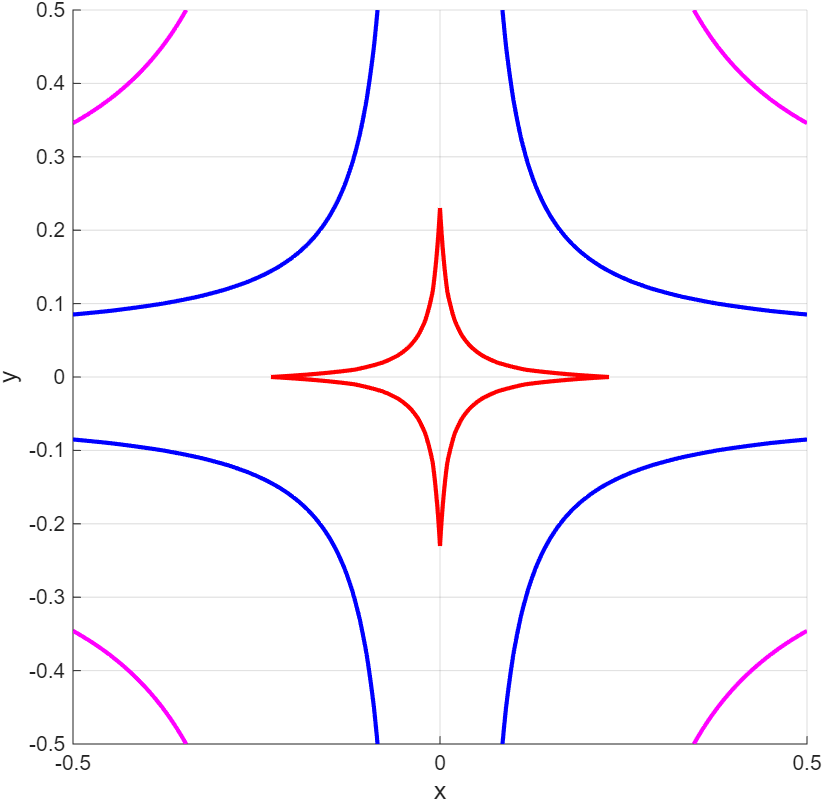}&
  \includegraphics[width=3.3cm]{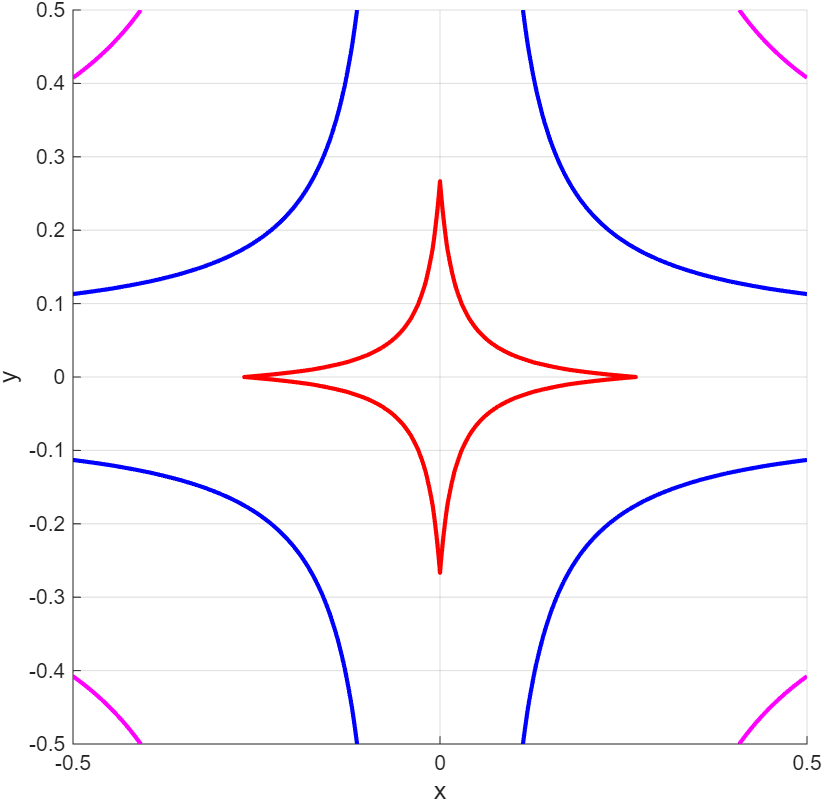}\\
  \quad$P_{0.1,0.5}$ & \quad$P_{0.1,0.7}$ & \quad$P_{0.1,0.9}$ & \quad$P_{0.1,1}$
  \end{tabular}
  \caption{Two-dimensional level lines of $P_{0.1,p}$ with $p\in\{0.5,0.7,0.9,1\}$}\label{fig-a=0.1}
\end{figure}

\begin{figure}[h]
    \centering
  \begin{tabular}{cccc}
  \includegraphics[width=3.3cm]{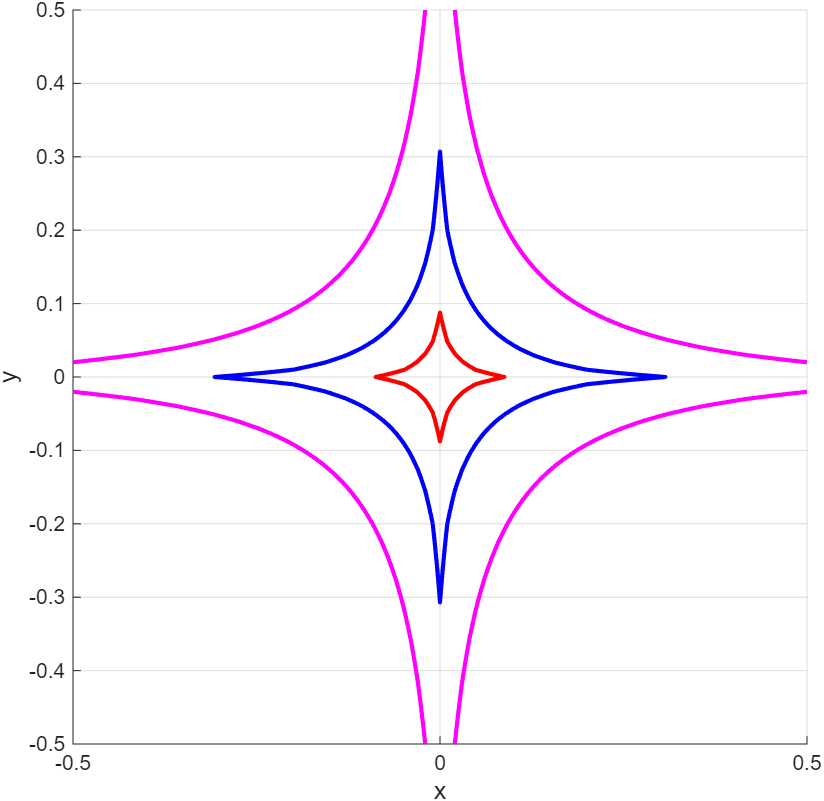}&
  \includegraphics[width=3.3cm]{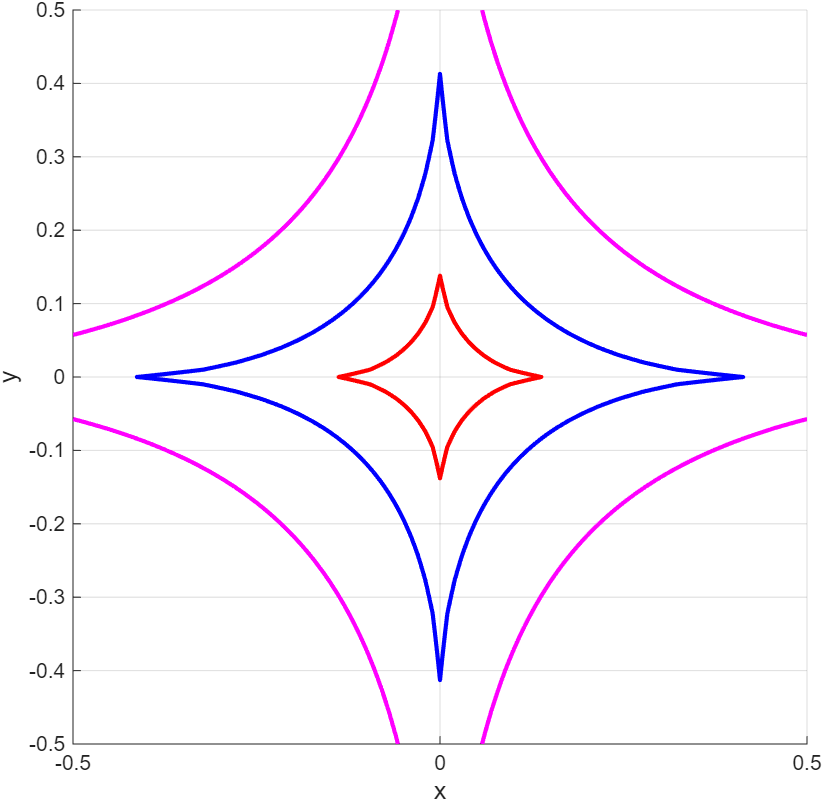}&
  \includegraphics[width=3.3cm]{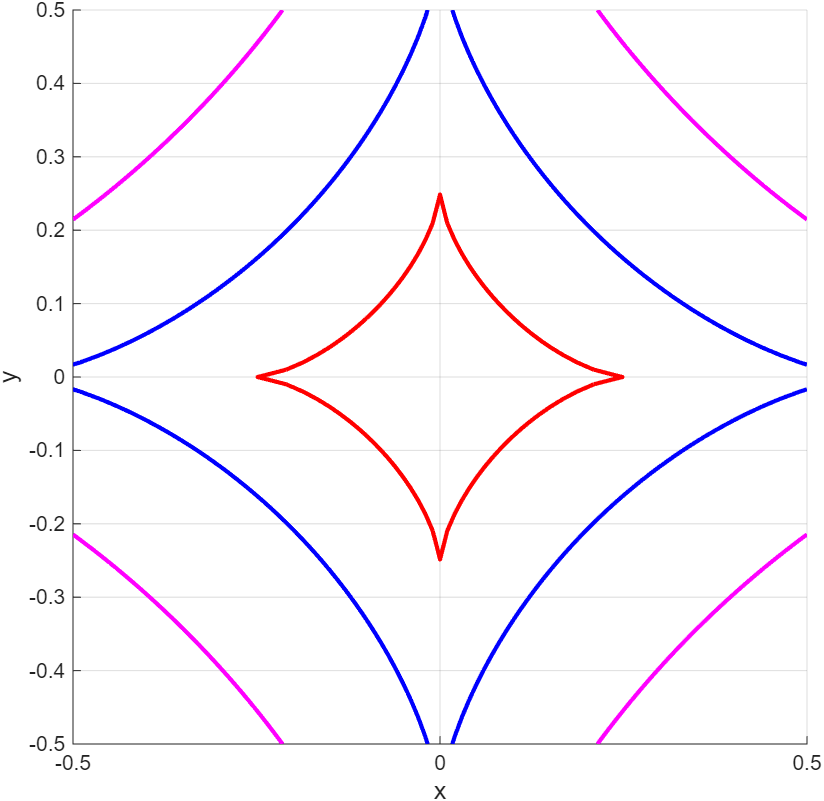}&
  \includegraphics[width=3.3cm]{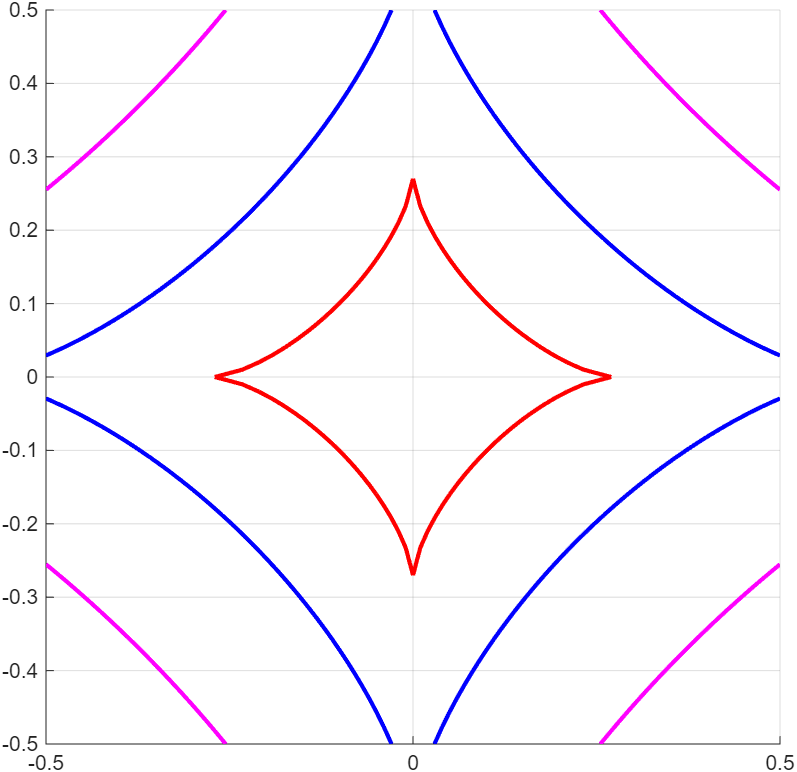}\\
  \quad$P_{0.5,0.7} $ &\quad$P_{1,0.7}$ & \quad $P_{10,0.7}$ & \quad $P_{0.7}$
  \end{tabular}
  \caption{Two-dimensional level lines of $P_{a,0.7}$ with $a\in\{0.5,1,10\}$ and  $P_{0.7}$}\label{fig-p=0.7}
\end{figure}

To visualize, we plot the \emph{two-dimensional} level lines
of several specific $P_{a,p}$ functions.
In Figure \ref{fig-a=0.1}, we compare the level lines of $P_{a,p}$
with $p=0.5$, $p=0.7$, $p=0.9$, or $p=1$ when $a=0.1$;
in Figure \ref{fig-p=0.7}, we compare the level lines of $P_p$ and of
$P_{a,p}$ with  $a=0.5$, $a=1$, or $a=10$ when $p=0.7$.
{\color{red} These figures show how the level lines of $P_{a,p}$ change as $a$ or $p$ varies.
To be precise, Figure \ref{fig-a=0.1} indicates that,
with $a=0.1$ fixed,
the curvatures near the origin
of level lines of $P_{0.1,p}$ become larger as $p$ decreases, while Figure \ref{fig-p=0.7} indicates that,
with $p=0.7$ fixed,
the level lines of $P_{a,0.7}$ become closer to the coordinate  axes
from positions nearer the origin as $a$ decreases.
In both cases, the level lines become progressively sharper
and compress more tightly towards  the coordinate axes,
which is visually consistent with a closer approximation
to the unit level set of $\ell_0$.}
{\color{red}
Unfortunately, such intuitive figures of level sets of penalty functions
in \emph{high-dimensional} cases do not exist.

However, RD$_P$, as formulated in  Proposition \ref{prop-RD},
does provide a  high-dimensional   metric to evaluate how closely
a penalty function approximates $\ell_0$, especially when visual distinctions fail.}
For instance, recall that the \emph{$\ell_a^p$ pseudo-norm} $\|\cdot\|_{\ell_a^p}$
is defined by setting
$$
\|x\|_{\ell_a^p}:=\sum_{i=1}^N \left[\frac{(a+1)|x_i|}{a+|x_i|}\right]^p,
\quad\forall\,x\in\mathbb{R}^N,
$$
which is a special case of \cite[Definition 2.2]{YY23} with
$r:=a$, $m:=N$, and $n:=1$. To compare it with the $P_{a,p}$ function,
we first plot their 2D level lines when $p=0.7$, $a=5$ and $p=0.7$, $a=1$ in Figure \ref{fig-TLp}.
However, we find that it is hard to distinguish the degree
to approximate $\ell_0$ of these two functions from this figure.
Instead, we calculate its relaxation degree,
\begin{equation}\label{e2.9}
\mathrm{RD}_{\ell_a^p}=\frac{a }{(a+1)N^{\frac{1}{p}}-1}\sqrt{N}.
\end{equation}
Let us consider the case  $N=512$. In this case,
$$\mathrm{RD}_{\ell_5^{0.7}}\approx 2.5\times 10^{-3}
>\mathrm{RD}_{P_{5,0.7}}\approx 2.4\times 10^{-3},$$
while
$$\mathrm{RD}_{\ell_{1}^{0.7}}\approx 1.5\times 10^{-3}>
\mathrm{RD}_{P_{1,0.7}}\approx 1.1\times 10^{-3},$$
which clearly indicate that the $P_{a,p}$ penalty functions
$P_{5,0.7}$ and $P_{1,0.7}$ approach
$\ell_0$ more closely than $\ell_5^{0.7}$ and $\ell_{1}^{0.7}$, respectively.
Thus, the relaxation degree RD$_P$ can help us to quantify how closely
a penalty function $P$ approximates $\ell_0$. Moreover,
using \eqref{e2.9}, we find that both
\eqref{RDPap1} and \eqref{RDPap2} with $P_{a,p}$ therein replaced by $\ell_a^p$
still hold, which indicate that, as $a\to\infty$, the $\ell_a^p$ penalty function
approaches $P_p$ and, as $a\to 0^+$ or $p\to 0^+$, the $\ell_a^p$ penalty function
approaches $\ell_0$.

\begin{figure}
      \centering
  \begin{tabular}{cccc}
  \includegraphics[width=3.3cm]{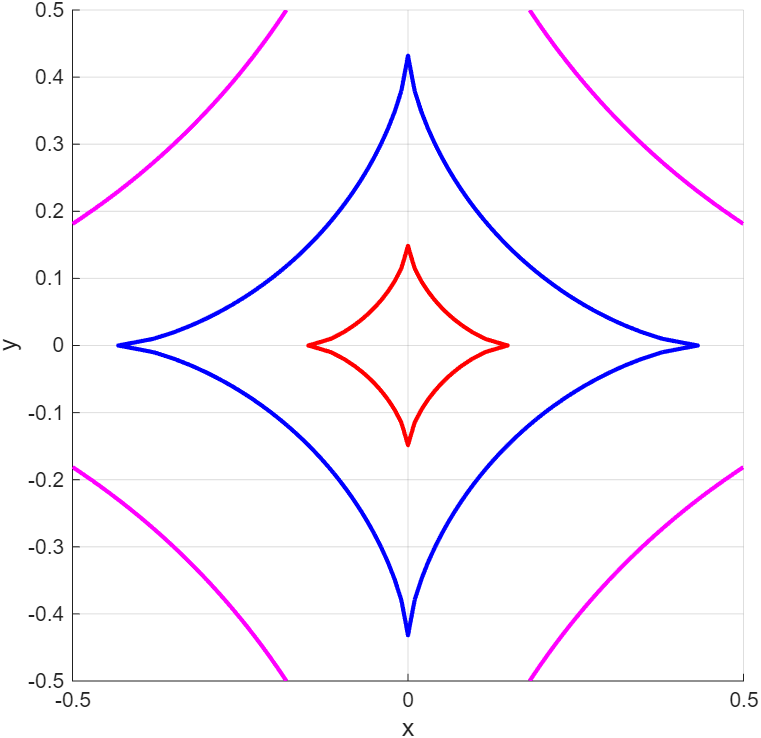}&
  \includegraphics[width=3.3cm]{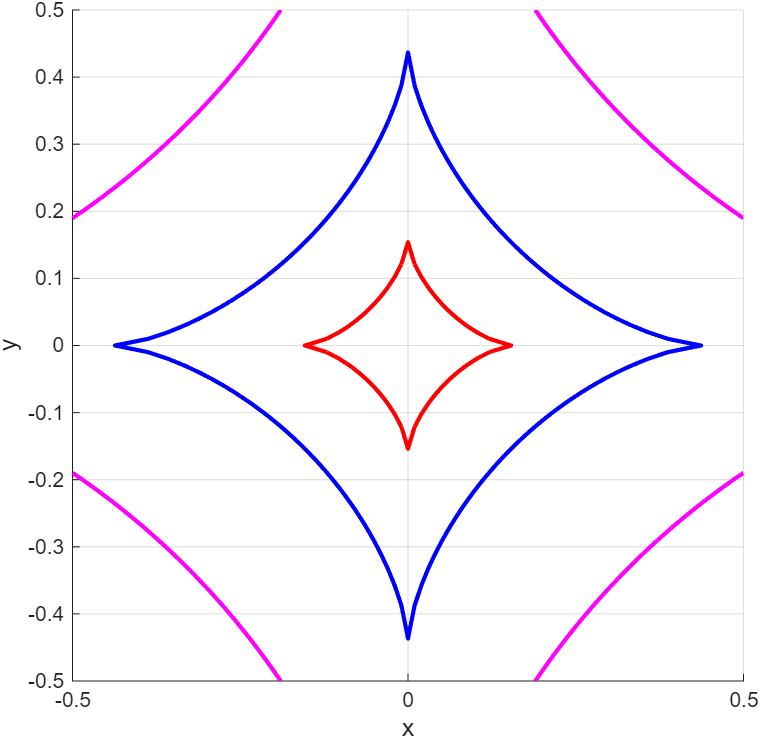}&
  \includegraphics[width=3.3cm]{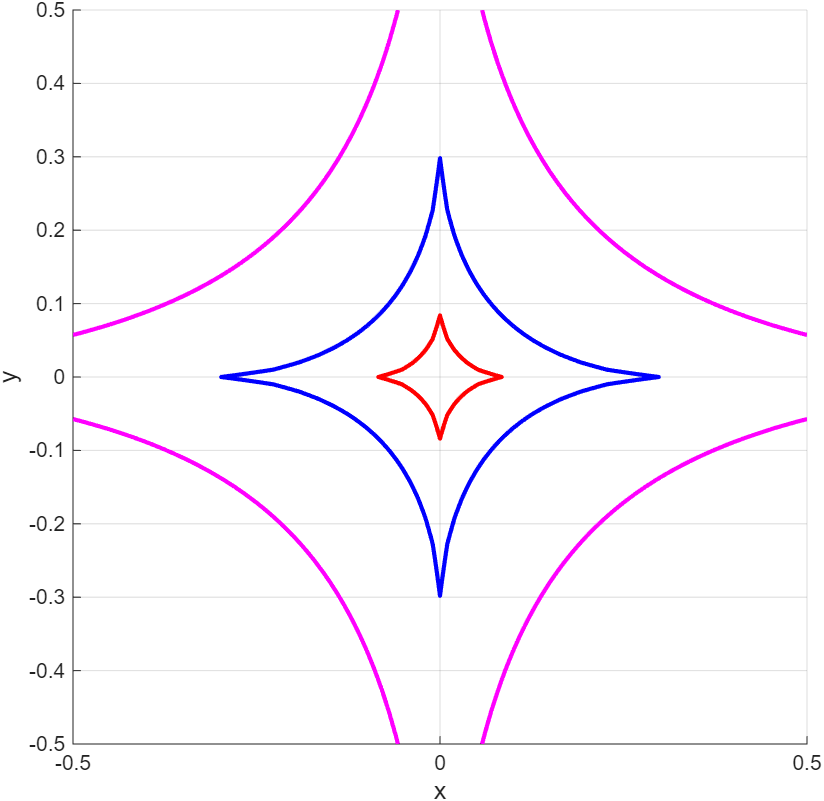}&
  \includegraphics[width=3.3cm]{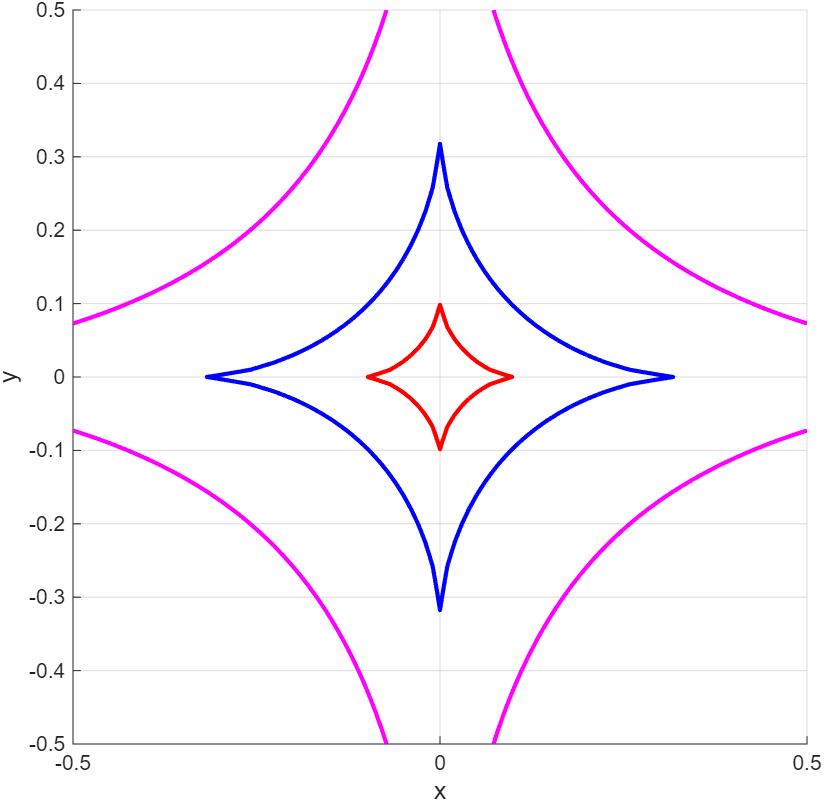}\\
    \quad$P_{5,0.7}$ & \quad $\ell_5^{0.7}$ & \quad $P_{1,0.7}$ & \quad$\ell_1^{0.7}$
    \end{tabular}
  \caption{Two-dimensional  level lines of
  $P_{5,0.7}$, $\ell_5^{0.7}$, $P_{1,0.7}$, and $\ell_1^{0.7}$}\label{fig-TLp}
\end{figure}

\subsection{$P_{a,p}$ Minimization Framework and Sparse Recovery}\label{subsec-main}

In this subsection, we
focus on the sparse recovery for $P_{a,p}$ minimization framework.
We begin with the following well-known  concept of the restricted
isometry property (RIP), which was introduced
by Cand\`{e}s and Tao \cite{CT05} and has become an essential tool in the study of sparse recovery.

\begin{definition}\label{def-RIP}
Let $A\in\mathbb{R}^{M\times N}$ be a matrix.
Then, for each given $s\in\mathbb{N}$ with $1\le s\le N$,
the matrix $A$ is said to have the  \emph{restricted isometry property} (RIP)
of order $s$ if there exists $\delta\in[0,1)$ such that,
for any $s$-sparse vector $x\in\mathbb{R}^{N}$,
\begin{equation}\label{eq-RIP}
(1-\delta)\left\|x\right\|_2^2\le\left\|Ax\right\|_2^2\le(1+\delta)\left\|x\right\|_2^2.
\end{equation}
The smallest $\delta$ satisfying \eqref{eq-RIP} is called the \emph{restricted isometry constant} (RIC)
and denoted by $\delta_s$. When $s$ is not an integer, $\delta_s$ is defined as $\delta_{\lceil s\rceil}$,
where $\lceil s\rceil$ denotes the smallest integer strictly bigger than $s$.
\end{definition}

It is known that many families of matrices, for instance,
the Gaussian or the Bernoulli random matrices,
have the RIP with a high probability (see \cite{BDDW08,CRT06b,CT05}).

Next, we establish the sparse recovery via the RIP condition of sensing matrices.

Due to the absence of the scaling property of the $P_{a,p}$ function,
we use the following normalization procedure.
Let $y\in\mathbb{R}^M$ and $\overline{x}$ be a feasible solution of \eqref{eq-Pl0}, not necessarily the optimal solution.
By Lemma \ref{lem-TLp,prop}(ii),
there always exists a  constant $\beta\in[1,\infty)$ such that $P_{a,p}(\overline{x}/\beta)\le1$.
Moreover, if we let
\begin{equation}\label{eq-C}
\beta\ge a^{-\frac{1}{p}} \|\overline{x}\|_{\infty} \left[(a+1) |{\mathop\mathrm{\,supp\,}}(\overline{x}) |-1\right]^{\frac{1}{p}},
\end{equation}
then
$$
P_{a,p}\left(\frac{\overline{x}}{\beta}\right)
\le\left|{\mathop\mathrm{\,supp\,}}(\overline{x})
\right|\rho_{a,p}\left(\frac{\|\overline{x}\|_{\infty}}{\beta}\right)
=\left|{\mathop\mathrm{\,supp\,}}(\overline{x})
\right|\frac{(a+1)\|\overline{x}\|_{\infty}^p}{a\beta^p +\|\overline{x}\|^p_{\infty}}\le1.
$$
Define
$y_\beta:=\frac{y}{\beta}$ and $\overline{x}_\beta:=\frac{\overline{x}}{\beta}.$
Then $\overline{x}_\beta$ is a solution of the scaled constraint $Ax=y_\beta$.

Similarly, let $\widehat{x}$ be a feasible solution of \eqref{eq-Pl0epsilon},
$\beta$ be chosen such that
\begin{equation}\label{eq-C,epsilon}
\beta\ge a^{-\frac{1}{p}} \|\widehat{x}\|_{\infty} \left[(a+1)|{\mathop\mathrm{\,supp\,}}(\widehat{x})|-1\right]^{\frac{1}{p}},
\end{equation}
and define
$y_\beta:=\frac{y}{\beta},$ $\widehat{x}_\beta:=\frac{\widehat{x}}{\beta},$ and
$\epsilon_\beta:=\frac{\epsilon}{\beta}.$
Then $\widehat{x}_\beta$ is a solution of the scaled constraint $\|Ax-y_\beta\|_2\le\epsilon_\beta$.

We also use the solutions of the following equation
\begin{equation}\label{eq-equation}
p\eta^{\frac{2}{p}}+2\eta-(2-p)\left(\frac{a+1}{a}\gamma\right)^{\frac{p}{2-p}}=0,
\end{equation}
where $\gamma\in[1,\infty)$ is a fixed parameter.
Note that the function $$f(\eta):=p\eta^{\frac{2}{p}}+2\eta-(2-p)\left(\frac{a+1}{a}\gamma\right)^{\frac{p}{2-p}}$$
is monotonously increasing on $[0,\infty)$, $f(0)<0$, and
$$
f\left(\left[1-\frac p2\right]\left[\frac{a+1}{a}\gamma\right]^{\frac{p}{2-p}}\right)>0.
$$
There must exist a \emph{unique positive solution} of the equation \eqref{eq-equation} between $0$ and
$(1-\frac p2)(\frac{a+1}{a}\gamma)^{\frac{p}{2-p}}$, which
we always denote by $\eta_0$ in what follows.
We then let
\begin{equation}\label{e2.14}
\mu_0:= \left(\frac{a+1}{a}\gamma \right)^{-\frac{p}{2-p}}\eta_0,\ \
\delta(p,a,\gamma):=\frac{\mu_0}{2-p-\mu_0},
\ \ \text{and}\ \
\overline{\delta}:=\delta(p,a,1).
\end{equation}
{\color{red}In particular, we use the notation
$\overline{\eta}_0$ and $\overline{\mu}_0$ to respectively denote $\eta$ and $\mu$
in \eqref{e2.14} with $\gamma=1$.}

Now, we let  $x^0$ be the $\ell_0$-minimizer of the constrained problem
 \eqref{eq-Pl0}, $x^0_\beta$ be the  $\ell_0$-minimizer of the normalized constrained $\ell_0$ minimization problem:
\begin{equation}\label{eq-Pl0,n}
\min_{x\in\mathbb{R}^N} \|x\|_0 \quad\text{subject to}\ \ Ax=y_\beta,
\end{equation}
and $x_\beta$ be the minimizer of the normalized constrained $P_{a,p}$ minimization problem:
\begin{equation}\label{eq-PTLp,n}
\min_{x\in\mathbb{R}^N} P_{a,p}(x) \quad\text{subject to}\ \ Ax=y_\beta,
\end{equation}
where $\beta$ satisfies \eqref{eq-C}.
We have the following exact $P_{a,p}$ sparse recovery.
\begin{theorem}\label{thm-exact}
Let both $a\in(0,\infty)$ and $p\in(0,1]$ be fixed.
For any given sensing matrix $A\in\mathbb{R}^{M\times N}$, $y\in\mathbb{R}^M$,
{\color{red}and $s\in\mathbb{N}$},
let $\beta$ satisfy \eqref{eq-C} and
$x^0_\beta\in\mathbb{R}^N$ be the $\ell_0$-minimizer of \eqref{eq-Pl0,n}.
If $A$ satisfies the \emph{RIP} of order $2s$   with
$\delta_{2s}<\overline{\delta}$ and {\color{red} if $x_\beta^0$ is $s$-sparse},
then the minimizer $x_\beta$ of \eqref{eq-PTLp,n} is unique and $x_\beta =x^0_\beta$;
moreover, $\beta x_\beta$ is precisely the unique minimizer of the $\ell_0$ minimization problem \eqref{eq-Pl0}.
\end{theorem}

We also have the following two stable $P_{a,p}$ sparse recovery results.

\begin{theorem}\label{thm-stable}
Let both $a\in(0,\infty)$ and $p\in(0,1]$ be fixed.
For any given sensing matrix $A\in\mathbb{R}^{M\times N}$, $y\in\mathbb{R}^M$,
{\color{red}and $s\in\mathbb{N}$},
let $x^0$ be the $\ell_0$-minimizer of \eqref{eq-Pl0},
 $\beta$ satisfy \eqref{eq-C}, and $x^0_\beta\in\mathbb{R}^N$ be the  $\ell_0$-minimizer of \eqref{eq-Pl0,n}.
If $A$ satisfies the \emph{RIP} of order $2s$  with
$\delta_{2s}< \overline{\delta}$ and {\color{red} if $x_\beta^0$ is $s$-sparse},
then the minimizer $x^\sharp_\beta$ of
the normalized \emph{$P_{a,p}$} minimization problem within  noisy measurements that
\begin{equation*}
\min_{x\in\mathbb{R}^N} P_{a,p}(x) \quad\text{subject to}\ \ \left\|Ax-y_\beta\right\|_2\le \epsilon
\end{equation*}
satisfies
\begin{equation}\label{eq-bound_0}
\left\|x^\sharp_\beta-x^0_\beta\right\|_2\le C_1 \epsilon;
\end{equation}
moreover,
\begin{equation*}
\left\|\beta x^\sharp_\beta-x^0\right\|_2\le C_1 \beta\epsilon,
\end{equation*}
where
{\color{red}\begin{equation*}
C_1:=\overline{\mu}_0\sqrt{1+\left(\frac{a+1}{a}\right)^{\frac{2}{p}}}
\frac{\sqrt{1+\delta_{2s}}\left(1-\overline{\mu}_0\right)
(2-p)+(2-p-\overline{\mu}_0)\sqrt{(1-p)(\overline{\delta}-\delta_{2s})}}
{(2-p-\overline{\mu}_0)^2(\overline{\delta}-\delta_{2s})}.
\end{equation*}}
\end{theorem}

\begin{theorem}\label{thm-stable,ep}
Let both $a\in(0,\infty)$ and $p\in(0,1]$ be fixed.
For any given sensing matrix $A\in\mathbb{R}^{M\times N}$,
$y\in\mathbb{R}^M$, {\color{red}and $s\in\mathbb{N}$},
let $x^{0,\epsilon}$ be an $\ell_0$-minimizer of \eqref{eq-Pl0epsilon},
 $\beta$ satisfy \eqref{eq-C,epsilon}, and $x^{0,\epsilon}_\beta:=\frac{x^{0,\epsilon}}{\beta}$.
If $A$ satisfies the \emph{RIP} of order $2s$ with
$\delta_{2s}< \overline{\delta}$ and {\color{red} if $x^{0,\epsilon}_\beta$ is $s$-sparse},
then the minimizer $x^\sharp_\beta$ of
the normalized \emph{$P_{a,p}$} minimization problem within  noisy measurements that
\begin{equation*}
\min_{x\in\mathbb{R}^N} P_{a,p}(x) \quad\text{subject to}\ \ \left\|Ax-y_\beta\right\|_2\le \frac{\epsilon}{\beta}
\end{equation*}
satisfies
\begin{equation}\label{eq-bound_ep}
\left\|x^\sharp_\beta-x^{0,\epsilon}_\beta\right\|_2\le C_2 \frac{\epsilon}{\beta} ;
\end{equation}
moreover,
\begin{equation*}
\left\|\beta x^\sharp_\beta-x^{0,\epsilon}\right\|_2\le C_2 \epsilon,
\end{equation*}
where
{\color{red}\begin{equation*}
C_2:=2\overline{\mu}_0\sqrt{1+\left(\frac{a+1}{a}\right)^{\frac{2}{p}}}
\frac{\sqrt{1+\delta_{2s}}\left(1-\overline{\mu}_0\right)
(2-p)+(2-p-\overline{\mu}_0)\sqrt{(1-p)(\overline{\delta}-\delta_{2s})}}
{(2-p-\overline{\mu}_0)^2(\overline{\delta}-\delta_{2s})}.
\end{equation*}}
\end{theorem}

\begin{remark}
\begin{enumerate}
\item[\rm(i)]
If $\delta_{2s}<1$ and  the true signal is $s$-sparse for some $s\in\mathbb{N}$,
then the $\ell_0$-minimizer is exactly the true signal.
Indeed, if we let $x^*$ be the true signal and $x^0$ be an
$\ell_0$-minimizer, then we have
$
\|x^0\|_0\le\|x^*\|_0\le s,
$
which implies that $x^0-x^*$ is $2s$-sparse, and hence, by the RIP,
$$
0=\left\|A\left(x^0-x^*\right)\right\|_2^2\ge (1-\delta_{2s})\left\| x^0-x^* \right\|_2^2.
$$
Thus, $x^0=x^*$.

\item[\rm(ii)]
We  compare the RIP condition here with some known ones.
When $p=1$, the $P_{a,p}$ function reduces to the TL1 function.
One can solve  equation \eqref{eq-equation} to obtain {\color{red}$\overline{\eta}_0=\sqrt{1+\frac{a+1}{a} }-1$}
and hence the RIP condition  becomes $\delta_{2s}<\frac{1}{\sqrt{1+\frac{a+1}{a} }}$.
Recall that the  RIP condition in \cite[Theorems 2.2 and 2.3]{ZX18} is that
$$
\delta_R +\left(\frac{a}{a+1}\right)^2\frac{R}{s}\delta_{R+s}<\left(\frac{a}{a+1}\right)^2\frac{R}{s}-1,
$$
where $R>s$. To guarantee the validity of this condition,
we must have $R>(\frac{a+1}a)^2s$ and hence, if $a\le 1$, then $R>4s$.
In this sense, the RIP conditions in
Theorems \ref{thm-exact}, \ref{thm-stable}, and \ref{thm-stable,ep}
are weaker than the corresponding ones in \cite[Theorems 2.2 and 2.3]{ZX18}.
On the other hand, as $a\to\infty$, {\color{red}$\overline{\eta}_0$} tends to
the unique positive solution of the equation
\begin{equation*}
\frac{p}{2}\eta^{\frac{2}{p}}+\eta-1+\frac{p}{2}=0,
\end{equation*}
and hence the upper bound $\overline \delta$ in \eqref{e2.14} of the RIP condition tends
to the sharp RIP upper bound for the $\ell_p$ minimization recovery
given by Zhang and Li \cite[Theorems 1.2, 1.3, and 1.4]{ZL19} and, especially when $p=1$,
to the sharp bound $\frac{\sqrt{2}}{2}$ for the classical $\ell_1$ minimization
established by Cai and Zhang in \cite{CZ14}. Thus, in this sense, the upper bound
$\overline \delta$ of the RIP condition in Theorems \ref{thm-exact},
\ref{thm-stable}, and \ref{thm-stable,ep} is sharp.

\item[\rm(iii)] As $a\to\infty$, the constants $C_1$
in \eqref{eq-bound_0} and $C_2$ in \eqref{eq-bound_ep}
can tend respectively to the one in the case $\rho=0$
and the one in the case $\rho=\epsilon$ therein of the stable recovery
via $\ell_p$ minimization with $p\in (0,1]$ given in \cite[Theorem 1.3]{ZL19},
which when $p=1$ also tend to the ones in \cite[Theorem 2.1]{CZ14}.
\end{enumerate}
\end{remark}

\subsection{Proofs of Theorems \ref{thm-exact}, \ref{thm-stable},
and \ref{thm-stable,ep}}\label{subsec-proof}

This subsection is devoted to proving the main results in Subsection \ref{subsec-main}.

The proof of Theorem \ref{thm-exact} depends on the following proposition.

\begin{proposition}\label{prop-u,exact}
Let both $a\in(0,\infty)$ and $p\in(0,1]$ be fixed, $\gamma\ge 1$,
and $A\in\mathbb{R}^{M\times N}$ be a matrix
satisfying the \emph{RIP} of order $2s$ with
$\delta_{2s}<\delta(p,a,\gamma)$
for some $s\in\mathbb{N}$. For any $u\in \mathrm{Ker}\, A$,
if there exists some index set $S$ with $|S|\le s$ such that $\|u_{S^\complement}\|_{\infty}\le 1$
and $u$ satisfies the $P_{a,p}$-cone constraint that
$$
P_{a,p}\left(u_{S^\complement}\right)\le \gamma P_{a,p}\left(u_{S}\right),
$$
then $u=\mathbf{0}$.
\end{proposition}

To prove Proposition \ref{prop-u,exact}, we need the following
two lemmas. The former is the key sparse
convex-combination technique given in \cite[Lemma 2.2]{ZL19} and the latter is a part of \cite[Lemma 5.3]{CZ13}.

\begin{lemma}\label{lem-sparse repre}
Let $p\in(0,1]$, $\alpha\in(0,\infty)$, and $s\in\mathbb{N}$. For any given $u\in\mathbb{R}^N$ with $|{\mathop\mathrm{\,supp\,}}(u)|=n\ge s$,
$\|u\|_p^p\le s\alpha^p$, and $\|u\|_\infty\le \alpha$,
$u$ can be represented as a convex combination of finite $s$-sparse vectors,
$$
u=\sum_{i=1}^L \lambda_i v_i\  \ for\ some \ L \in\mathbb{N},
$$
where $\sum_{i=1}^L  \lambda_i =1$ with $\lambda_i\in(0,1]$ and where $v_i$ is $s$-sparse with
${\mathop\mathrm{\,supp\,}}(v_i)\subset {\mathop\mathrm{\,supp\,}}(u)$.
Moreover,
\begin{equation}\label{eq-decom,est}
\sum_{i=1}^L  \lambda_i \left\|v_i\right\|_2^2\le
\min\left\{\frac{n}{s}\left\|u\right\|_2^2,\alpha^p\left\|u\right\|_{2-p}^{2-p}\right\}.
\end{equation}
\end{lemma}

\begin{remark}
According to the proof of \cite[Lemma 2.2]{ZL19},
\eqref{eq-decom,est} still holds if $\alpha$ therein is replaced by $\|u\|_\infty$.
\end{remark}

\begin{lemma}\label{lem-Tc<T}
Let $k,l\in\mathbb{N}$ satisfy $k\le l$ and $a_1\ge a_2\ge\ldots\ge a_l\ge0$.
If $\sum_{j=1}^k a_j\ge \sum_{j=k+1}^l a_j$, then, for any $\alpha\ge 1$,
$\sum_{j=1}^k a_j^\alpha\ge \sum_{j=k+1}^l a_j^\alpha.$
\end{lemma}

\begin{proof}[Proof of Proposition \ref{prop-u,exact}]
Let $u\in \mathrm{Ker}\, A$ and $T$ be the index set of the $s$ largest
components of $u$ in magnitude.
Then, by the assumptions, we find that
$\|u_{T^\complement}\|_{\infty}\le1$ and $u$ satisfies
the $P_{a,p}$-cone constraint with respect to $T$.

To prove the present proposition by contradiction, we assume that $u\neq \mathbf{0}$.
By $\|u_{T^\complement}\|_{\infty}\le1$, Lemma \ref{lem-TLp,prop}(iii),
$\gamma\ge 1$, and the $P_{a,p}$-cone constraint,
we find that
\begin{equation}\label{eq-DScfz}
\left\|u_{T^\complement}\right\|^p_\infty \le\rho_{a,p}
\left(\|u_{T^\complement}\|_\infty\right)\le\frac{P_{a,p}(u_{T})}{s}
\le\frac{\gamma P_{a,p}\left(u_{T}\right)}{s}
\end{equation}
and
\begin{equation}\label{eq-DSc1}
\left\|u_{T^\complement}\right\|_p^p\le P_{a,p}(u_{T^\complement})
\le\gamma P_{a,p}\left(u_{T}\right).
\end{equation}
Combining \eqref{eq-DScfz} and \eqref{eq-DSc1} and applying Lemma \ref{lem-sparse repre},
we obtain the following finite convex decomposition
$u_{T^\complement}=\sum_{i=1}^L \lambda_i v_i$
for some $L \in\mathbb{N}$, where $\sum_{i=1}^L  \lambda_i =1$ with $\lambda_i\in(0,1]$,
$v_i$ is $s$-sparse, and, moreover,
\begin{equation}\label{eq-moreover}
\sum_{i=1}^L  \lambda_i \left\|v_i\right\|_2^2
\le \left\|u_{T^\complement}\right\|_\infty^p\left\|u_{T^\complement}\right\|_{2-p}^{2-p}.
\end{equation}
Note that
$$
\left\|u_{T}\right\|_{p}^p\le s^{1-\frac{p}{2}}\left\|u_{T}\right\|_{2}^p\ \ \text{and}\ \
\left\|u_{T^\complement}\right\|^p_\infty \le\frac{\|u_T\|^p_p}{s}.
$$
Applying these, H\"{o}lder's  inequality, \eqref{eq-DSc1}, and Lemma \ref{lem-TLp,prop}(iii) to \eqref{eq-moreover},
we conclude that
\begin{align}\label{eq-est1}
\sum_{i=1}^L  \lambda_i \left\|v_i\right\|_2^2
&\le \left\|u_{T^\complement}\right\|_\infty^p\left(\|u_{T^\complement}\|_{2}^{2}\right)^{\frac{2-2p}{2-p}}
\left(\|u_{T^\complement}\|_{p}^{p}\right)^{\frac{p}{2-p}}\nonumber\\
&\le\left(\frac{a+1}{a}\gamma\right)^{\frac{p}{2-p}}
\frac{(\|u_{T}\|_{p}^p)^{\frac{2}{2-p}}(\|u_{T^\complement}\|_{2}^{2})^{\frac{2-2p}{2-p}}}{s}\nonumber\\
&\le\left(\frac{a+1}{a}\gamma\right)^{\frac{p}{2-p}}
\left( \left\|u_{T}\right\|_{2}^2\right)^{\frac{p}{2-p}}
\left(\|u_{T^\complement}\|_{2}^{2}\right)^{\frac{2-2p}{2-p}}=:\Pi .
\end{align}

Now, for any $i\in\{1,\ldots,L \}$, we let $w_i:=u_T+\mu v_i$, where
$\mu\in\mathbb{R}$ is a constant which will be determined later.
By $\sum_{j=1}^L  \lambda_j =1$, we have
\begin{align}\label{eq-equality}
&\sum_{i=1}^L  \lambda_i\left\|A\,\left(\sum_{j=1}^L  \lambda_j w_j -\frac p2 w_i\right)\,\right\|_2^2
+\frac{1-p}{2}\sum_{i,j=1}^L  \lambda_i \lambda_j  \left\|A(w_i-w_j)\right\|_2^2 \nonumber\\
&\quad=(1-p)\left\|A\left(\sum_{j=1}^L  \lambda_jw_j\right)\right\|_2^2+\frac{p^2}{4}\sum_{i=1}^L  \lambda_i \left\|Aw_i\right\|_2^2\nonumber\\
&\quad\quad+(1-p)\left[\sum_{i=1}^L \lambda_i \left\|Aw_i\right\|_2^2
-\left\|A\left(\sum_{i=1}^L  \lambda_i w_i\right)\right\|_2^2\right]\nonumber\\
&\quad=\left(1-\frac{p}{2}\right)^2\sum_{i=1}^L \lambda_i \left\|Aw_i\right\|_2^2.
\end{align}
For simplicity, we denote by $\mathrm{LHS}$ the left-hand side
and by $\mathrm{RHS}$ the right-hand side of \eqref{eq-equality}.
Since, by $\sum_{j=1}^L  \lambda_j =1$, we have, for each $i\in\{1,\ldots,L \}$,
$$
\sum_{j=1}^L  \lambda_j w_j -\frac p2 w_i=\mu u +\left(1-\frac p2 -\mu\right) u_T -\frac{p\mu}{2}v_i,
$$
then, from $u-u_T=\sum_{i=1}^L \lambda_i v_i$, $w_i-w_j=\mu(v_i-v_j)$, the Cauchy--Schwarz inequality,
the RIP, and $Au=\mathbf{0}$, it follows that
\begin{align*}
\mathrm{LHS}&=\sum_{i=1}^L  \lambda_i\left\|A\,\Bigg(\left(1-\frac p2 -\mu\right) u_T -\frac{p\mu}{2}v_i\Bigg)\,\right\|_2^2
+\frac{1-p}{2}\mu^2\sum_{i,j=1}^L  \lambda_i \lambda_j  \left\|A(v_i-v_j)\right\|_2^2\\
&\quad+\mu\langle Au, \mu(1-p) Au+(2-p)(1-\mu)Au_T\rangle \\
&\le\left(1-\frac p2 -\mu\right) ^2\left\|Au_T\right\|_2^2
+\frac{p^2\mu^2}{4}\sum_{i=1}^L  \lambda_i\left\|Av_i\right\|_2^2
+\frac{1-p}{2}\mu^2\sum_{i,j=1}^L  \lambda_i \lambda_j  \left\|A(v_i-v_j)\right\|_2^2\\
&\quad+\mu^2(1-p) \left\|Au\right\|_2^2
+\mu(1-\mu)(2-p) \left\|Au\right\|_2 \left\|Au_T\right\|_2\\
&\le(1+\delta_{2s})\left[\left(1-\frac p2 -\mu\right) ^2\left\|u_T\right\|_2^2
+\frac{p^2\mu^2}{4}\sum_{i=1}^L  \lambda_i\left\|v_i\right\|_2^2
+\frac{1-p}{2}\mu^2\sum_{i,j=1}^L \lambda_i \lambda_j  \left\|v_i-v_j\right\|_2^2\right]\\
&=(1+\delta_{2s})\left[\left(1-\frac p2 -\mu\right) ^2\left\|u_T\right\|_2^2
+\frac{p^2\mu^2}{4}\sum_{i=1}^L  \lambda_i\left\|v_i\right\|_2^2\right]\\
&\quad+(1+\delta_{2s})(1-p)\mu^2\left(\sum_{i=1}^L   \lambda_i  \left\|v_i\right\|_2^2
-\left\|\sum_{i=1}^L  \lambda_i v_i\right\|_2^2\right).
\end{align*}
On the other hand, by $|{\mathop\mathrm{\,supp\,}}(w_i)|\le 2s$ and $\sum_{j=1}^L  \lambda_j =1$, we have
\begin{align*}
\mathrm{RHS}
\ge (1-\delta_{2s}) \left(1-\frac{p}{2}\right)^2\sum_{i=1}^L  \lambda_i\left\| w_i\right\|_2^2
= (1-\delta_{2s}) \left(1-\frac{p}{2}\right)^2\Bigg(\left\| u_T \right\|_2^2
+\mu^2\sum_{i=1}^L  \lambda_i\left\|v_i\right\|_2^2\Bigg).
\end{align*}
Combining these two inequalities, we obtain
\begin{align*}
&\left[ (1-\delta_{2s}) \left(1-\frac{p}{2}\right)^2-(1+\delta_{2s})\left(1-\frac p2 -\mu\right)^2\right]\left\|u_T\right\|_2^2\\
&\quad\le\left[(1+\delta_{2s})\left(1-p+\frac{p^2}{4}\right)- (1-\delta_{2s}) \left(1-\frac{p}{2}\right)^2\right]\mu^2\sum_{i=1}^L  \lambda_i\left\|v_i\right\|_2^2\\
&\quad\quad-(1+\delta_{2s})(1-p)\mu^2\left\|u_{T^\complement}\right\|_2^2\\
&\quad=  2\left(1-\frac p2\right)^2 \delta_{2s}\mu^2\sum_{i=1}^L  \lambda_i \left\|v_i\right\|_2^2
-(1+\delta_{2s})(1-p)\mu^2\left\|u_{T^\complement}\right\|_2^2
\end{align*}
and hence, by \eqref{eq-est1},
\begin{align*}
&\left[(1+\delta_{2s})\left(1-\frac p2 -\mu\right)^2 -(1-\delta_{2s}) \left(1-\frac{p}{2}\right)^2\right]\left\|u_T\right\|_2^2\\
&\qquad+2\left(1-\frac p2\right)^2 \delta_{2s}\mu^2\Pi
-(1+\delta_{2s})(1-p)\mu^2\left\|u_{T^\complement}\right\|_2^2\\
&\quad\ge 0.\nonumber
\end{align*}
This further yields
\begin{align}\label{eq-delta>}
\delta_{2s}&\ge\max\left\{0,\frac{[(2-p)\mu-\mu^2]\|u_T\|_2^2+(1-p)\mu^2\|u_{T^\complement}\|_2^2}
{[(1-\frac p2 -\mu)^2+(1-\frac p2)^2]\|u_T\|_2^2+2(1-\frac p2)^2
\mu^2\Pi-(1-p)\mu^2\|u_{T^\complement}\|_2^2}\right\}\nonumber\\
&\ge\left\{2\left(1-\frac p2\right)^2\frac{\|u_T\|_2^2+\mu^2\Pi}
{[(2-p)\mu-\mu^2]\|u_T\|_2^2+(1-p)\mu^2\|u_{T^\complement}\|_2^2}-1\right\}^{-1}.
\end{align}
Note that $\mu$ can be chosen arbitrarily in $\mathbb{R}$.
By calculating the derivative of the function (with respect to $\mu$) on the right-hand side of \eqref{eq-delta>},
the minimum of this function is achieved when
$$
\mu=\frac{(1-p)\|u_{T^\complement}\|_2^2-\|u_T\|_2^2+
\sqrt{[\|u_T\|_2^2-(1-p)\|u_{T^\complement}\|_2^2]^2+(2-p)^2\Pi\|u_T\|_2^2}}{(2-p)\Pi}.
$$
Thus, we obtain
\begin{align*}
\delta_{2s}&\ge\left\{\sqrt{\left[(1-p)\frac{\|u_{T^\complement}\|_2^2}
{\|u_T\|_2^2}-1\right]^2+(2-p)^2\frac{\Pi}{\|u_T\|_2^2}}
-(1-p)\frac{\|u_{T^\complement}\|_2^2}{\|u_T\|_2^2}\right\}^{-1}\\
&=\left\{\sqrt{\left[(1-p)t-1\right]^2+(2-p)^2\left(\frac{a+1}{a}
\gamma\right)^{\frac{p}{2-p}}t^{\frac{2-2p}{2-p}}}
-(1-p)t\right\}^{-1},
\end{align*}
where $t:=\|u_{T^\complement}\|_2^2/\|u_T\|_2^2$.
Moreover, by calculating the derivative of the function
$$
f(t):=\sqrt{\left[(1-p)t-1\right]^2+(2-p)^2\left(\frac{a+1}{a}\gamma\right)^{\frac{p}{2-p}}
t^{\frac{2-2p}{2-p}}}-(1-p)t,\ \ \forall\,t\in(0,\infty),
$$
we find that, when $t=\eta_0^{(2-p)/p}$,
where $\eta_0$ is the only positive solution of   equation \eqref{eq-equation},
$f$ can arrive at its maximum; that is,
$$
f\left(\eta_0^{\frac{2-p}{p}}\right)=\eta_0^{\frac{2}{p}-1}+1-(1-p)\eta_0^{\frac{2}{p}-1}=
\frac{p\eta_0^{\frac2p}+\eta_0}{\eta_0}=\frac{(2-p)(\frac{a+1}{a}\gamma)^{\frac{p}{2-p}}-\eta_0}{\eta_0}.
$$
Consequently, we conclude that
$$
\delta_{2s}\ge \left[f\left(\eta_0^{\frac{2-p}{p}}\right)\right]^{-1}
=\frac{\eta_0}{(2-p)(\frac{a+1}{a}\gamma)^{\frac{p}{2-p}}-\eta_0}=:\delta(p,a,\gamma)
$$
which contradicts the assumption $\delta_{2s}<\delta(p,a,\gamma)$.
Thus, $u=\mathbf{0}$.
This completes the proof of Proposition \ref{prop-u,exact}.
\end{proof}

\begin{proof}[Proof of Theorem \ref{thm-exact}]
For simplicity, we write $x^0_\beta$ as $x^0$ and $x_\beta$ as $x$.
Let $u:=x-x^0$ and $T:={\mathop\mathrm{\,supp\,}}(x^0)$.
With the assumption on $\beta$,
we have $P_{a,p}(x)\le P_{a,p}(\overline{x}_\beta)\le 1$.
Then, by Lemma \ref{lem-TLp,prop}(iii) and $u_{T^\complement}=x_{T^\complement}$, we further obtain
$\|u_{T^\complement}\|_{\infty}\le 1$.
Since $x$ is the minimizer  of \eqref{eq-Pl0,n}, from the separability
of $P_{a,p}$ and from Lemma \ref{lem-TLp,prop}(iv),
we deduce that
\begin{align*}
P_{a,p}(x^0)&\ge P_{a,p}(x)=P_{a,p}(x^0+u)=P_{a,p}(x^0 +u_T)+P_{a,p}(u_{T^\complement})\\
&\ge P_{a,p}(x^0)-P_{a,p}(u_T)+P_{a,p}(u_{T^\complement}),\nonumber
\end{align*}
which   implies that $u$ satisfies the $P_{a,p}$-cone constraint that
$P_{a,p}(u_{T^\complement})\le P_{a,p}(u_T).$
Obviously, $$Au=Ax-Ax^0=\mathbf{0}.$$
Thus, applying these and Proposition \ref{prop-u,exact} with $\gamma=1$,
we conclude that
$u=\mathbf{0}$, which means $x=x^0$.
This completes the proof of Theorem \ref{thm-exact}.
\end{proof}

To prove Theorems \ref{thm-stable} and \ref{thm-stable,ep},
we need the following more general variant of Proposition \ref{prop-u,exact}.

\begin{proposition}\label{prop-u}
Let both $a\in(0,\infty)$ and $p\in(0,1]$ be fixed, $\gamma\ge 1$,
and $A\in\mathbb{R}^{M\times N}$ be a matrix
satisfying the \emph{RIP} of order $2s$ with
$\delta_{2s}<\delta(p,a,\gamma)$
for some $s\in\mathbb{N}$.
For any $u\in\mathbb{R}^N$,
if $u$ satisfies the tube constraint $\|Au\|_2\le \epsilon$
for some $\epsilon\in[0,\infty)$ and the $P_{a,p}$-cone constraint
$P_{a,p}(u_{S^\complement})\le \gamma P_{a,p}(u_{S})$
for some index set $S$ with $|S|\le s$ and $\|u_{S^\complement}\|_{\infty}\le 1$,
then there exists a positive constant $C_0$ such that
$\left\|u\right\|_2\le C_0\epsilon,$
where
\begin{align}\label{eq-def-C0}
C_0&:=\frac{\sqrt{1+\delta_{2s}}\left(1-\mu_0\right)
(2-p)+(2-p-\mu_0)\sqrt{(1-p)[\delta(p,a,\gamma) -\delta_{2s}]}}
{(2-p-\mu_0)^2[\delta(p,a,\gamma) -\delta_{2s}]}\nonumber\\
&\qquad\times\mu_0\sqrt{1+\left(\frac{a+1}{a}\gamma\right)^{\frac{2}{p}}}.
\end{align}
\end{proposition}

\begin{proof}
The proof of the present proposition is similar
to that of Proposition \ref{prop-u,exact} and we only indicate
their differences here by using the same notation as in the proof
of Proposition \ref{prop-u,exact}.
Let $T$ be the index set of the $s$ largest components of $u$ in magnitude.
In the present case, the left-hand side of \eqref{eq-equality}
is estimated as follows
\begin{align*}
\mathrm{LHS}
&\le(1+\delta_{2s})\left[\left(1-\frac p2 -\mu\right) ^2\left\|u_T\right\|_2^2
+\frac{p^2\mu^2}{4}\sum_{i=1}^L  \lambda_i\left\|v_i\right\|_2^2
+\frac{1-p}{2}\mu^2\sum_{i,j=1}^L  \lambda_i \lambda_j  \left\|v_i-v_j\right\|_2^2\right]\\
&\quad+\mu^2(1-p) \epsilon^2
+\mu(1-\mu)(2-p) \epsilon\sqrt{1+\delta_{2s}}  \left\|u_T\right\|_2\\
&=(1+\delta_{2s})\left[\left(1-\frac p2 -\mu\right) ^2\left\|u_T\right\|_2^2
+\frac{p^2\mu^2}{4}\sum_{i=1}^L  \lambda_i\left\|v_i\right\|_2^2\right]\\
&\quad+(1+\delta_{2s})(1-p)\mu^2\left(\sum_{i=1}^L   \lambda_i  \left\|v_i\right\|_2^2
-\left\|\sum_{i=1}^L \lambda_i v_i\right\|_2^2\right)\\
&\quad+\mu^2(1-p) \epsilon^2
+\mu(1-\mu)(2-p) \epsilon\sqrt{1+\delta_{2s}}  \left\|u_T\right\|_2.
\end{align*}
Observe that we still have the same estimate for the right-hand side of \eqref{eq-equality}
as in the proof of  Proposition \ref{prop-u,exact}.
Therefore, we have
\begin{align*}
&\left[ (1-\delta_{2s}) \left(1-\frac{p}{2}\right)^2-(1+\delta_{2s})\left(1-\frac p2 -\mu\right)^2\right]\left\|u_T\right\|_2^2
-\sqrt{1+\delta_{2s}} \,\mu(1-\mu)(2-p)\epsilon\left\|u_T\right\|_2\\
&\quad\le 2\left(1-\frac p2\right)^2 \delta_{2s}\mu^2\sum_{i=1}^L  \lambda_i \left\|v_i\right\|_2^2
-(1+\delta_{2s})(1-p)\mu^2\left\|u_{T^\complement}\right\|_2^2
+\mu^2(1-p) \epsilon^2.
\end{align*}
Using \eqref{eq-est1}, we conclude that
\begin{align*}
&\left[(1+\delta_{2s})\left(1-\frac p2 -\mu\right)^2 -(1-\delta_{2s}) \left(1-\frac{p}{2}\right)^2\right]\left\|u_T\right\|_2^2
+\sqrt{1+\delta_{2s}} \,\mu(1-\mu)(2-p)\epsilon\left\|u_T\right\|_2\nonumber\\
&\qquad+2\left(1-\frac p2\right)^2 \delta_{2s}\mu^2\Pi
-(1+\delta_{2s})(1-p)\mu^2\left\|u_{T^\complement}\right\|_2^2+\mu^2(1-p) \epsilon^2\\
&\quad\ge 0,
\end{align*}
where the left-hand side (which is a function of $\|u_{T^\complement}\|_{2}^{2}$) can arrive at its maximum when
$$\left\|u_{T^\complement}\right\|_{2}^{2}=
\frac{a+1}{a}\gamma \left[\frac{(2-p)\delta_{2s}}{1+\delta_{2s}}\right]^{\frac{2-p}{p}}
\left\|u_{T}\right\|_{2}^2.$$
Thus, we obtain
\begin{align*}
0&\le\left[(1+\delta_{2s})\left(1-\frac p2 -\mu\right)^2 -(1-\delta_{2s})
\left(1-\frac{p}{2}\right)^2\right]\left\|u_T\right\|_2^2\\
&\quad+\frac p2(1+\delta_{2s})\left[\frac{(2-p)\delta_{2s}}{1+\delta_{2s}}\right]^{\frac{2-p}{p}}\mu^2
\frac{a+1}{a}\gamma \left\|u_{T}\right\|_{2}^2\\
&\quad+\sqrt{1+\delta_{2s}} \,\mu(1-\mu)(2-p)\epsilon\left\|u_T\right\|_2+\mu^2(1-p) \epsilon^2.
\end{align*}
Since $\mu$ is arbitrary, in the above inequality, if we let
$$
\mu:=\mu_0=\left(\frac{a+1}{a}\gamma\right)^{-\frac{p}{2-p}}\eta_0,
$$
then, from the assumption that $\eta_0$ satisfies \eqref{eq-equation}, $\mu_0<1-\frac{p}{2}$,
and the assumption $\delta_{2s}<\delta(p,a,\gamma)$ which implies
$\frac{(2-p)\delta_{2s}}{1+\delta_{2s}}<\mu_0$,  we further infer that
\begin{align}\label{e2.27}
0&\le(2-p-\mu_0)\left[(2-p-\mu_0)\delta_{2s}-
\mu_0\right]\left\|u_T\right\|_2^2\notag\\
&\quad+\sqrt{1+\delta_{2s}} \mu_0\left(1-\mu_0\right)
(2-p)\epsilon\left\|u_T\right\|_2+\mu_0^2 (1-p) \epsilon^2.
\end{align}
Using $\mu_0<1-\frac{p}{2}$ and the assumption $\delta_{2s}<\delta(p,a,\gamma)$
again, we find that the coefficient  of $\|u_T\|_2^2$ is negative.
By this, together with the quadratic formula, and by the elementary inequality $\sqrt{a+b}\le\sqrt{a}+\sqrt{b}$
for any $a,b\in[0,\infty)$, we conclude that \eqref{e2.27} implies that
$$
\left\|u_{T}\right\|_2\le\frac{\sqrt{1+\delta_{2s}}\left(1-\mu_0\right)
(2-p)+(2-p-\mu_0)\sqrt{(1-p)[\delta(p,a,\gamma)-\delta_{2s}]}}
{(2-p-\mu_0)^2(\delta(p,a,\gamma)-\delta_{2s})}\mu_0\epsilon.
$$
Note that, by \eqref{eq-DSc1}, one has
\begin{equation*}
\left\|u_{T^\complement}\right\|_p^p \le\gamma P_{a,p}\left(u_{T}\right)
\le\frac{a+1}{a}\gamma\left\|u_{T}\right\|_p^p.
\end{equation*}
It then follows from Lemma \ref{lem-Tc<T} with $\alpha=2/p$ that
$$
\left\|u_{T^\complement}\right\|_2
\le\left(\frac{a+1}{a}\gamma\right)^{\frac{1}{p}}\left\|u_{T}\right\|_2.
$$
Thus, we obtain
\begin{align*}
\left\|u\right\|_2&\le \sqrt{\left\|u_{T}\right\|_2^2+\left\|u_{T^\complement}\right\|_2^2}
\le\sqrt{1+\left(\frac{a+1}{a}\gamma\right)^{\frac{2}{p}}}\left\|u_{T}\right\|_2
\le C_0\epsilon,
\end{align*}
where $C_0$ is as in \eqref{eq-def-C0}.
This completes the proof of Proposition \ref{prop-u}.
\end{proof}

Now, we use Proposition \ref{prop-u} to prove Theorems \ref{thm-stable}
and \ref{thm-stable,ep}.

\begin{proof}[Proof of Theorem \ref{thm-stable}]
We write $x^0_\beta$ as $x^0$ and $x^\sharp_\beta$ as $x^\sharp$.
Let $u:=x^\sharp-x^0$ and $T:={\mathop\mathrm{\,supp\,}}(x^0)$.
Similar to the proof of Theorem \ref{thm-exact}, we still have
$\|u_{T^\complement}\|_{\infty}\le 1$ and
$
P_{a,p}(u_{T^\complement})\le P_{a,p}(u_T).
$
By the triangle inequality of $\|\cdot\|_2$, $u$ also satisfies the tube constraint that
\begin{align*}
\left\|Au\right\|_2\le \left\|Ax^\sharp-y_\beta\right\|_2+\left\|y_\beta-Ax^0\right\|_2
\le\epsilon+0=\epsilon .
\end{align*}
These, combined with Proposition \ref{prop-u} in the case $\gamma=1$, further imply that,
when $\delta_{2s}<\overline{\delta}$,
\begin{align*}
\|x^\sharp-x^0\|_2&\le \sqrt{1+\left(\frac{a+1}{a}\right)^{\frac{2}{p}}}\\
&\quad\times\frac{\sqrt{1+\delta_{2s}}\left(1-\overline{\mu}_0\right)
(2-p)+(2-p-\overline{\mu}_0)\sqrt{(1-p)(\overline{\delta}-\delta_{2s})}}
{(2-p-\overline{\mu}_0)^2[\overline{\delta}-\delta_{2s}]}\overline{\mu}_0\epsilon,
\end{align*}
which completes the proof of Theorem \ref{thm-stable}.
\end{proof}

\begin{proof}[Proof of Theorem \ref{thm-stable,ep}]
The proof of the present theorem is similar to that of Theorem \ref{thm-stable}
except that $u$ is defined as $x^\sharp_\beta-x^{0,\epsilon}_\beta$
and the  tube constraint  should be replaced by
\begin{align*}
\left\|Au\right\|_2\le \left\|Ax^\sharp-y_\beta\right\|_2+\left\|y_\beta-Ax^0\right\|_2\le \frac{2\epsilon}{\beta} .
\end{align*}
We omit the details. This completes the proof of Theorem \ref{thm-stable,ep}.
\end{proof}

\section{Algorithms}\label{sec-algorithm}

 The idea of the iteratively re-weighted least squares (IRLS) algorithm appears
 in the approximation practice by Lawson   for the first time in 1961.
 In the investigation of signal processing,
 IRLS is a vital technique for sparse reconstruction via $\ell_p$ minimization;
 we refer to, for example, \cite{DDFG10,FR13,LXY13}.

In this section, we use a modified IRLS method to solve
the minimization problem with $P_{a,p}$ penalty (IRLSTLp) and then
present the overall IRLSTLp algorithm to solve the
unconstrained $P_{a,p}$ minimization problem in Subsection \ref{subsec-IRLS}.
Finally, we establish some convergence results on outer and inner
iterations of the IRLSTLp algorithm,
respectively, in Subsections \ref{subsec-convergIRLS}
and \ref{subsec-convergDCA}.

\subsection{IRLSTLp Algorithm for Unconstrained $P_{a,p}$ Minimization}\label{subsec-IRLS}

We begin with a modified IRLS method.
Let $a\in(0,\infty)$ and $p\in(0,1]$.
For a given $\epsilon>0$ and a weight vector $\omega:=(\omega_1,\ldots,\omega_N)\in\mathbb{R}^N$
with each $\omega_i>0$,
we define a \emph{functional $\mathcal{J}_{a,p}$} by setting, for any $x:=(x_1,\ldots,x_N)\in\mathbb{R}^N$,
\begin{equation}\label{e3.1}
\mathcal{J}_{a,p}(x,\omega,\epsilon):=
\frac{p(a+1)}{2}\left[\sum_{i=1}^N \frac{x_i^2 +\epsilon^\kappa}
{(a+|x_i|^p)^{\frac 2p}}\omega_i
+ \frac{2-p}{p}\omega_i ^{-\frac{p}{2-p}}\right],
\end{equation}
where $\kappa$ is a positive parameter.
The \emph{modified} IRLS \emph{algorithm} is to alternately update
the minimizer $x$, the weight $\omega$ ,
and $\epsilon>0$, with $x$ and $\omega$
depending on $\mathcal{J}_{a,p}$,
which is described as in Algorithm \ref{algo1}.

\begin{algorithm}
\caption{Modified IRLS}
\label{algo1}
\hspace*{1em}\textbf{Input:} $A\in {\mathbb R}^{M\times N}$, $y\in {\mathbb R}^M$,
$\delta>0$, $s\in\mathbb N$ \\
\hspace*{1em}\textbf{Define:} $\varepsilon_{\mathrm{outer}}>0$ \\
\hspace*{1em}\textbf{Initialize:} $x^0=\mathbf{0}$ and $\epsilon_0=1$ \\
\hspace*{1em}\textbf{For} $n=0,1,2,\ldots$ \quad \textbf{do} \\
\hspace*{2em} $\omega^{n}=\displaystyle\mathop{\arg\min\,}_{
\omega:=(\omega_1,\ldots,\omega_N)\ \mathrm{with\ each}\ \omega_i>0} \mathcal{J}_{a,p}(x^n,\omega,\epsilon_n)$ \\
\hspace*{2em} $x^{n+1}=\displaystyle\mathop{\arg\min\,}_{Ax =y} \mathcal{J}_{a,p}(x,\omega^{n},\epsilon_n)$ \\
\hspace*{2em} $\epsilon_{n+1}=\min\left\{\epsilon_n , \frac{r(x^{n+1})_{s+1}}{\delta}\right\}$ \\
\hspace*{2em} \textbf{if} $\|x^{n+1}-x^n\|_{\infty}< \varepsilon_{\mathrm{outer}}$ \quad\textbf{then break} \\
\hspace*{1em}\textbf{Output:} $x_{\mathrm{new}}=x^{n+1}$
\end{algorithm}

In each step, the weight $\omega^{n}:=(\omega_1^{n},\ldots,\omega^{n}_N)$ is updated as
\begin{equation}\label{eq-update,w}
\omega^{n}_i :=\frac{(|x_i^n|^2+\epsilon_n^\kappa)^{\frac{p-2}{2}}}
{(a+|x_i^n|^p)^{1-\frac{2}{p}}},   \quad i\in\{1,\ldots,N\}.
\end{equation}
Once the new weight is found, to update $x$,
we need to solve a weighted minimization problem
\begin{equation}\label{wsp}
\min_{Ax=y} (a+1)\sum_{i=1}^N \frac{x_i^2+\epsilon_n^\kappa}{(a+|x_i|^p)^{\frac{2}{p}}}\omega_i^{n}.
\end{equation}
Instead, we consider its approximation
\begin{align}\label{eq-approximation}
\min_{Ax=y} (a+1)\sum_{i=1}^N
\left[\frac{x_i^2}{a+|x_i|^p}+\frac{\epsilon_n^\kappa}{a+|x^n_i|^p}\right]
\left[\frac{a+|x^n_i|^p}{(a+|x^n_i|^p)^{\frac{2}{p}}}
\omega_i^{n}\right]  .
\end{align}
Now, we introduce $w^n:=(w^n_1,\ldots,w^n_N)$ by setting, for any $i\in\{1,\ldots,N\}$,
\begin{equation}\label{eq-def-w}
w_i^{n}:=\frac{a+|x^n_i|^p}{(a+|x_i^n|^p)^{\frac{2}{p}}}\omega_i^{n}
= \left(|x_i^n|^2+\epsilon_n^\kappa\right)^{\frac{p-2}{2}}.
\end{equation}
Then \eqref{eq-approximation} can be  recast as
\begin{align*}
\min_{Ax=y} (a+1)\sum_{i=1}^N \frac{x_i^2}{a+|x_i|^p}w_i^{n}.
\end{align*}
Here, one can also see that the introduction of $\epsilon_{n}$ regularizes $w^n$ because
$\|w^{n}\|_\infty \le \epsilon_{n}^{-\frac{\kappa(2-p)}{2}}$.

Next, we use the modified IRLS method to solve
the unconstrained minimization
\begin{equation}\label{eq-unproblem}
\min_{x\in\mathbb{R}^N} Q_{a,p}(x) := \min_{x\in\mathbb{R}^N} \lambda P_{a,p}(x) + \frac{1}{2}\left\|Ax-y\right\|_2^2,
\end{equation}
where $\lambda\in(0,\infty)$ is the regularity parameter.

The concrete algorithm is presented as {\color{red}in Algorithm \ref{alg2}}.

\begin{algorithm}
\caption{IRLSTLp  for  unconstrained $P_{a,p}$ minimization \eqref{eq-unproblem}}\label{alg2}
\hspace*{1em}\textbf{Input:} $A\in {\mathbb R}^{M\times N}$, $y\in {\mathbb R}^M$,
$\delta>0$, $s\in\mathbb N$, and $\kappa>0$ \\
\hspace*{1em}\textbf{Define:} $\varepsilon_{\mathrm{outer}},\varepsilon'_{\mathrm{outer}}>0$ \\
\hspace*{1em}\textbf{Initialize:} $x^0=\mathbf{0}$ and $\epsilon_0=1$ \\
\hspace*{1em}\textbf{For} $n=0,1,2,\ldots$ \quad \textbf{do} \\
\hspace*{2em} $w_i^{n}= (|x_i^n|^2+\epsilon_n^\kappa)^{\frac{p-2}{2}}$ \\
\hspace*{2em} $x^{n+1}=\displaystyle\mathop{\arg\min\,}_{x:=(x_1,\ldots,x_N)\in \mathbb{R}^N}
  \lambda (a+1)\sum_{i=1}^N \frac{x_i^2}{a+|x_i|^p}w_i^{n}+ \frac{1}{2}\|Ax-y\|_2^2$ \\
\hspace*{2em} $\epsilon_{n+1}=\min\left\{\epsilon_n , \frac{r(x^{n+1})_{s+1}}{\delta}\right\}$ \\
\hspace*{2em} \textbf{if} $|r(x^{n+1})_{s+1}-r(x^{n})_{s+1}|<\varepsilon_{\mathrm{outer}}$ \textbf{or} $r(x^{n+1})_{s+1}<\varepsilon'_{\mathrm{outer}}$ \quad\textbf{then break} \\
\hspace*{1em}\textbf{Output:} $x_{\mathrm{new}}=x^{n+1}$
\end{algorithm}

In general, the algorithm generates  a sequence $\{w^n,x^n\}_n$
and a decreasing sequence $\{\epsilon_n\}_n$ of nonnegative numbers;
especially, if $\epsilon_{n_0}=0$ for some $n_0\in\mathbb{N}$, we stop the algorithm
and define $x^k:=x^{n_0}$ and $w^{k-1}:=w^{n_0-1}$
for $k>n_0$.

In each step of Algorithm \ref{alg2}, to update $x$, we need to
solve an unconstrained sub-problem
\begin{equation}\label{eq-subproblem}
\min_{x\in \mathbb{R}^N}  \lambda (a+1)\sum_{i=1}^N \frac{x_i^2}{a+|x_i|^p}w_i^{n}
+ \frac{1}{2}\left\|Ax-y\right\|_2^2,
\end{equation}
where $w_i^{n}$ is defined as in \eqref{eq-def-w}.
Although the sub-problem \eqref{eq-subproblem} is convex,
its explicit solution is hard to obtain.
To escape this, we use the difference of convex functions (DC) programming (see, for example,
\cite{DL98}).
Note that the function $\rho(t):=t^2/(a+|t|^p)$, $t\in\mathbb{R}$, can
be written as the following difference of two convex functions:
$$
\rho(t)=\frac{t^2}{a}-\left(\frac{t^2}{a}-\frac{t^2}{a+|t|^p}\right)
=\frac{t^2}{a}-\frac{|t|^{p+2}}{a(a+|t|^p)}.
$$
Then, for a given weight $w:=(w_1,\ldots,w_N)$, we define the function $f_w$
on $\mathbb{R}^N$ by setting, for any $x:=(x_1,\ldots,x_N)\in\mathbb{R}^N$,
$$f_w(x):=\lambda (a+1)\sum_{i=1}^N \frac{x_i^2}{a+|x_i|^p}w_i
+ \frac{1}{2}\left\|Ax-y\right\|_2^2,$$
and one can easily obtain a DC decomposition of $f_w$ as
\begin{equation}\label{e3.7}
f_w = g_w-h_w,
\end{equation}
where, for any $x\in \mathbb{R}^N$,
$$
g_w(x):= \frac{\lambda (a+1)}{a}\|x\|^2_{\ell_2(w)} + \frac{1}{2}\left\|Ax-y\right\|_2^2 + c\|x\|_2^2
$$
and
\begin{align*}
h_w(x):&= \frac{\lambda (a+1)}{a}\left[\|x\|^2_{\ell^2(w)} -
\sum_{i=1}^N w_i \frac{x_i^2}{a+|x_i|^p}\right] + c\|x\|_2^2\\
&=\frac{\lambda (a+1)}{a} \sum_{i=1}^N w_i \frac{|x_i|^{p+2}}{a(a+|x_i|^p)} + c\|x\|_2^2\\
&=:\frac{\lambda (a+1)}{a} \varphi_w(x) + c\|x\|_2^2
\end{align*}
with $\|x\|^2_{\ell_2(w)}:= \sum_{i=1}^N w_ix_i^2$ and
$c$ being any given positive constant.
We note that the additional term $c\|x\|_2^2$ with $c>0$ here is used
to promote the convexity of both $g_w$ and $h_w$.
We can also calculate the gradient $\nabla f_w$ of $f_w$ as
that, for any $x\in\mathbb{R}^N$,
{\color{red}\begin{equation}\label{e3.6}
\nabla f_w(x) = \frac{2\lambda (a+1)}{a} W x  -\frac{\lambda (a+1)}{a}
\nabla\varphi_{w}(x) + A^T Ax-A^T y,
\end{equation}}
where $W:=\mathrm{diag}(w_1,\ldots,w_N)$.

Then the DCA algorithm for the sub-problem \eqref{eq-subproblem} is described
in Algorithm \ref{alg3}.

\begin{algorithm}
\caption{DCA for the weighted unconstrained sub-problem  \eqref{eq-subproblem}}
\label{alg3}
\hspace*{1em}\textbf{Input:} $A\in {\mathbb R}^{M\times N}$, $y\in {\mathbb R}^M$,
and $w^n:=(w^n_1,\ldots,w^n_N)\in \mathbb{R}^N$ with each $w^n_i$ as in \eqref{eq-def-w}\\
\hspace*{1em}\textbf{Define:} $\varepsilon_{\mathrm{inner}}>0$ and $W=\mathrm{diag}(w^n_1,\ldots,w^n_N)$ \\
\hspace*{1em}\textbf{Initialize:} $x^0=\mathbf{0}$ \\
\hspace*{1em}\textbf{For} $k=0,1,2,\ldots$ \quad \textbf{do} \\
\hspace*{2em} $v^k=\nabla h_{w^n}(x^k)=\frac{\lambda (a+1)}{a} \nabla\varphi_{w^n}(x^k) + 2cx^k$ \\
\hspace*{2em} $x^{k+1}=\displaystyle\mathop{\arg\min\,}\{x\in \mathbb{R}^N :\ g_{w^n}(x)-\langle x, v^k\rangle\}$ \\
\hspace*{4em} $=\left[A^T A +  2cI + \frac{2\lambda (a+1)}{a}W\right]^{-1}(A^T y + v^k)$ \\
\hspace*{2em} \textbf{if} $\|x^{k+1}-x^k\|_{\infty}< \varepsilon_{\mathrm{inner}}$ \quad \textbf{then break} \\
\hspace*{1em}\textbf{Output:} $x_{\mathrm{new}}=x^{k+1}$
\end{algorithm}

{\color{red}
\begin{remark}
IRLS and DCA are standard optimization techniques.
The contribution presented in the above Algorithms \ref{algo1}, \ref{alg2},
and \ref{alg3} is their specialization
to the nonconvex nonhomogeneous $P_{a,p}$ penalty: the approximation in
\eqref{eq-approximation} produces the weighted objective in
\eqref{eq-subproblem}, and the DC decomposition in \eqref{e3.7} makes that
objective computationally accessible. This construction has nested outer and inner
iterations and can therefore incur more computation than a single-loop method.
The detailed runtime comparisons among IRLSTLp, DCATL1,
and TL1IT can be found in Section \ref{sec-test} below.
\end{remark}
}

\subsection{Convergence of Modified IRLS Algorithm}\label{subsec-convergIRLS}

In this subsection, we establish some convergence results of the algorithms.
We begin with the following concept of the null space property (NSP) (see \cite[(3.8)]{CDD08} with $X:=P_p$).
\begin{definition}
Let $p\in(0,1]$, $A\in\mathbb{R}^{M\times N}$, and $k\in\mathbb{N}$.
The matrix $A$ is said to have the \emph{$p$-Null Space Property} ($p$-NSP)
of order $k$ for some $\gamma\in(0,\infty)$ if,
for any $x\in \mathrm{Ker}\, A$
and any set $T$ with cardinality $|T|\le k$,
$
\|x_T\|_p^p\le \gamma\|x_{T^\complement}\|_p^p.
$
\end{definition}

A well-known result is that  RIP implies $p$-NSP; see \cite{CDD08}.
In addition, Sun \cite{S11} introduced the sparse approximation property
which is a weaker variant of the RIP but stronger than $p$-NSP.

We also need several auxiliary lemmas.
The first one is precisely \cite[Lemma 4.1]{DDFG10}.
In what follows,
for any $x\in\mathbb{R}^N$,
let $\{r(x)_i\}_{i=1}^N$ and $\{\sigma_j(x)_1\}_{i=1}^N$
be as  in \eqref{eq-1.4}.

\begin{lemma}\label{lem-Lip}
The map $x\mapsto r(x):=\{r(x)_{i}\}_{i=1}^N$
mapping $\mathbb{R}^N$ to $\mathbb{R}^N$ is Lipschitz continuous on $(\mathbb{R}^N,\|\cdot\|_{\infty})$; that is,
for any $x_1, x_2\in\mathbb{R}^N$,
$$
\left\|r(x_1)-r(x_2)\right\|_\infty \le \left\|x_1-x_2\right\|_\infty.
$$
Moreover,  for any $j\in\{1,\ldots,N\}$ and $x_1,x_2\in\mathbb{R}^N$,
$$
\left|\sigma_j(x_1)_1 - \sigma_j(x_2)_1\right|\le \left\|x_1 - x_2\right\|_1
$$
and, for any $i>j$,
\begin{equation*}
(i-j) r(x_1)_i \le \|x_1 -x_2\|_1 + \sigma_j(x_2)_1.
\end{equation*}
\end{lemma}

\begin{lemma}\label{prop-J}
Let  $\{x^n\}_{n\in\mathbb{N}}$, $\{\omega^n\}_{n\in\mathbb{N}}$,
and $\{\epsilon_n\}_{n\in\mathbb{N}}$  be the sequences
generated by Algorithm \ref{algo1}. Then the following assertions hold.
\begin{enumerate}
\item[\rm(i)]
For any $n\in\mathbb{N}$,
\begin{equation}\label{eq-equ,J}
\mathcal{J}_{a,p}\left(x^{n},\omega^{n},\epsilon_{n}\right)
= \sum_{i=1}^N  \frac{(a+1)(|x^n_i|^2+\epsilon_n^\kappa)^{\frac p2}}{a+|x^n_i|^p}.
\end{equation}
\item[\rm(ii)]
The functional $\mathcal{J}_{a,p}$ has the following monotonously decreasing
property
\begin{align}\label{eq-J,decrease}
\mathcal{J}_{a,p}\left(x^{n+1},\omega^{n+1},\epsilon_{n+1}\right)
&\le\mathcal{J}_{a,p}\left(x^{n+1},\omega^{n},\epsilon_{n+1}\right)\\
&\le\mathcal{J}_{a,p}\left(x^{n+1},\omega^{n},\epsilon_{n}\right)
\le\mathcal{J}_{a,p}\left(x^{n},\omega^{n},\epsilon_{n}\right).\nonumber
\end{align}
\item[\rm(iii)]
$\{\epsilon_n\}_{n\in\mathbb{N}}$ converges to some $\epsilon\ge 0$ as $n\to\infty$.
\item[\rm(iv)]
For any $n\in\mathbb{N}$, $P_{a,p}(x^n)\le \mathcal{J}_{a,p}(x^0,\omega^0,\epsilon_0)$.
\item[\rm(v)]  $\{x^n\}_{n\in\mathbb{N}}$ is $\ell_\infty$-bounded.
\end{enumerate}
\end{lemma}

\begin{proof}
Property (i) follows from \eqref{eq-update,w}, and property
(ii) is a consequence of both the decreasing property of $\{\epsilon_n\}_{n\in\mathbb{N}}$ and the minimization property.
Property (iii) follows from each $\epsilon_n\ge 0$ for any $n\in\mathbb{N}$ and
the decreasing property of $\{\epsilon_n\}_{n\in\mathbb{N}}$.
By  \eqref{eq-equ,J} and \eqref{eq-J,decrease}, property
(iv) can be derived as follows, for any $n\in \mathbb{N}$,
$$
P_{a,p}(x^n)\le\mathcal{J}_{a,p}\left(x^{n},\omega^{n},\epsilon_{n}\right)
\le \mathcal{J}_{a,p}\left(x^{0},\omega^{0},\epsilon_{0}\right).
$$
Property (v) is a consequence of (iv) and the strictly increasing property of $\rho_{a,p}$.
This completes the proof of Lemma \ref{prop-J}..
\end{proof}

\begin{lemma}\label{lem-NSP}
Assume that $c_0\in(0,\infty)$, $y\in\mathbb{R}^M$, and
a matrix $A$ has the $p$-\emph{NSP} of order $s\in\mathbb{N}$ for some $\gamma\in(0,\frac{a}{a+(2c_0)^p})$.
Then, for any $x,x'\in \mathbb{R}^N$ with $Ax=y$, $Ax'=y$, $\|x_{T^\complement}\|_\infty\le c_0$,
and $\|x'_{T^\complement}\|_\infty\le c_0$,
$$
P_{a,p}\left(x'-x\right)\le C_{a,p,s,\gamma,c_0}\left[P_{a,p}(x')-P_{a,p}(x)+2\sum_{i=s+1}^N \rho_{a,p}\left(r(x)_i\right)\right],
$$
where $C_{a,p,s,\gamma,c_0}$ is a positive constant depending only on $a,p,s,\gamma$, and $c_0$.
\end{lemma}

\begin{proof}
Let $T$ be   the index set of the $s$ largest components of $x$ in magnitude.
Then, by Lemma \ref{lem-TLp,prop}(vi) and the separability of $P_{a,p}$, we find that
\begin{align*}
P_{a,p}\left(x'_{T^\complement}-x_{T^\complement}\right)&\le P_{a,p}
\left(x'_{T^\complement}\right)+P_{a,p}\left(x_{T^\complement}\right)
=P_{a,p}\left(x'\right)-P_{a,p}\left(x'_T\right)+P_{a,p}\left(x_{T^\complement}\right)\\
&=P_{a,p}\left(x\right)+P_{a,p}\left(x'\right)-P_{a,p}\left(x\right)-P_{a,p}
\left(x'_T\right)+P_{a,p}\left(x_{T^\complement}\right)\\
&=P_{a,p}\left(x_T\right)-P_{a,p}\left(x'_T\right)  +P_{a,p}
\left(x'\right)-P_{a,p}\left(x\right) +2P_{a,p}\left(x_{T^\complement}\right)\\
&\le P_{a,p}\left(x'_T-x_T\right) + P_{a,p}\left(x'\right)-P_{a,p}
\left(x\right) +2P_{a,p}\left(x_{T^\complement}\right).
\end{align*}
Using the fact that $x'-x\in \mathrm{Ker}\, A$ and
the $p$-NSP of $A$, together with the assumptions that $\|x_{T^\complement}\|_\infty\le c_0$
and $\|x'_{T^\complement}\|_\infty\le c_0$, we further conclude that
\begin{align*}
P_{a,p}\left(x'_{T}-x_{T}\right)&\le \frac{a+1}{a}\left\|x'_{T}-x_{T}\right\|_p^p
\le \gamma\frac{a+1}{a}\left\|x'_{T^\complement}-x_{T^\complement}\right\|_p^p\le \gamma\frac{a+(2c_0)^p}{a}P_{a,p}\left(x'_{T^\complement}-x_{T^\complement}\right)\\
&\le \gamma\frac{a+(2c_0)^p}{a}
\left[P_{a,p}\left(x'_T-x_T\right) + P_{a,p}\left(x'\right)-P_{a,p}\left(x\right) +2P_{a,p}\left(x_{T^\complement}\right)\right],
\end{align*}
and hence
\begin{align*}
P_{a,p}\left(x'_{T}-x_{T}\right)
\le \frac{\gamma[a+(2c_0)^p]}{a-\gamma[a+(2c_0)^p]}
\left[ P_{a,p}\left(x'\right)-P_{a,p}\left(x\right) +2P_{a,p}\left(x_{T^\complement}\right)\right].
\end{align*}
Altogether, we obtain
\begin{align*}
P_{a,p}\left(x'-x\right)&\le P_{a,p}\left(x'_{T^\complement}-x_{T^\complement}\right)+P_{a,p}\left(x'_{T}-x_{T}\right)\\
& \le C_{a,p,s,\gamma}\left[ P_{a,p}\left(x'\right)-P_{a,p}\left(x\right) +2P_{a,p}\left(x_{T^\complement}\right)\right].
\end{align*}
This completes the proof of Lemma \ref{lem-NSP}.
\end{proof}

Now, we give the convergence result of Algorithm \ref{algo1}.

\begin{theorem}\label{thm-convergence}
Let  $\{x^n\}_{n\in\mathbb{N}}$ and $\{\epsilon_n\}_{n\in\mathbb{N}}$ be two sequences
generated by Algorithm \ref{algo1},
$s\in\mathbb{N}$ be the same positive integer as in Algorithm \ref{algo1},
and the matrix $A$ be as in Lemma \ref{lem-NSP}.
If $\epsilon_n\to0^+$ as $n\to\infty$,
then $x^n$ converges to an $s$-sparse vector $x^*$ as $n\to\infty$.
\end{theorem}

\begin{proof}
We consider two cases for $\{\epsilon_n\}_{n\in\mathbb{N}}$.

\textbf{Case 1.} Suppose that there exists some $n_0\in\mathbb{N}$ such that $\epsilon_{n_0}=0$
and $\epsilon_{n_0-1}>0$.
Then we find that, for any $n\ge n_0$, $x^n=x^{n_0}$, which means
$x^*:=\lim_{n\to\infty}x^n=x^{n_0}$.
Furthermore, by $\epsilon_{n_0-1}>0$, we also conclude that $r(x^{n_0})_{s+1}=0$,
which implies that $x^{n_0}$ is $s$-sparse, and hence $x^*$ is also $s$-sparse.

\textbf{Case 2.} Suppose that, for any $n\in\mathbb{N}$,  $\epsilon_{n}>0$.
We may choose a subsequence $\{\epsilon_{n_k}\}_{k\in\mathbb{N}}$ such that
$\epsilon_{n_{k+1}}<\epsilon_{n_k}$ for any $k\in\mathbb{N}$.
By Lemma \ref{prop-J}(v), there exists a subsequence of $\{x^{n_k}\}_{k\in\mathbb{N}}$,
which we still denote by $\{x^{n_k}\}_{k\in\mathbb{N}}$ for simplicity,
converging to some  $x^*$ in $\ell_\infty$. Applying Lemma \ref{lem-Lip}, we conclude that
$r(x^{n_k})_{s+1}$ also converges to $r(x^*)_{s+1}$ as $k\to \infty$ and therefore,
by the strictly decreasing property of $\{\epsilon_{n_k}\}_k$,
$$
r(x^*)_{s+1}=\lim_{k\rightarrow\infty}r(x^{n_k})_{s+1}
\le \lim_{k\rightarrow\infty}\delta\epsilon_{n_k -1}=0,
$$
which means that $x^*$ is $s$-sparse.
Next, we show $x^n\to x^*$ as $n\to\infty$.
Indeed, on the one hand, by $x^{n_k}\to x^*$ and $\epsilon_{n_k}\to 0$
as $k\to\infty$ and   equality \eqref{eq-equ,J},
we have
$$\lim_{k\rightarrow\infty}\mathcal{J}_{a,p}\left(x^{n_k},\omega^{n_k},
\epsilon_{n_k}\right)=P_{a,p}(x^*).$$
Since $\mathcal{J}_{a,p}$ is monotonously decreasing, from this we further deduce that
$$\lim_{n\rightarrow\infty}\mathcal{J}_{a,p}\left(x^{n},\omega^{n},
\epsilon_{n}\right)=P_{a,p}(x^*).$$
On the other hand, by the elementary inequality \eqref{eq-elementary},
we find that, for any $n\in\mathbb{N}$,
$$
\mathcal{J}_{a,p}\left(x^{n},\omega^{n},\epsilon_{n}\right)-N\frac{a+1}{a}
\epsilon_n^{\frac{\kappa p}{2}}\le P_{a,p}(x^n)
\le\mathcal{J}_{a,p}\left(x^{n},\omega^{n},\epsilon_{n}\right).
$$
Thus, we obtain $P_{a,p}(x^n)\to P_{a,p}(x^*)$ as $n\to\infty$.
Using this and Lemma \ref{lem-NSP}, we further conclude that
$$
\lim_{n\rightarrow\infty}P_{a,p}\left(x^n-x^*\right)\le
C \left[\lim_{n\rightarrow\infty}P_{a,p}(x^n)-P_{a,p}(x^*)\right]=0
$$
with $C$ as in Lemma \ref{lem-NSP}, which shows $x^n\to x^*$ as $n\to\infty$.
This then completes the proof of Theorem \ref{thm-convergence}.
\end{proof}

\subsection{Convergence of DC Algorithm}\label{subsec-convergDCA}

In this subsection, we establish the convergence of Algorithm \ref{alg3} for the sub-problem  \eqref{eq-subproblem}.
We begin with a concept of moduli of strong convexity; see \cite[(5)]{DL98}.

\begin{definition}
Let $f$ be a convex function defined on $\mathbb{R}^N$. The \emph{modulus of strong convexity}
$m(f)$ of $f$ is defined by setting
$$
m(f):=\sup\left\{\rho\in[0,\infty):\ f-\frac{\rho}{2}\|\cdot\|_2^2 \text{ is convex on } \mathbb{R}^N\right\}.
$$
\end{definition}

The following lemma is a part of \cite[Proposition A.1]{DL98}.
\begin{lemma}\label{lem-DC1}
Let $f=g-h$ be a DC decomposition with $m(g)>0$ and $m(h)>0$,
and let $\{x^n\}_{n\in\mathbb{N}}$ be a sequence generated by DCA. Then, for any $n\in\mathbb{N}$,
$$
\left\|x^{n+1}-x^n\right\|_2^2 \le \frac{2}{m(g)+m(h)} \left[ f(x^n) - f(x^{n+1})\right].
$$
\end{lemma}

Now, we establish the convergence  of Algorithm \ref{alg3}.

\begin{theorem}\label{thm-con-DCA}
Let $w:=(w_1,\ldots,w_N)\in\mathbb{R}^N$ be the input weight in
Algorithm \ref{alg3} with each $w_i >0$
and $\{x^n\}_{n\in\mathbb{N}}$ and $\{f_w(x^n)\}_{n\in\mathbb{N}}$ be
two sequences generated by Algorithm \ref{alg3}. Then
the following assertions hold.
\begin{enumerate}
\item[\rm(i)] $\{f_w(x^n)\}_{n\in\mathbb{N}}$ is decreasing and convergent.

\item[\rm(ii)] $\{x^n\}_{n\in\mathbb{N}}$ has the asymptotic regularity that
$$
\lim_{n\rightarrow\infty}\left\|x^{n+1}-x^n\right\|_2=0.
$$
\item[\rm(iii)] If $\lambda>\frac{\|w^{-1}\|_\infty \| y \|_2^2}{2(a+1)}$,
then $\{x^n\}_{n\in\mathbb{N}}$ is $\ell_\infty$-bounded
and, for any accumulation point $x^*$ of $\{x^n\}_{n\in\mathbb{N}}$,
$\nabla f_w(x^*)=0,$ where $w^{-1}:=(\frac{1}{w_1},\ldots,\frac{1}{w_N})$.
\end{enumerate}
\end{theorem}

\begin{proof}
To prove (i), we decompose $f_w=g_w-h_w$ as in \eqref{e3.7}
and then, by the definitions of $g_w$ and $h_w$,
we easily conclude $m(g_w)\ge 2c$ and $m(h_w)\ge 2c$ with $c>0$
as in the definitions of $g_w$ and $h_w$. From this and
Lemma \ref{lem-DC1}, it follows that
$f_w(x^n) - f_w(x^{n+1})\ge 0$ for any $n\in\mathbb{N}$,
which further implies its decreasing property.
Moreover, this, combined with the fact that $\{f_w(x^n)\}_{n\in\mathbb{N}}$ are nonnegative,
yields the convergence of $\{f_w(x^n)\}_{n\in\mathbb{N}}$. This completes the proof of (i).

Now, we show (ii). Since $\{f_w(x^n)\}_{n\in\mathbb{N}}$ converges, from Lemma \ref{lem-DC1}
we deduce that
$$
\left\|x^{n+1}-x^n\right\|_2^2 \le \frac{f_w(x^n) - f_w(x^{n+1})}{2c} \rightarrow 0    \quad\text{as}\ n\rightarrow\infty,
$$
which completes the proof of (ii).

To prove (iii), by the decreasing property of $\{f_w(x^n)\}_{n\in\mathbb{N}}$
and $x^0=\mathbf{0}$, we have
$$
\lambda (a+1)\sum_{i=1}^N \frac{|x^n_i|^2}{a+|x^n_i|^p}w_i
+ \frac{1}{2}\left\|Ax^n-y\right\|_2^2 =f_w(x^n)\le f_w(x^0)=\frac{1}{2}\left\| y\right\|_2^2.
$$
This  implies that, for each $i\in\{1,\ldots,N\}$,
$$
 \frac{|x^n_i|^2}{a+|x^n_i|^p} \le\frac{\|w^{-1}\|_\infty \| y \|_2^2}{2\lambda (a+1)}.
$$
Thus, if $\lambda>\frac{\|w^{-1}\|_\infty \| y \|_2^2}{2(a+1)}$,
we find that, for each $i\in\{1,\ldots,N\}$,
either $|x^n_i|\le 1$ or, by $|x_i^n|^2>|x_i^n|^p$ when $|x^n_i|>1$,
$$
|x^n_i|^2\le\frac{a\|w^{-1}\|_\infty \| y \|_2^2}{2\lambda (a+1)-\|w^{-1}\|_\infty \| y \|_2^2},
$$
which means that $\{x^n\}$ is $\ell_\infty$-bounded.

Now, let $\{x^{n_k}\}_{k\in\mathbb{N}}$ be a subsequence of $\{x^n\}_{n\in\mathbb{N}}$ with the limit point $x^*$.
Then, in the $n_k$-th step of Algorithm \ref{alg3}, by \eqref{e3.6} we have
\begin{align*}
0&= \nabla g_w(x^{n_k})-v^{n_k-1}\\
&=A^T Ax^{n_k}-A^T y + \frac{2\lambda (a+1)}{a} W x^{n_k} + 2c(x^{n_k}-x^{n_k -1})
-\frac{\lambda (a+1)}{a} \nabla\varphi_{w}(x^{n_k-1})  ,
\end{align*}
where $W:=\mathrm{diag}(w_1,\ldots,w_N)$.
Noting $\|x^{n+1}-x^n\|_2\to0 $ as $n\to\infty$
and $\|x^{n_k}-x^*\|_\infty\to0 $  as $k\to \infty$,
by setting $k\to\infty$,
we further obtain
\begin{align*}
0=A^T Ax^*-A^T y + \frac{2\lambda (a+1)}{a} W x^{*}  -\frac{\lambda (a+1)}{a}
\nabla\varphi_{w}(x^*)= \nabla f_w(x^*) ,
\end{align*}
which completes the proof of Theorem \ref{thm-con-DCA}.
\end{proof}

{\color{red}\begin{remark}\label{limitations}
Observe that the IRLSTLp algorithm in Algorithm \ref{alg2}
has some theoretical limitations due to the absence of its complete
convergence result, although Theorem \ref{thm-convergence} provides
the convergence analysis of the modified IRLS algorithm
and Theorem \ref{thm-con-DCA} provides the convergence
analysis of the DCA algorithm.
Indeed, since directly solving its weighted subproblem \eqref{wsp}
of the modified IRLS algorithm is difficult,
to circumvent this Algorithm \ref{alg2} introduces
the approximation \eqref{eq-approximation} of \eqref{wsp}.
However, the proof of Theorem \ref{thm-convergence} does not
work anymore if \eqref{wsp} is replaced by its
approximation \eqref{eq-approximation}, whose convergence is still unknown.
Instead of this, we in Subsection \ref{subsec-emconverg} exhibit
the convergent behavior by some numerical experiments.
\end{remark}}

\section{Numerical Experiments}\label{sec-test}

In this section, we   test
the performance of the proposed IRLSTLp algorithm for the unconstrained $P_{a,p}$ minimization.
All the experiments were performed on a Thinkpad desktop with 32 GB of RAM
and 13-th Generation Intel Core i9-13900H Processor.

In Subsection \ref{subsec-exp-a,p}, we test the IRLSTLp with different parameters $a$ and $p$;
in Subsection \ref{subsec-exp-compareTL1}, we compare the IRLSTLp in the case $p=1$
with two algorithms (the DCA  and the iterative thresholding) of the TL1 minimization;
in Subsection \ref{subsec-exp-compare}, we compare the IRLSTLp
with the following three algorithms:
\begin{enumerate}
\item[\rm(i)]
DCA of TL1 (DCATL1) \cite{ZX18},
\item[\rm(ii)]
DCA of $\ell_1-\ell_2$ (DCA $\ell_1-\ell_2$) \cite{YLHX15},
\item[\rm(iii)]
IRLS of $\ell_{q}$ (IRucLq) \cite{LXY13}.
\end{enumerate}

The true signal $x_0$ is  a randomly generated sparse vector and
 the recovered vector $x$ is regarded as a success one and recorded if the relative error
 $\|x-x_0\|_2/\|x_0\|_2<10^{-3}$.
 For each numerical test,
 we sample 100 times and calculate its success rate.

 The IRLSTLp algorithm includes double loops.
 The stopping conditions for the inner loop are the relative iteration error
 $\frac{\|x^{n+1}-x^n\|_\infty}{\max\{\|x^{n+1}\|_\infty,1\}}<10^{-8}$
 and the maximum iteration steps  $20$
 while the stopping conditions for the outer loop are the $(s+1)$-th largest component in magnitude
 $(|x^{n+1}|)_{s+1}<10^{-8}$,
the relative iteration error
 $\frac{|(|x^{n+1}|)_{s+1}-(|x^{n}|)_{s+1}|}{\max\{(|x^{n}|)_{s+1},1\}}<10^{-8}$,
 and the maximum iteration steps  $2000$.
The starting point of the iteration also needs to be chosen suitably
because the IRLSTLp algorithm cannot guarantee a global optimization
in general due to the nonconvexity of the problem.
In all these experiments, the iteration is initialized by zero vectors
although the success rate will increase if it is initialized surround the true signal.
 Here we only consider the noiseless case
 and the regularization parameter $\lambda=10^{-6}$.
{\color{red}
The parameter $\kappa$ is chosen among $\{3,4,5\}$.}

\subsection{RD$_{P_{a,p}}$ and Choices of Parameters $a$ and $p$}\label{subsec-exp-a,p}

The $P_{a,p}$ penalty function involves two key parameters $a$ and $p$.
When  $a$ approaches zero, the $P_{a,p}$ penalty approaches  $\ell_0$;
when $a$ approaches infinity and $p$ approaches 1, the $P_{a,p}$ penalty behaves like $\ell_1$.
Thus, the choice of $(a,p)$ pair  influences the success rate of the proposed IRLSTLp algorithm.
In this subsection, we use  $64\times 256$ Gaussian random matrices
generated by the normal  distribution $\mathcal{N}(0,I)$ to
compare the numerical results with different $a$ and $p$.
The true signal $x_0$ is  randomly generated with the sparsity level varying from 14 to 32 with step size 2
and the parameter $\kappa$ is chosen as $3$.

\begin{figure}[ht]
  \centering
  \begin{tabular}{cc}
  \includegraphics[width=6.5cm]{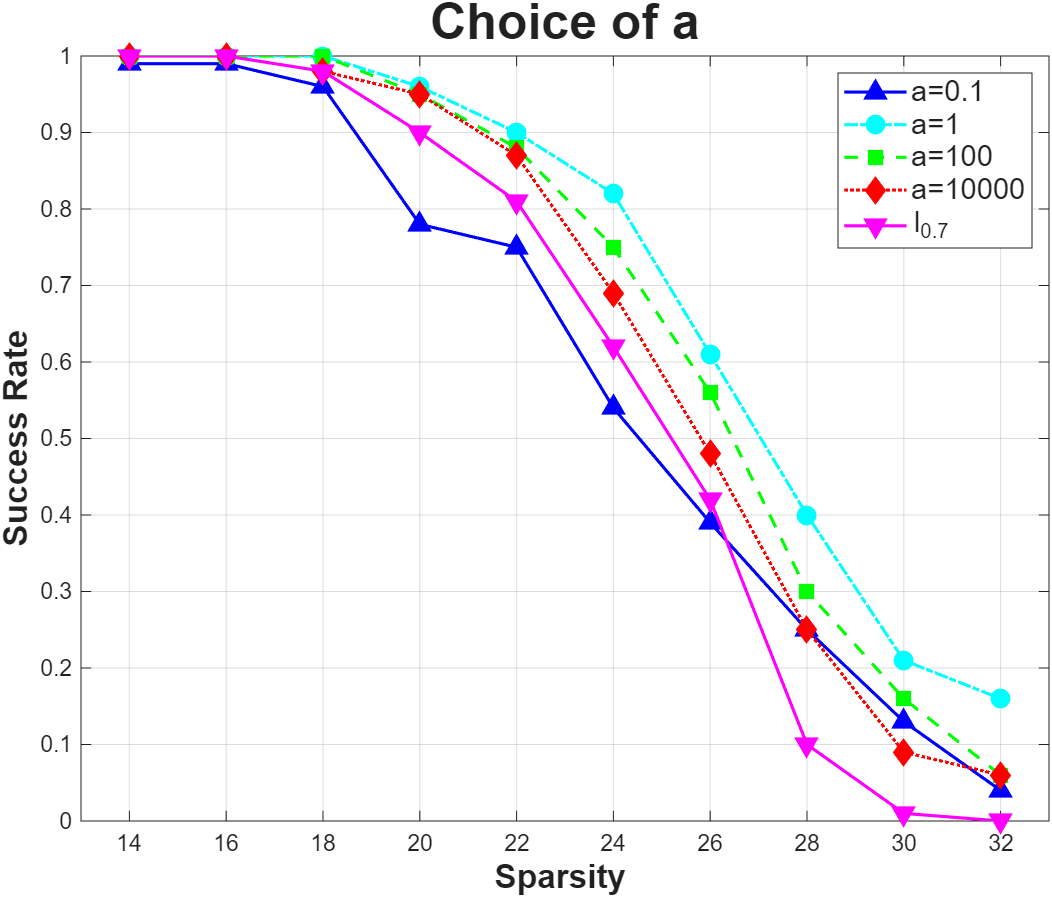}&
  \includegraphics[width=6.5cm]{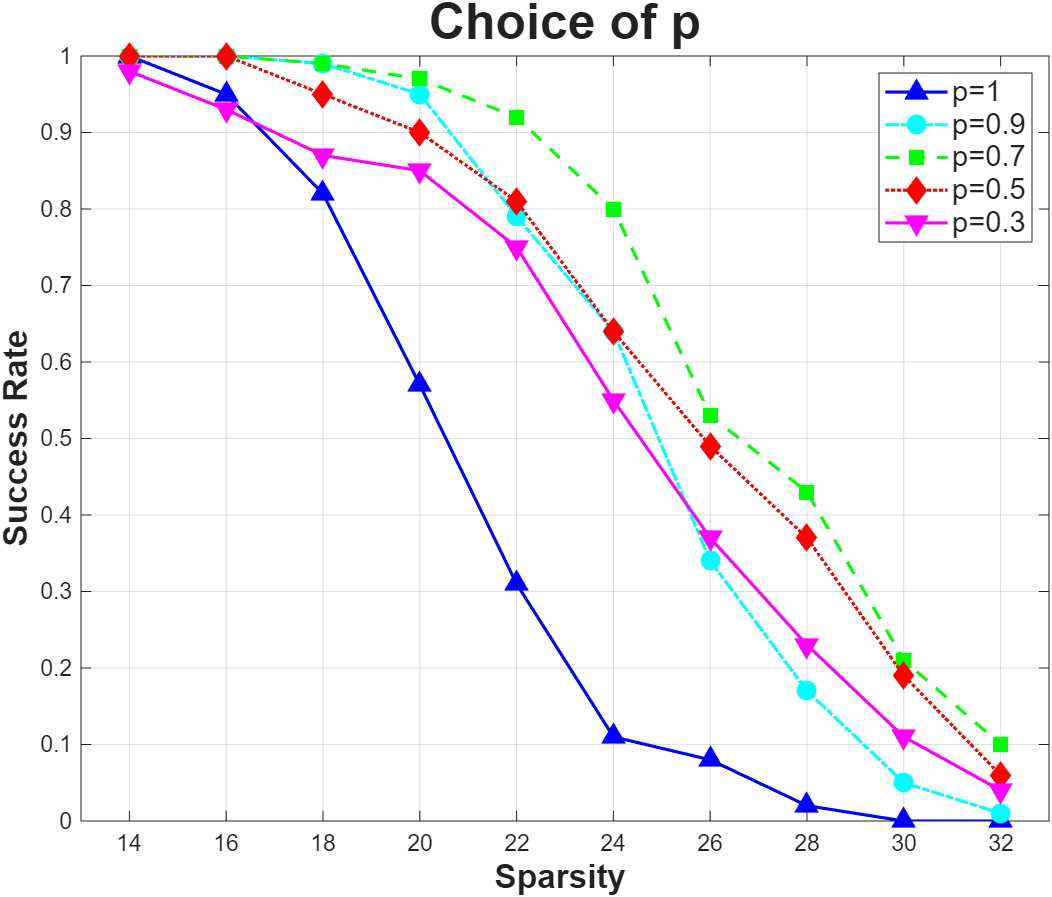}\\
   (a) & (b)
   \end{tabular}
  \caption{Numerical tests on $p$ and $a$ by $64\times 256$ Gaussian matrix}\label{fig-apchoice}
\end{figure}

We first test the IRLSTLp algorithm with different
$a\in\{0.1,1,100,10000\}$
and fixed $p=0.7$.
As the $P_{a,p}$ penalty approaches $P_{p}$ when $a$ approaches infinity,
we also test the IRucLq algorithm for $\ell_{q}$ penalty when $q=0.7$.
Figure \ref{fig-apchoice}(a) presents the success rate.
From this, we can see that, among these tests, the IRLSTLp algorithm with $a=1$ behaves the best.
As `$a$' is getting larger,  the behavior is close to $P_{0.7}$.
As `$a$' is chosen smaller, for instance, $a=0.1$, the success rate obviously decreases.
We also test  the IRLSTLp algorithm
with different $p\in\{0.3,0.5,0.7,0.9,1\}$ and fixed $a=5$,
whose  success rate is presented in  Figure \ref{fig-apchoice}(b).
We find that the algorithm with $p=0.7$ behaves the best among these tests.

{\color{red}
Recall that the relaxation degree defined in Definition \ref{rd} can quantitatively measure
how closely   the penalty function approaches  $\ell_0$.
To examine whether   RD$_{P_{a,p}}$ empirically correlates with recovery,
we conduct numerical experiments  with several pairs $(a,p)$ while keeping
RD$_{P_{a,p}}$ fixed and the sparsity level at $24$.
The success rates are recorded in Table \ref{figRD1}
when RD$_{P_{a,p}}\approx4.2\times 10^{-3}$ and in Table \ref{figRD2}
when RD$_{P_{a,p}}\approx2.4\times 10^{-4}$.
These experiments test the aforementioned correlation between RD$_{P_{a,p}}$ and recovery
within the $P_{a,p}$ family; they do not imply
that the same RD$_P$ values produce the same
recovery behavior for arbitrary penalties.}

\begin{table}[htbp]
\centering
\renewcommand{\arraystretch}{1.15}
\resizebox{0.8\textwidth}{!}{
\begin{tabular}{c|c|c|c|c|c}
\hline
$(a,p)$ & $(0.18,0.9)$ & $(0.3,0.85)$& $(0.54,0.8)$ & $(1.1,0.75)$ & $(4,0.7)$\\
\hline
Success Rate & $76\%$ & $75\%$ & $84\%$ & $79\%$ & $79\%$\\
\hline
\end{tabular}}
\caption{Numerical tests with fixed RD$_{P_{a,p}}\approx4.2\times 10^{-3}$ } \label{figRD1}
\end{table}

\begin{table}[htbp]
\centering
\renewcommand{\arraystretch}{1.15}
\resizebox{0.8\textwidth}{!}{
\begin{tabular}{c|c|c|c|c|c}
\hline
$(a,p)$ & $(0.12,0.7)$ & $(0.23,0.65)$& $(0.5,0.6)$ & $(1.3,0.55)$ & $(1000,0.5)$\\
\hline
Success Rate & $58\%$ & $63\%$ & $58\%$ & $62\%$ & $67\%$\\
\hline
\end{tabular}}
\caption{Numerical tests with fixed RD$_{P_{a,p}}\approx2.4\times 10^{-4}$ } \label{figRD2}
\end{table}

{\color{red} We find that the success rates remain comparable if RD$_{P_{a,p}}$
is held fixed by changing $a$ and $p$.
This supports RD$_{P_{a,p}}$ as a useful indicator.}

Since the experimental results of the IRucLq algorithm
by Gaussian random matrices indicate that the case
$p=0.5$ performs the best (see, for example, \cite{LXY13}) and Figure \ref{fig-apchoice}(b)
tells that the case $p=0.7$ performs the best,
a guess is that, as $a$ increases, the optimal value
of $p$ is likely to decrease.
To verify this, we test the IRLSTLp algorithm
with  different  $a$ varying among $[1,5]$ and $p$ varying among $\{0.1,0.2,\ldots,0.9,1\}$
when the  sparsity level is $24$.
The result is presented in Figure \ref{fig-heat}(a).
As $a$ is a function of $p$ when $N$ and RD$_{P_{a,p}}$ are fixed by Proposition \ref{prop-RD},
we plot the graphs  of $a$ as a function of $p$ with $N=256$
and various RD$_{P_{a,p}}$ in Figure \ref{fig-heat}(b).
{\color{red}
One can observe that each curve nearly coincides with a region filled by the same color;
this approximate alignment is only qualitative rather than exact.
This provides a second indication that RD$_{P_{a,p}}$
serves as an empirical indicator correlated with the model performance.}

\begin{figure}[ht]
  \centering
  \begin{tabular}{cc}
  \includegraphics[width=7.2cm]{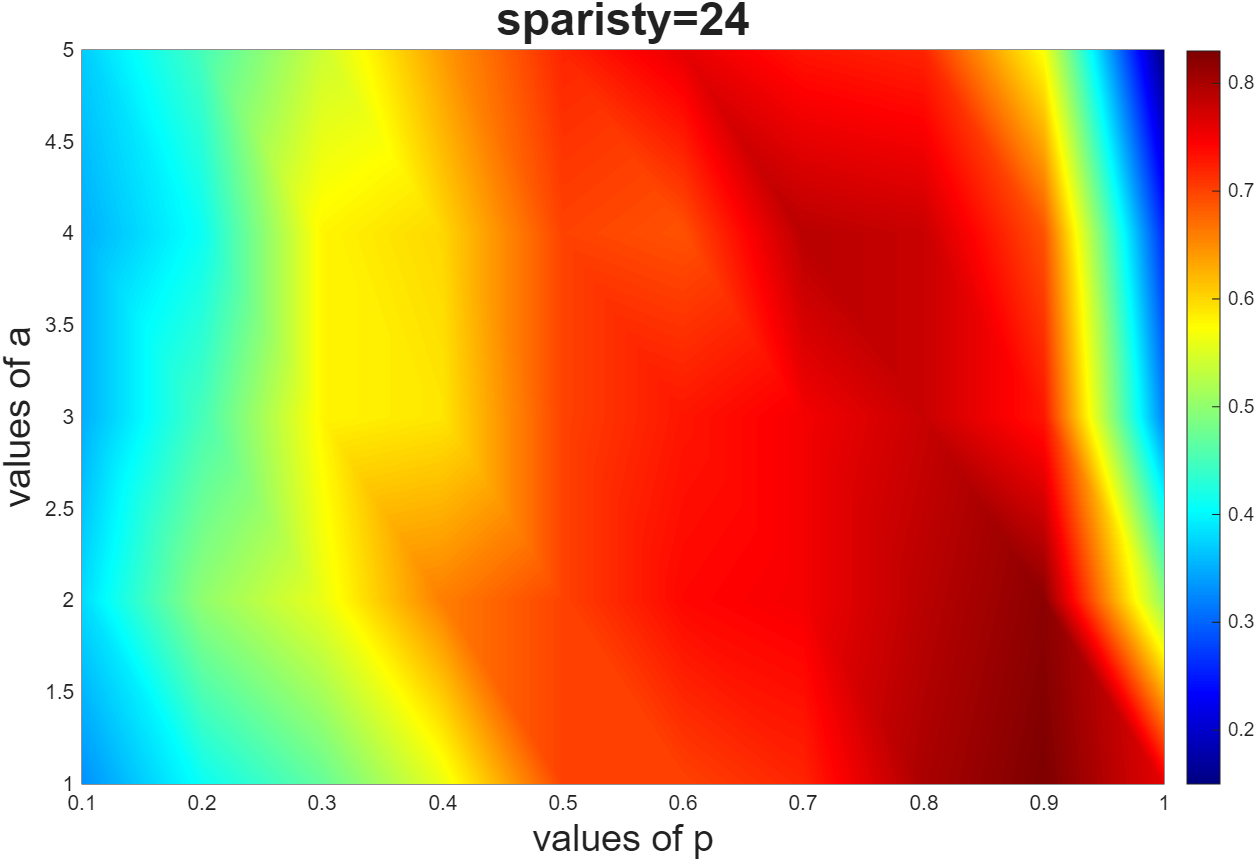}&\quad
  \includegraphics[width=5.9cm]{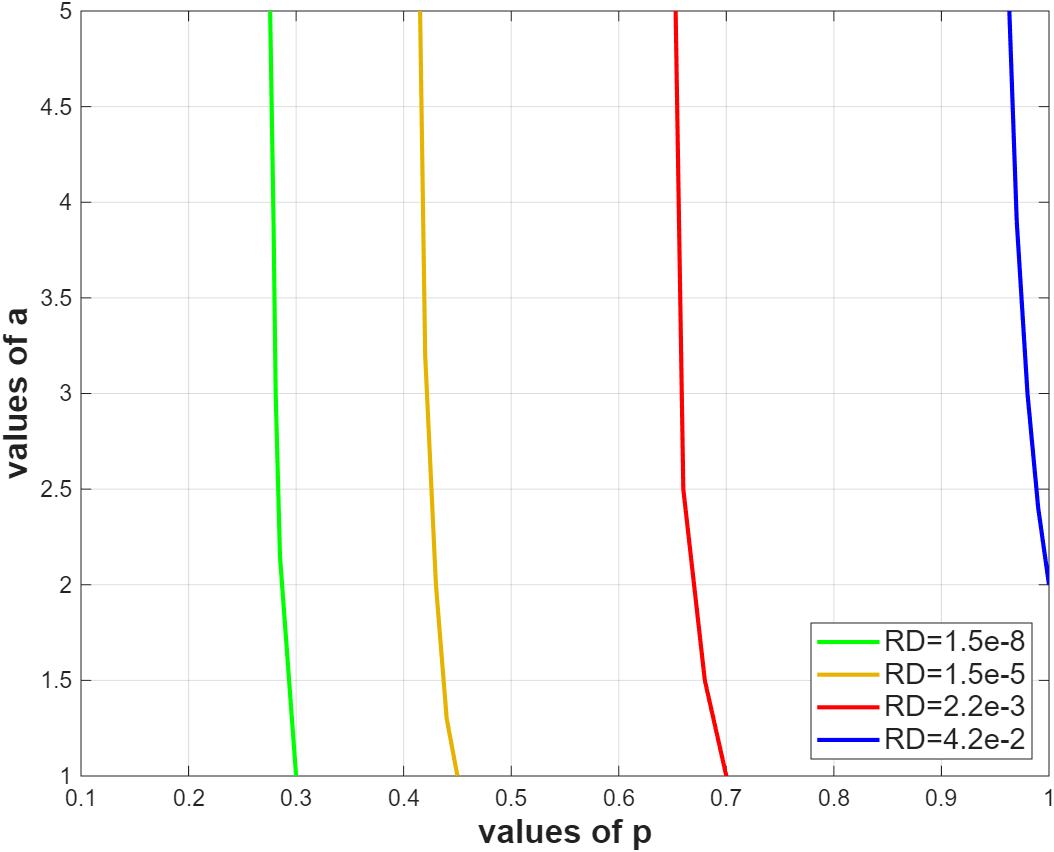}\\
   (a) & \quad (b)
   \end{tabular}
  \caption{(a) Numerical tests by $64\times 256$ Gaussian matrix
  when the sparsity is 24; (b) graphs  of $a$ as a function of $p$ with $N=256$
and various RD$_{P_{a,p}}$}\label{fig-heat}
\end{figure}

{\color{red}
Moreover, RD$_{P_{a,p}}$ supplies a practical reparameterization. If a target RD$_{P_{a,p}}$ value $d$ is given,
letting $r:=\frac{d}{\sqrt{N}}$, by Proposition \ref{prop-RD} we obtain
\begin{equation*}
a=\frac{(N-1)r^p}{1-Nr^p},
\end{equation*}
whenever $Nr^p\in(0,1)$. Consequently, choosing $d$ and then varying $p$ traces a
one-dimensional curve of $(a,p)$ pairs.
Although this does not completely eliminate the tuning burden
because the target RD$_{P_{a,p}}$ value $d$ and
$p$ still have to be selected, it does reduce some burden by restricting
the search of $(a,p)$ pairs to this curve instead of an unrestricted
two-dimensional grid.}

\begin{figure}[ht]
  \centering
  \includegraphics[width=7.2cm]{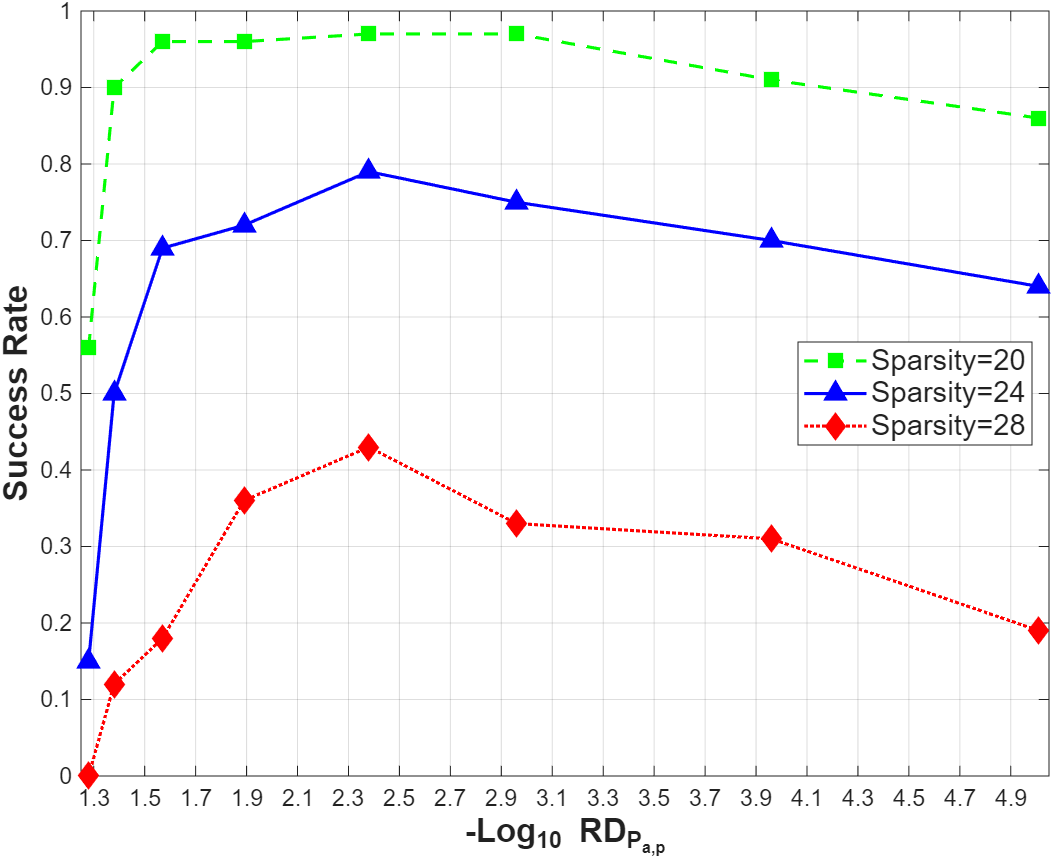}
  \caption{Numerical tests on different RD$_{P_{a,p}}$ by $64\times 256$ Gaussian matrix}\label{fig-RDchoice}
\end{figure}

Based on these analysis,  to further explore the trend of the success rate as RD$_{P_{a,p}}$ changes,
we conduct experiments with various RD$_{P_{a,p}}$
with the sparsity level among $\{20,24,28\}$  and plot the success rates with respect to
the values of negative 10-base logarithm of RD$_{P_{a,p}}$ in Figure \ref{fig-RDchoice}.
From Figure \ref{fig-RDchoice}, we observe that, in all these three cases on sparsity,
the success rate first increases to a peak region and then decreases
as the value of RD$_{P_{a,p}}$ decreases.
{\color{red} This is because the objective
function is much more nonconvex and hence more difficult
to solve its associated minimization problem when RD$_{P_{a,p}}$ is smaller,
while iterations are more likely to stop at a local
minimum far from $\ell_0$ solution
if RD$_{P_{a,p}}$ is too large.}

\subsection{Comparison with Single- and Double-Loop Algorithms for TL1}\label{subsec-exp-compareTL1}

{\color{red}
In general, iterative thresholding (IT) algorithms for nonconvex penalties rely on
the exact computation of their proximity operators.
For classical $\ell_p$ minimization, computing this operator is non-trivial;
related analyses can be found in \cite{CSS16,PSY,SY}.
Zhang and Xin \cite{ZX17} showed that TL1 admits a closed-form proximity
operator and developed the single-loop TL1IT method. They also studied the
double-loop DCATL1 method in \cite{ZX18}. Since $P_{a,p}$ when
$p=1$ reduces to TL1, this special case permits
a comparison of IRLSTL1 with both a single-loop
alternative and another double-loop method for the same penalty.
We use $64\times256$ Gaussian matrices with
$k\in\{14,16,\ldots,32\}$ and $64\times1024$ Gaussian matrices with
$k\in\{6,8,\ldots,24\}$. Figure \ref{fig-compare} reports the
recovery success rates across different sparsity levels,
while Table \ref{table-time1} presents the runtime of the
corresponding experiments.
}

\begin{figure}[ht]
  \centering
  \begin{tabular}{cccc}
  \includegraphics[width=3.3cm]{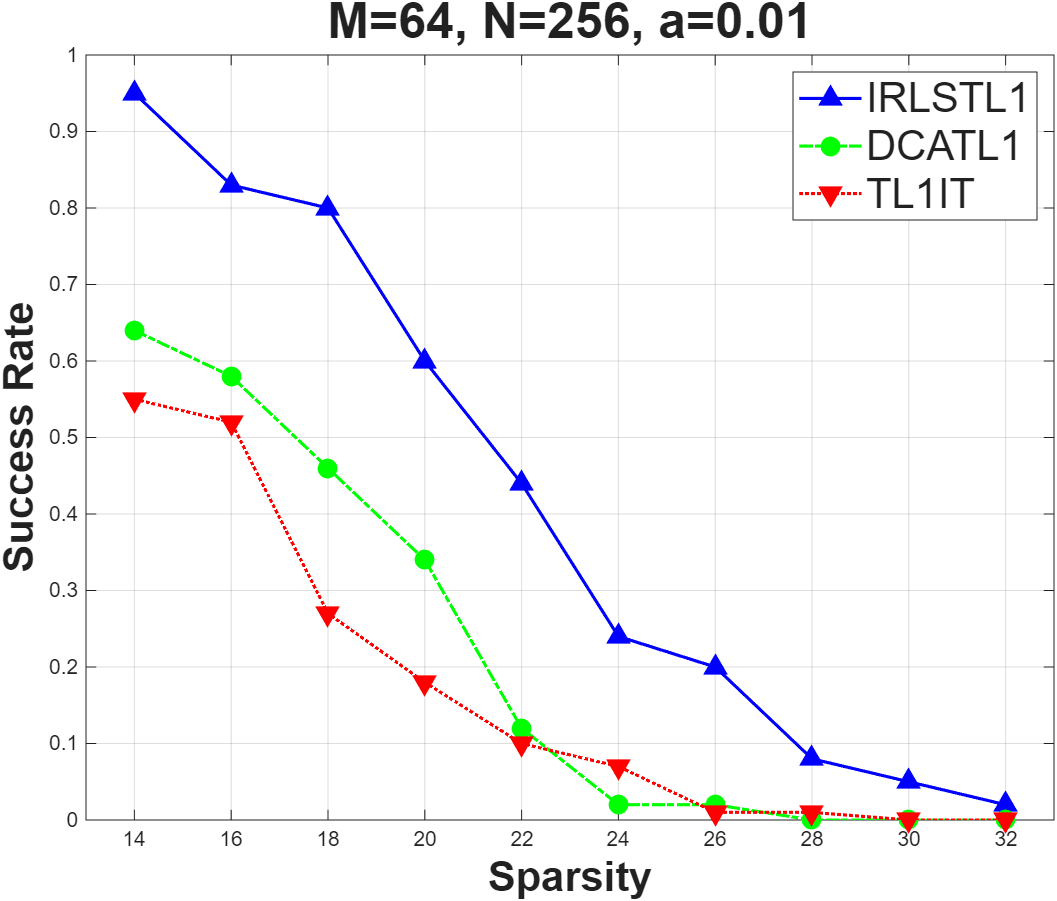}&
  \includegraphics[width=3.3cm]{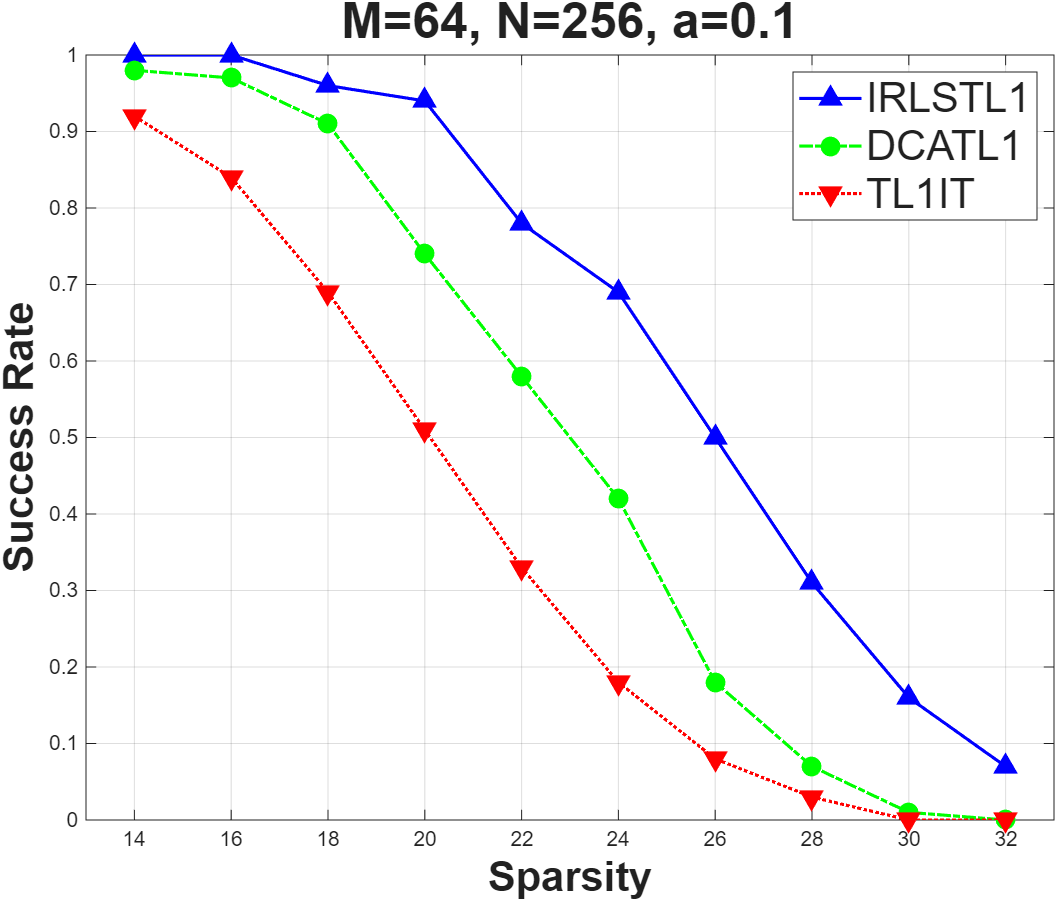}&
    \includegraphics[width=3.3cm]{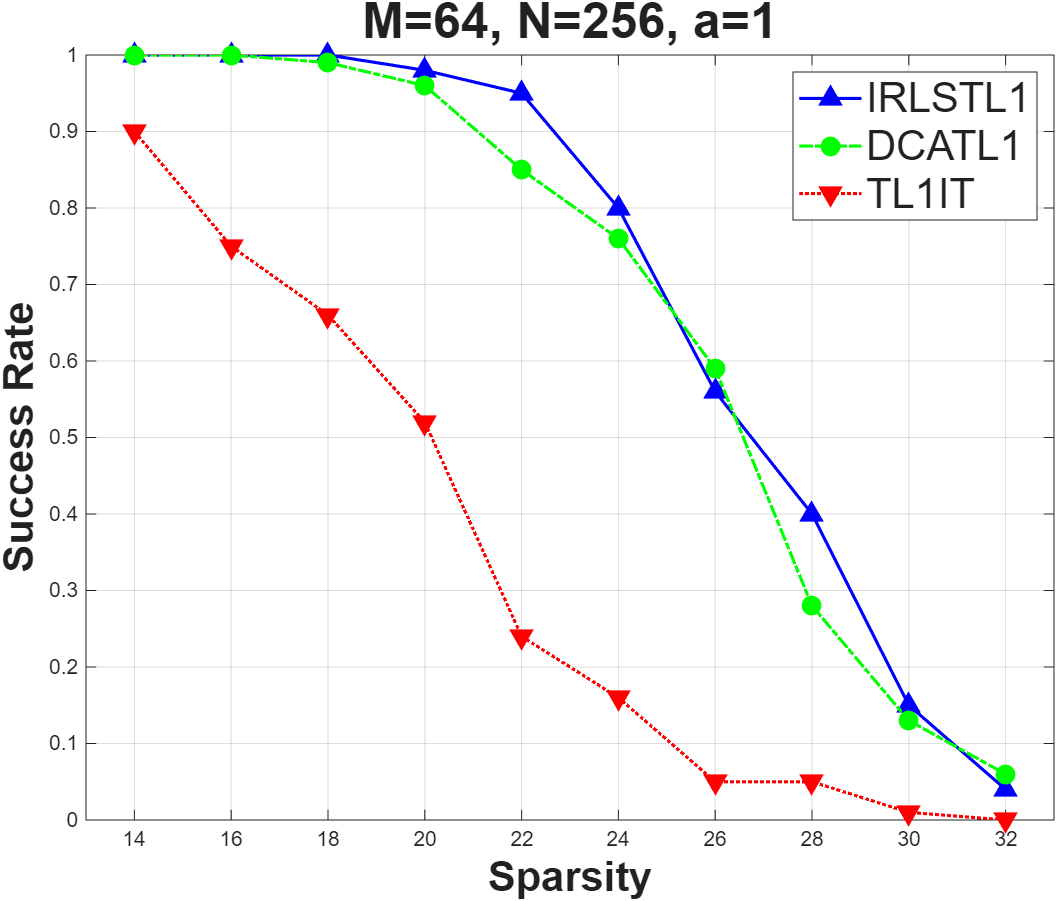}&
    \includegraphics[width=3.3cm]{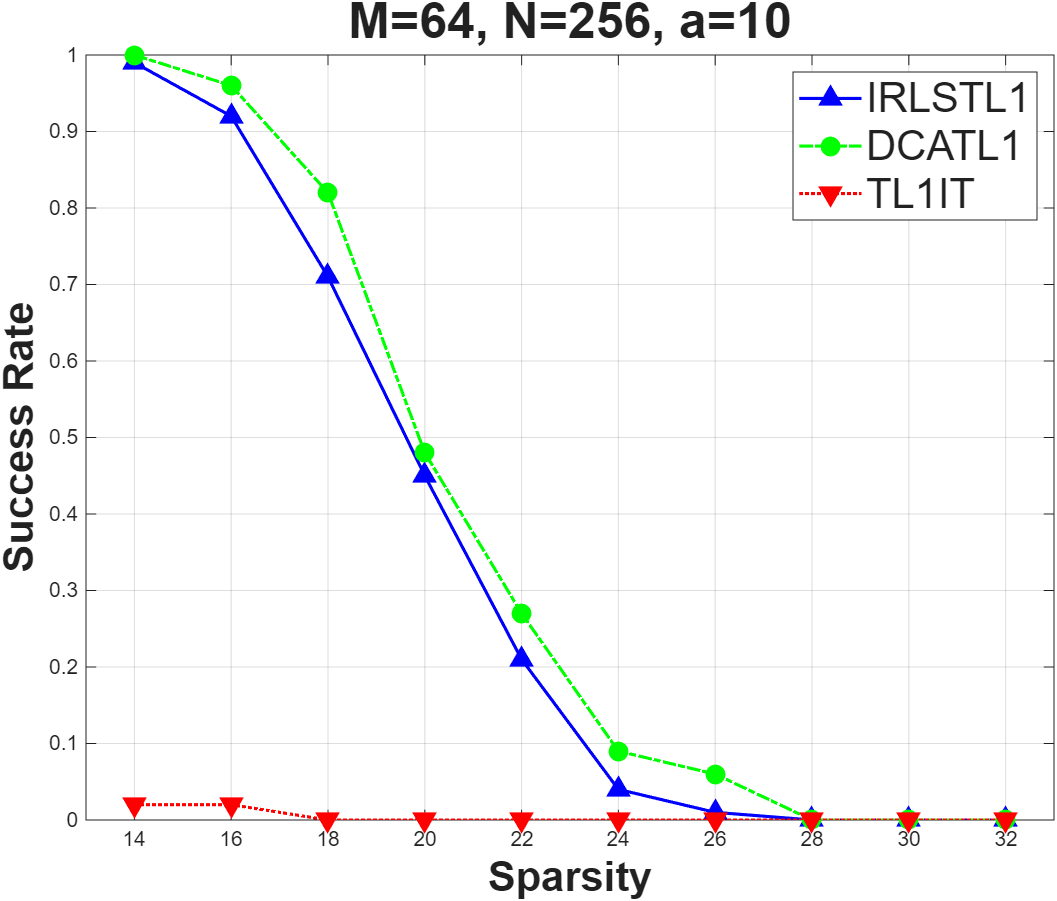}\\
  (a) & (b) & (c) & (d)  \\
  \vspace{0.1cm}
    \includegraphics[width=3.3cm]{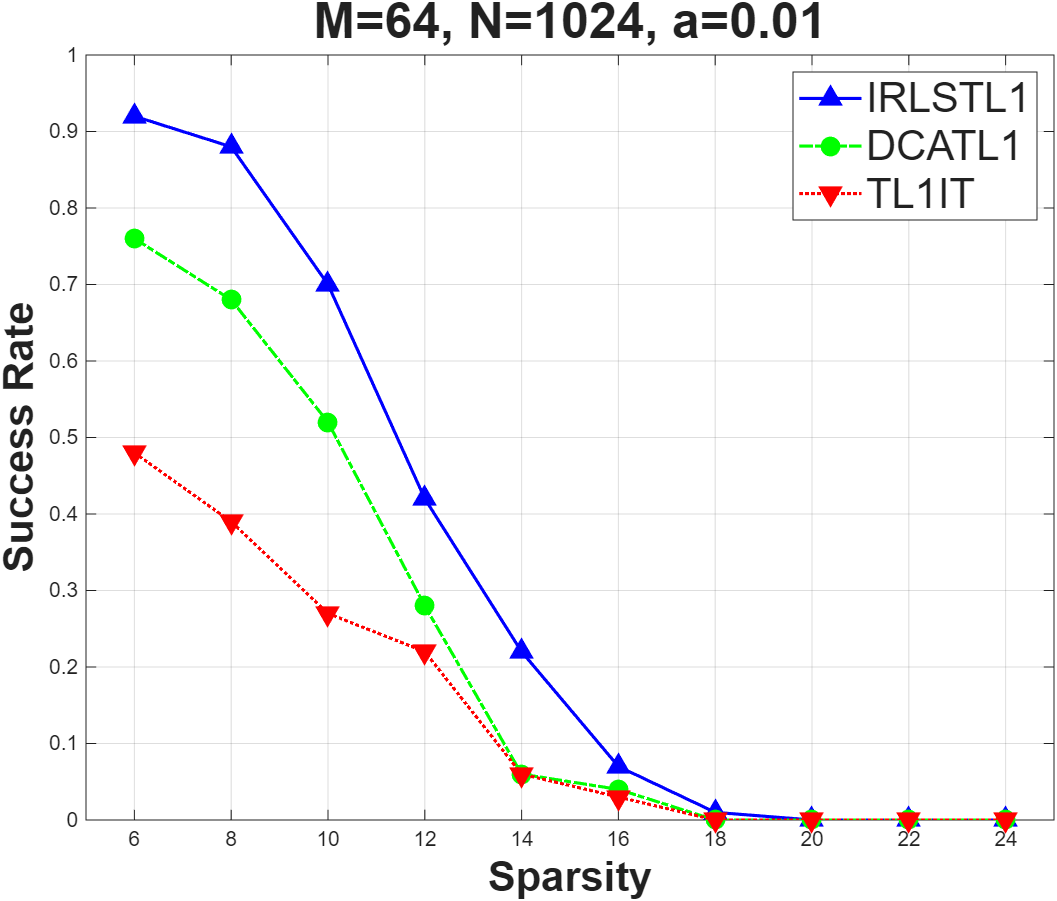}&
    \includegraphics[width=3.3cm]{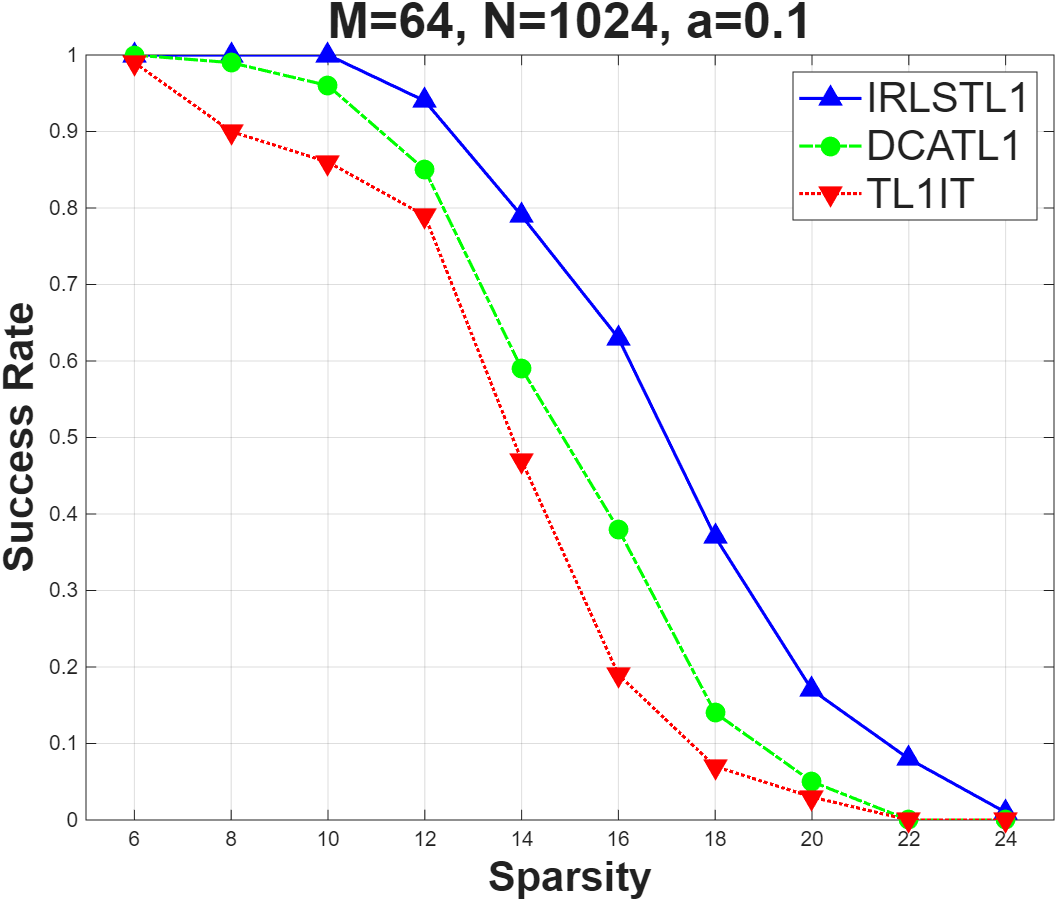}&
    \includegraphics[width=3.3cm]{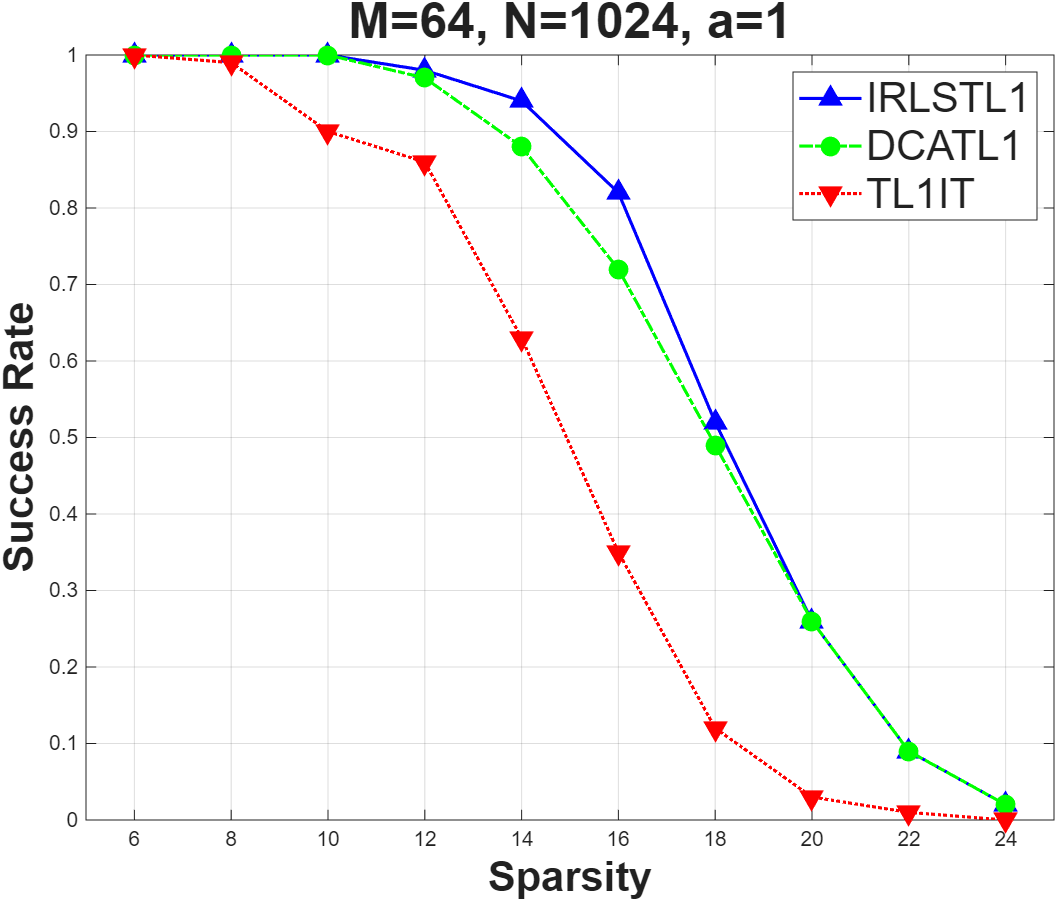}&
    \includegraphics[width=3.3cm]{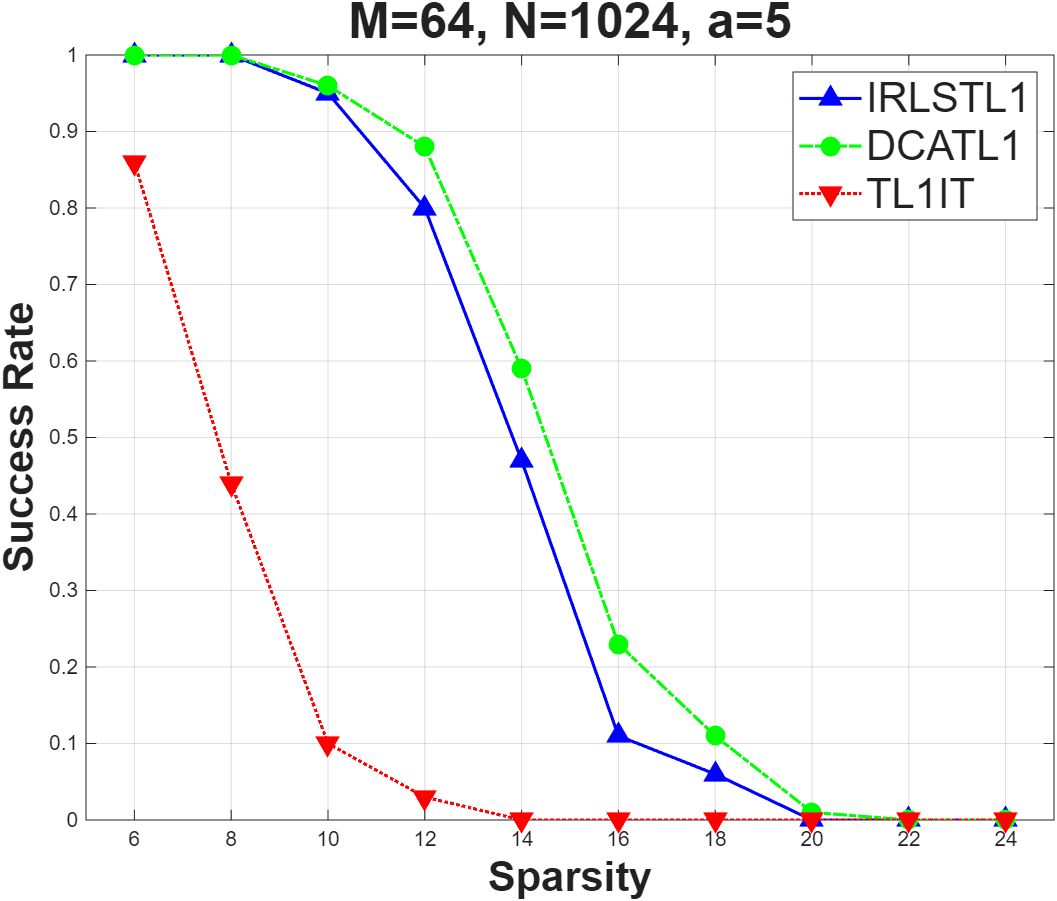}\\
   (e) & (f) & (g) & (h)
        \end{tabular}
\caption{Recovery comparison of IRLSTL1, the double-loop DCATL1, and the
single-loop TL1IT on Gaussian sensing problems}\label{fig-compare}
\end{figure}

{\color[rgb]{1.00,0.00,0.00}
Table \ref{table-time1} confirms the expected computational advantage of a
single-loop scheme: TL1IT is faster than IRLSTL1 in six of the eight panels.
IRLSTL1 is faster in panels (a) and (e), so loop count alone cannot determine
wall-clock time. Comparing the two double-loop methods, DCATL1 takes between
$1.83$ and $9.36$ times the IRLSTL1 runtime
across the eight panels. Thus, the proposed nested
implementation IRLSTL1 is consistently
faster than the tested DCATL1 implementation, but it is
faster only in two of the eight panels than
the single-loop TL1IT baseline.

\begin{table}[htbp]
\centering
\renewcommand{\arraystretch}{1.1}
\resizebox{0.8\textwidth}{!}{
  \begin{tabular}{ccccccccc}
  \hline
Method & (a) &(b)&(c)&(d)& (e)&(f)&(g)&(h)\\
  \hline
 IRLSTL1 & 30.34 &18.87& 26.53&23.71 & 73.69 & 75.10 & 124.51&145.34\\
    \hline
DCATL1 & 76.96 & 61.58 & 48.58 & 69.28 & 689.95 & 400.58 &322.01&396.62\\
    \hline
TL1IT & 39.00 & 15.31& 20.68 & 17.94& 128.44 &57.48& 110.73&71.93\\
  \hline
  \end{tabular}}
  \caption{Total runtime (in seconds) per method for experiments presented in Figure \ref{fig-compare}}
  \label{table-time1}
 \end{table}

Figure \ref{fig-compare} supplies the complementary recovery comparison.
TL1IT gives the consistently lowest success rates in the displayed settings, while IRLSTL1 and
DCATL1 are closer when $a$ is relatively large. For
choices $a\in\{0.01,0.1\}$, IRLSTL1 yields higher success rates than DCATL1 in
the corresponding tests. To sum up, Figure \ref{fig-compare}
and Table \ref{table-time1} show a trade-off between
accuracy and runtime: TL1IT often reduces runtime, while IRLSTL1 is
more accurate in these tests and is less costly than DCATL1.
}

\subsection{Comparison of $P_{a,p}$ with Various Penalties}\label{subsec-exp-compare}

In this subsection, we use two classes of random matrices to
compare the performance of the IRLSTLp
with  DCATL1, DCA  $\ell_1-\ell_2$,
and IRucLq: Gaussian random matrices in Subsection \ref{s4.3.1} and
over-sampled DCT matrices in Subsection \ref{s4.3.2}.
Empirically, we set the parameter $q = 0.5$ for the IRucLq method.

\subsubsection{Tests by Gaussian Random Matrices}\label{s4.3.1}

We use the class of Gaussian matrices generated by the multi-variable normal distribution $\mathcal{N}(0,\Sigma)$ to test the four algorithms above,
where the covariance matrix
$\Sigma:=\{(1-r)\mathbf{1}_{\{i=j\}}+r\}_{i,j}$
with $r\in[0,1)$. Generally speaking, it will be more difficult
to recover the true sparse signal as $r$ gets larger.

We test  by $64\times 256$ Gaussian matrices with
$r\in\{0,0.4,0.8\}$ and plot their success rates  in Figure \ref{fig-Gauss}.
The parameter of DCATL1 is set to $a=1$ and
the parameters of IRLSTLp are set to $a=1$, $p=0.8$, and $\kappa=3$.
Figure \ref{fig-Gauss} shows that all these four algorithms
are hardly affected by the value of $r$ and,
in each case, the  IRLSTLp, the DCATL1, and the IRucLq are comparable
while the DCA of $\ell_1-\ell_2$ has the lowest success rate.

\begin{figure}[ht]
  \centering
 \includegraphics[width=4.5cm]{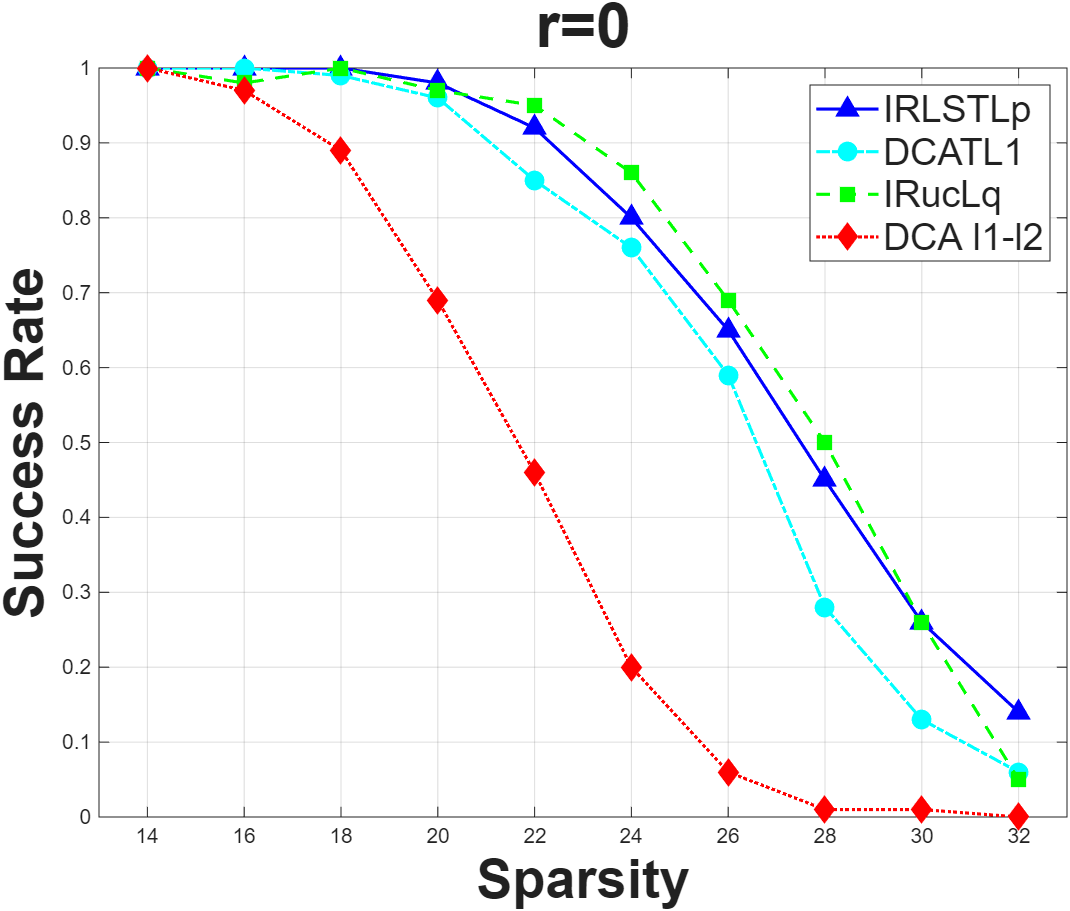}
  \quad
  \includegraphics[width=4.5cm]{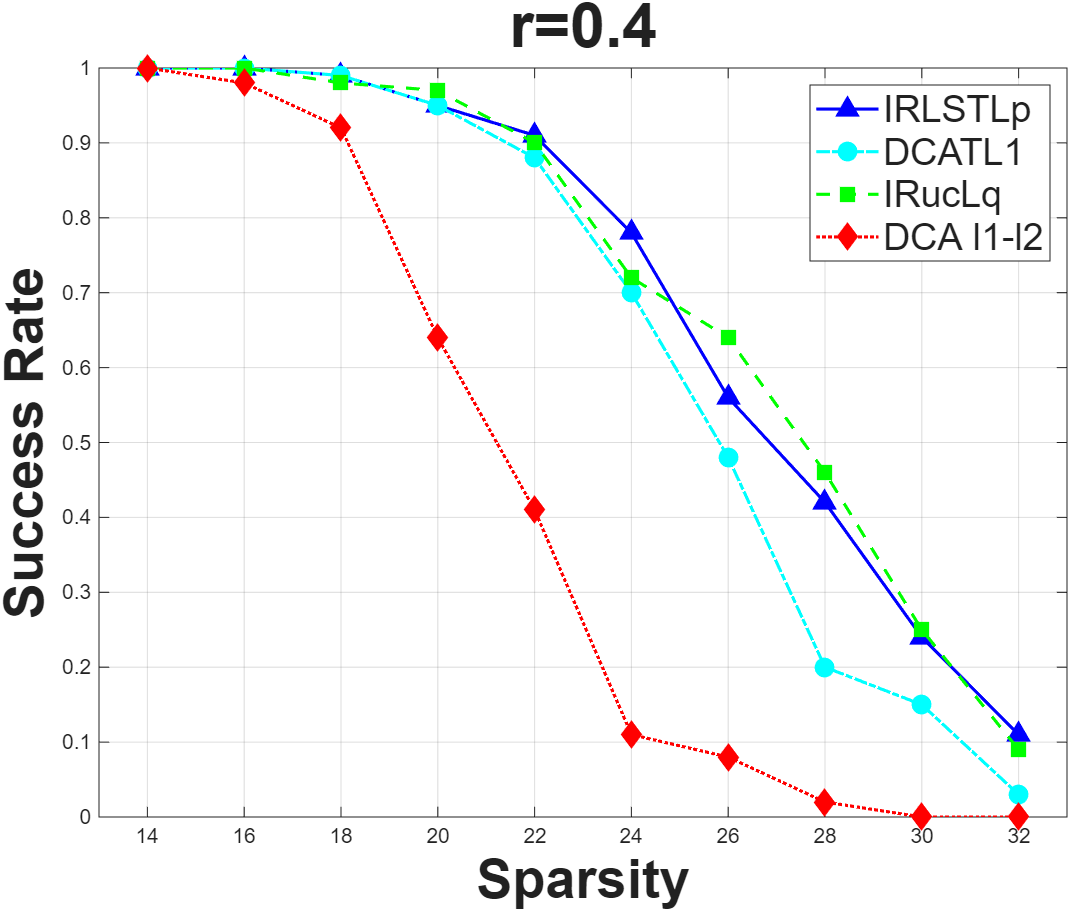}
    \quad
  \includegraphics[width=4.5cm]{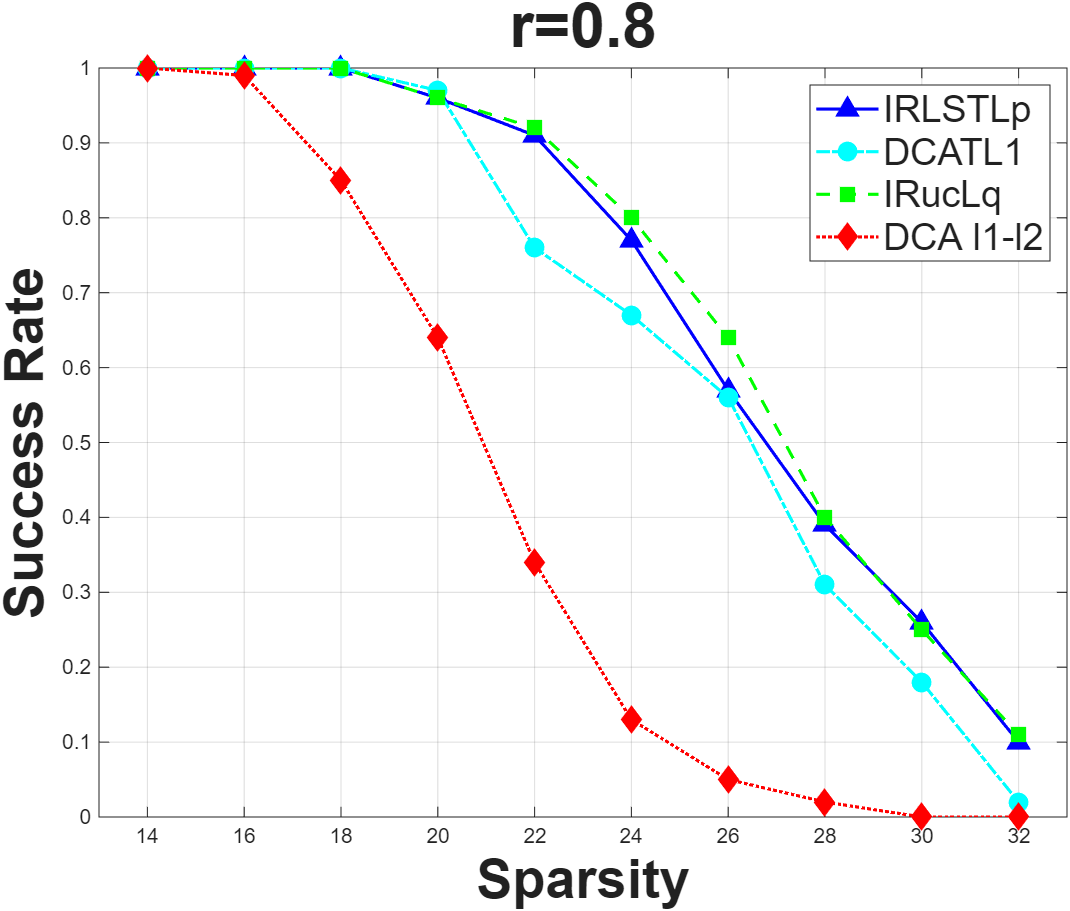}
  \caption{Numerical tests by $64\times 256$ Gaussian matrix with different $r$}\label{fig-Gauss}
\end{figure}

\begin{table}[H]
\centering
\resizebox{0.95\textwidth}{!}{
\renewcommand{\arraystretch}{1.2}
\begin{tabular}{c|ccc|ccc|ccc|ccc}
\hline
Method &\multicolumn{3}{c|}{IRLSTLp} & \multicolumn{3}{c|}{DCATL1}
& \multicolumn{3}{c|}{IRucLq} & \multicolumn{3}{c}{DCA $\ell_1-\ell_2$}\\
\hline
$r$ & 0 & 0.4 & 0.8 & 0 & 0.4 & 0.8 & 0 & 0.4 & 0.8 & 0 & 0.4 & 0.8 \\
\hline
Time & 46.81 & 56.34& 61.80 & 40.68& 42.92 & 50.33 & 8.43 & 9.74& 9.98& 34.82& 33.78& 30.76\\
\hline
\end{tabular}}
\caption{Total runtime (in seconds) per method for experiments presented in Figure \ref{fig-Gauss}}
\label{table-time-r}
\end{table}

{\color{red} In addition, Table \ref{table-time-r}  reports the
runtime of each method in Figure \ref{fig-Gauss}.
As $r$ changes from $0$ to $0.8$,
the IRLSTLp runtime increases from $46.81$ to $61.80$ seconds, while the
IRucLq runtime remains below $10$ seconds and is the smallest in all three
settings. As demonstrated by the results in Figure \ref{fig-Gauss} and Table \ref{table-time-r},
IRLSTLp achieves recovery performance comparable to that of the
competing methods in this Gaussian experiment, although its runtime is not competitive.
}

\subsubsection{Tests by Over-Sampled DCT Matrices}\label{s4.3.2}

We  use the class of over sampled DCT matrices $A=(a_1,\ldots,a_N)\in\mathbb{R}^{M\times N}$
to compare the performance of these four  algorithms under varying degrees of matrix coherence,
where, for any $i\in\{1,\ldots,N\}$,
$$
a_i:=\frac{1}{\sqrt{M}}\cos\left(\frac{2\pi[i-1]\omega}{F}\right),
$$
$\omega$ is a random vector uniformly from $(0,1)^M$, and $F>0$ is the frequency parameter.
A known result is that the DCA of $\ell_1-\ell_2$ has a nice sparse recovery performance
when the sensing matrix is highly coherent; see, for example, \cite{LYHX15}.
A key property of these over-sampled DCT matrices is precisely their high coherence.
For instance, the coherence of a $100\times 1000$ over-sampled DCT matrix with $F=20$ is approximately $0.9999$.
We refer to \cite{FL12} for more descriptions of such over-sampled DCT matrices.

{\color{red}
First, we use $100\times 1500$ over-sampled DCT matrices with $F\in\{2,4,6,8,10,16\}$ to test
these four algorithms.
We fix the parameter $a=5$ for both DCATL1 and IRLSTLp.
For IRLSTLp, we evaluate the coarse candidate set
$p\in\{0.8,0.9,0.95,0.99\}$ and display the
best result at any value of $F$.
Their success rates at the sparsity $k\in\{14,16,\ldots,32\}$ are presented in Figure \ref{fig-DCT}
and the total runtime per method is recorded in Table \ref{table-time-F}.
This design is intended to measure the adaptability of IRLSTLp
available from its parameters, not the performance of a universal
default setting. The parameter-free DCA $\ell_1-\ell_2$, IRucLq with $q=0.5$,
and DCATL1 with $a=5$ require almost no tuning in this comparison.
Figure \ref{fig-DCT} shows that they are all sensitive to the
coherence of the measurement matrix.
To be precise, DCATL1 and
DCA $\ell_1-\ell_2$ have lower success in the low-coherence cases displayed,
whereas IRucLq almost fails to recover the sparse solution in high-coherence cases.
The selected IRLSTLp achieves improved recovery performance
compared with DCATL1 and DCA $\ell_1-\ell_2$ for low-coherence matrices and
compared with IRucLq for high-coherence matrices,
despite incurring some tuning burden.}

\begin{figure}[ht]
  \centering
  \includegraphics[width=4.5cm]{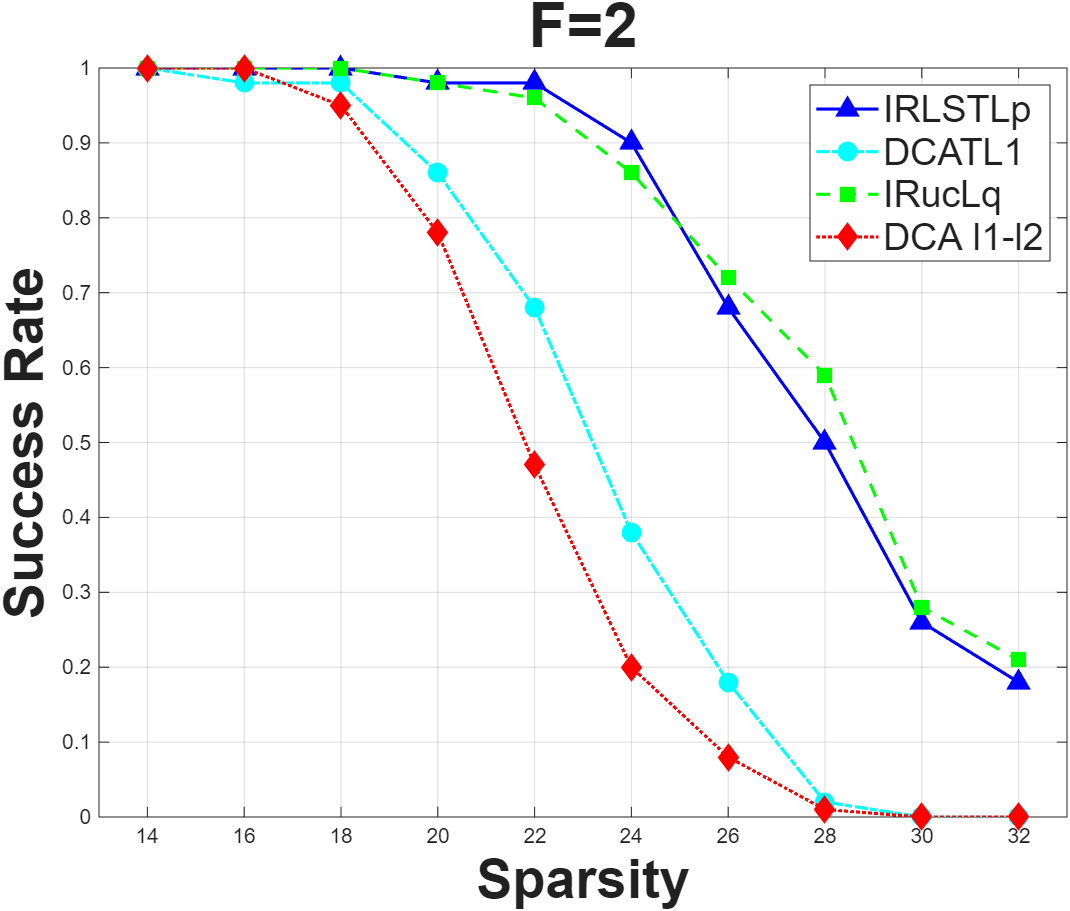}
  \quad
  \includegraphics[width=4.5cm]{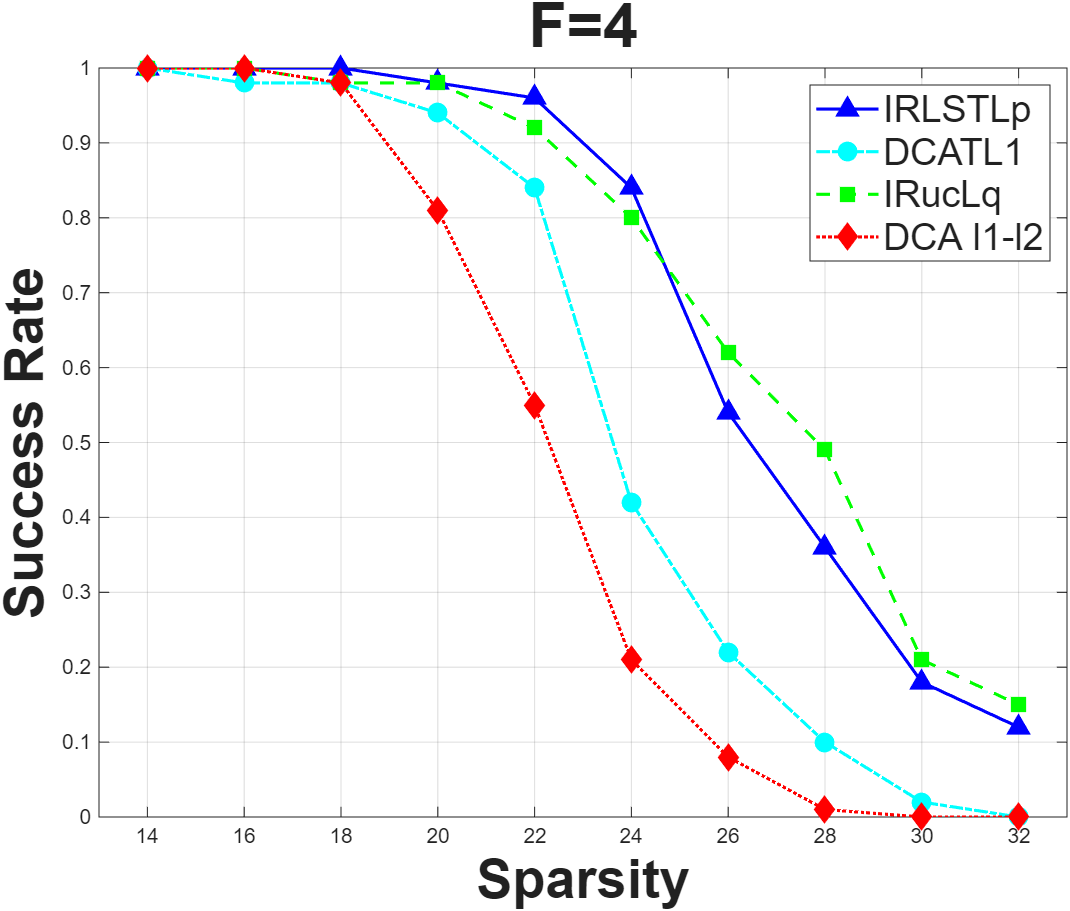}
  \quad
  \includegraphics[width=4.5cm]{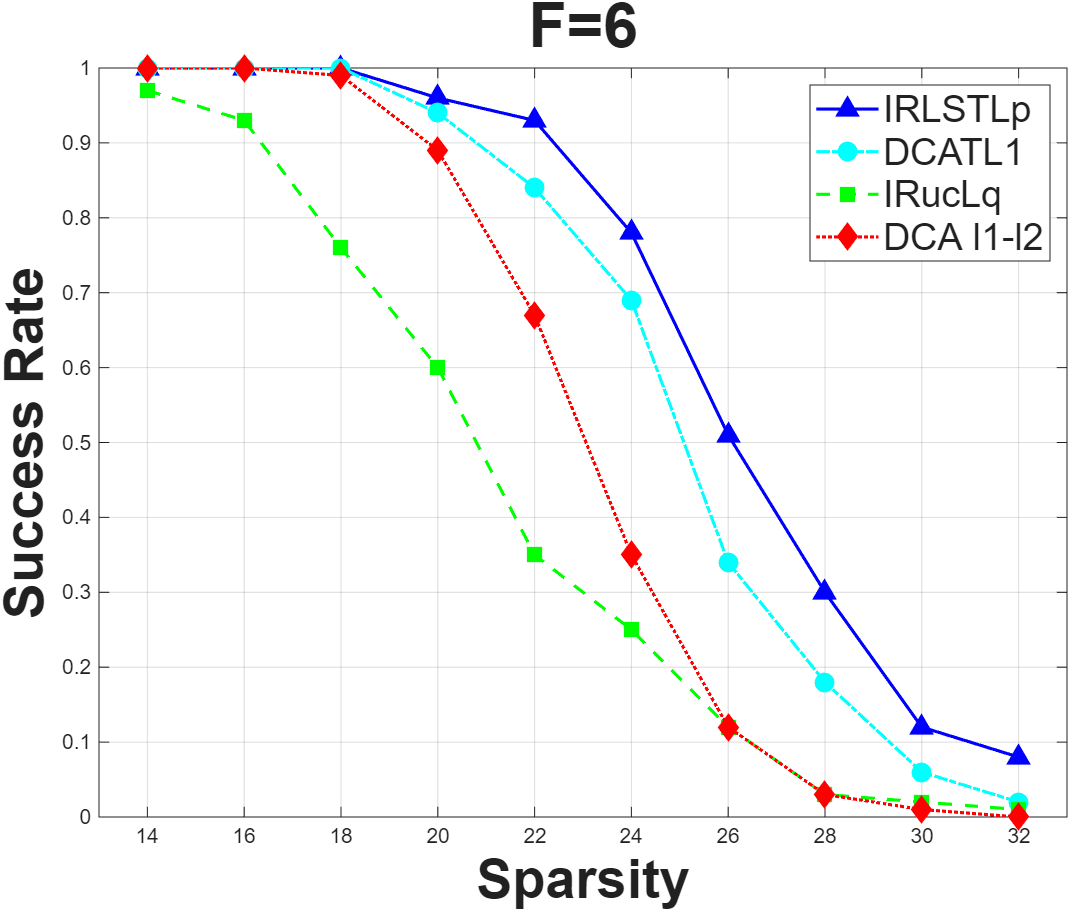}\\
  \vspace{0.5cm}
  \includegraphics[width=4.5cm]{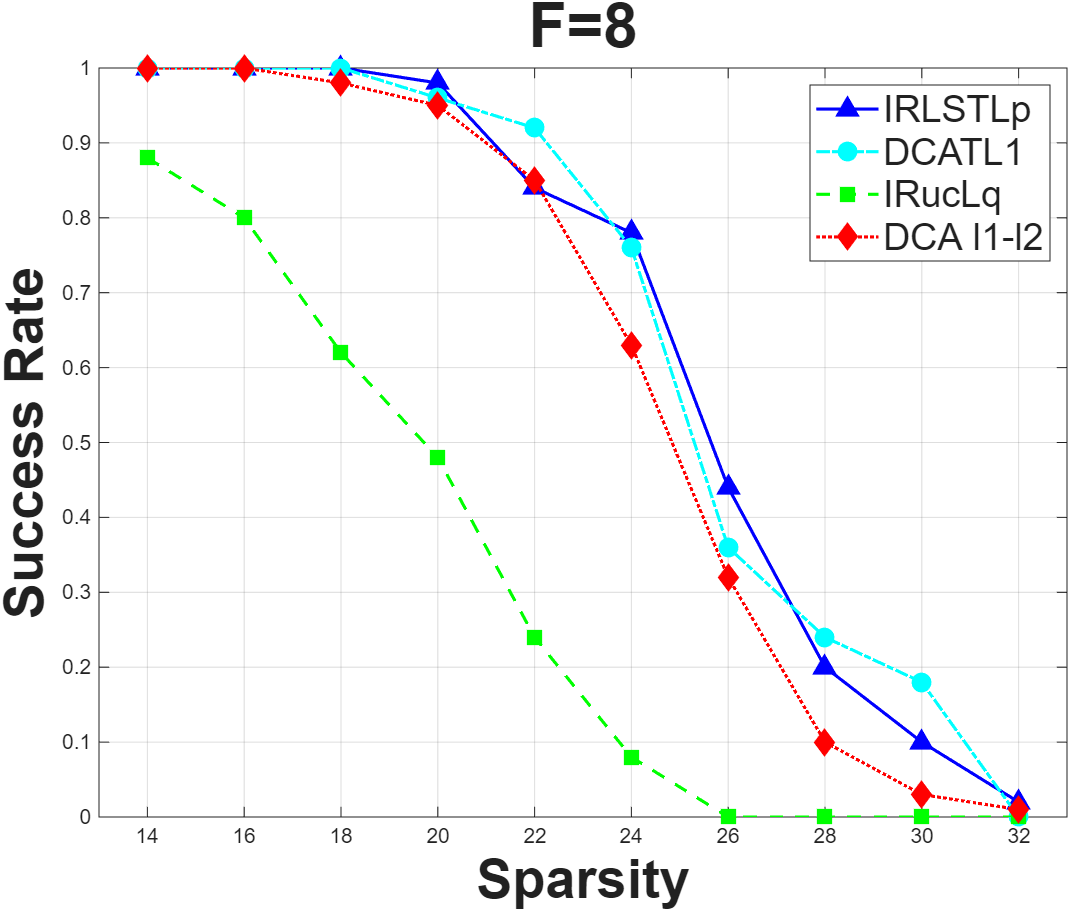}
 \quad
  \includegraphics[width=4.5cm]{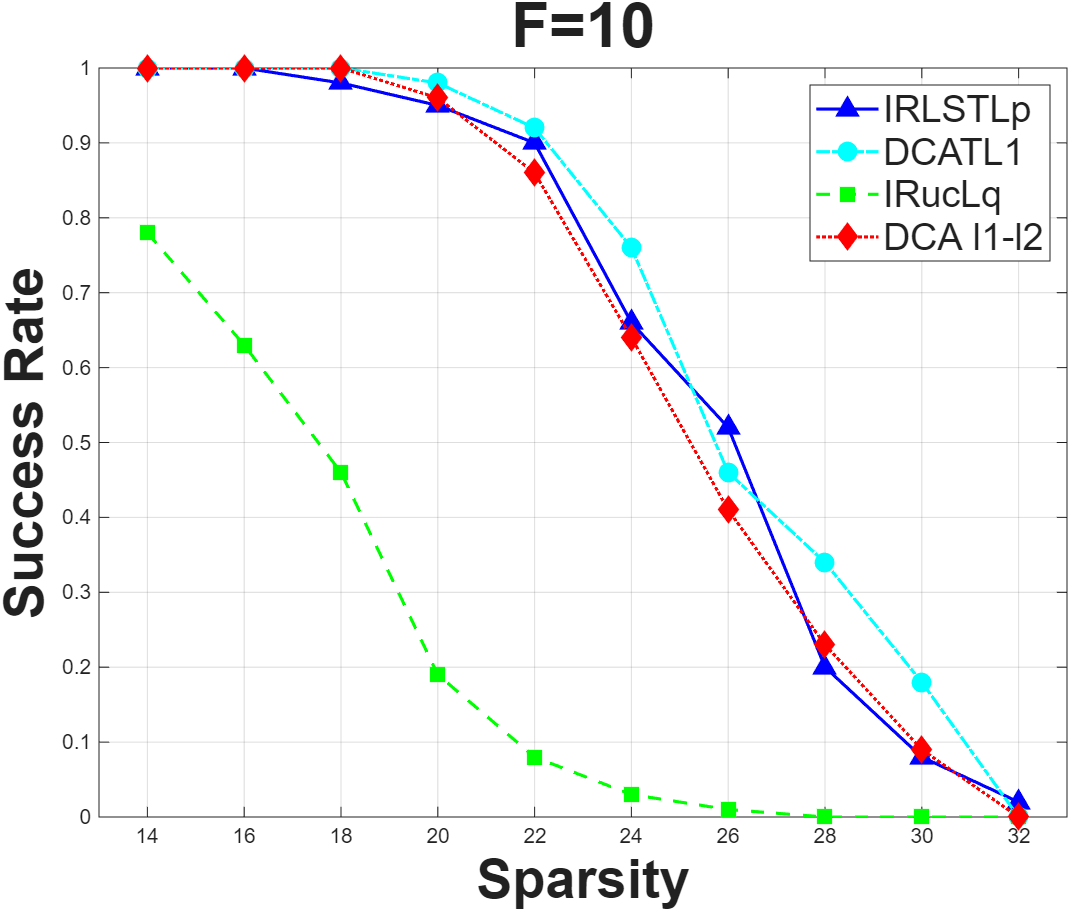}
   \quad
  \includegraphics[width=4.5cm]{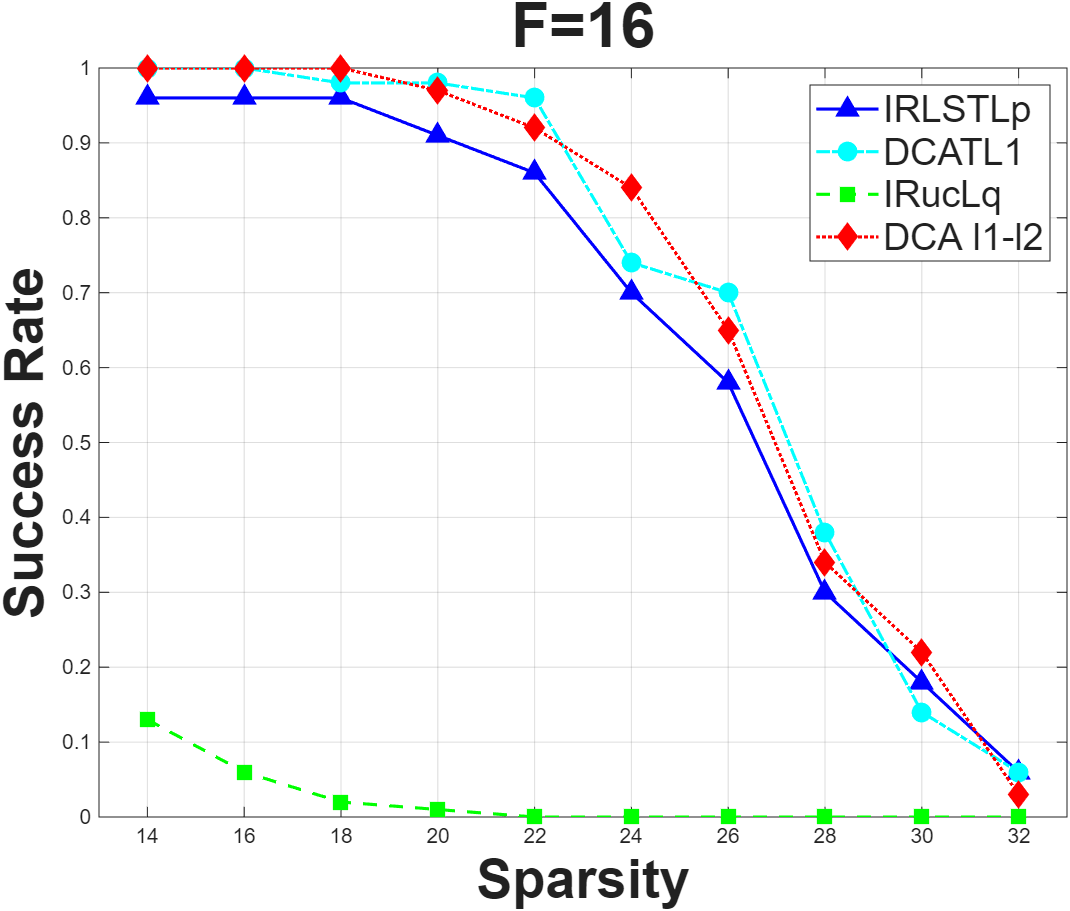}
\caption{Numerical tests by  $100\times 1500$ over-sampled DCT matrix with different $F$ }\label{fig-DCT}
\end{figure}

\begin{table}[h]
\centering
\resizebox{0.7\textwidth}{!}{
\renewcommand{\arraystretch}{1}
\begin{tabular}{ccccccc}
\hline
$F$& 2 & 4 & 6 &8&10&16\\
\hline
IRLSTLp & 348.74 & 596.30 & 710.82 & 1141.81 & 1150.68 & 876.90 \\
\hline
DCATL1 & 951.21 & 910.88& 1142.04 & 808.58& 793.94 &742.11 \\
\hline
 IRucLq & 77.04 & 86.52& 142.49 & 160.43& 177.81&243.44 \\
\hline
DCA $\ell_1-\ell_2$ & 525.94 & 545.11& 497.62& 490.41& 481.27&404.46 \\
\hline
\end{tabular}}
\caption{Total runtime (in seconds) per method for experiments presented in Figure \ref{fig-DCT}}
\label{table-time-F}
\end{table}

{\color{red}
As shown in Table \ref{table-time-F},  IRLSTLp
is not the fastest method in those  high-coherence cases while IRucLq is approximately
from $3$ to $7$ times faster. Compared with DCATL1, IRLSTLp is faster when
$F\in\{2,4,6\}$ and slower when $F\in\{8,10,16\}$. Compared with DCA
$\ell_1-\ell_2$, IRLSTLp is faster only at $F=2$.
This shows that, although IRLSTLp offers the advantage of adaptability,
it indeed incurs a higher computational cost especially in high-coherence cases.
}

\begin{figure}[H]
  \centering
  \includegraphics[width=11cm]{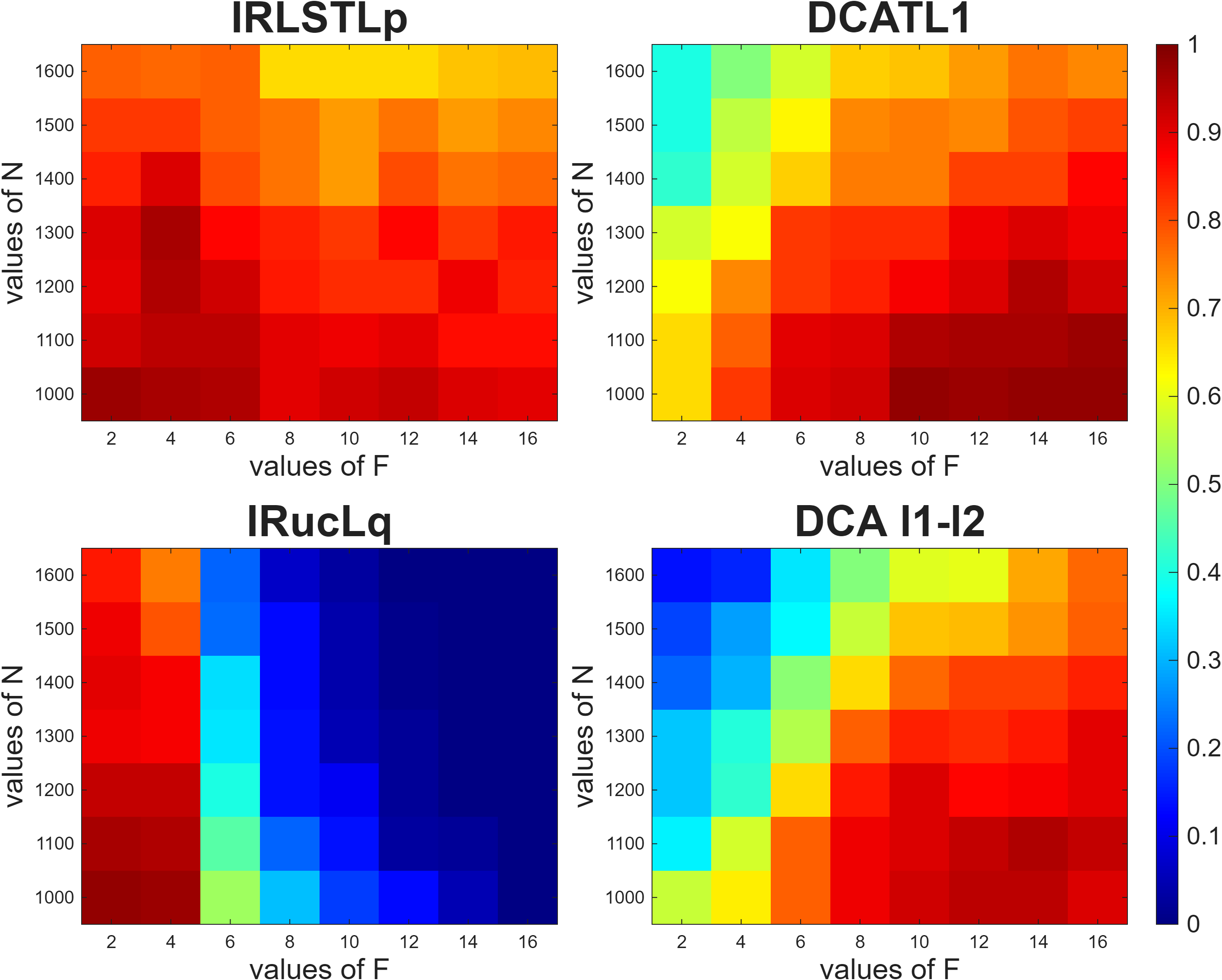}
\caption{Success rates at the sparsity $k=24$ for IRLSTLp,
DCATL1, IRucLq, and DCA $\ell_1-\ell_2$
  by over-sampled DCT matrices with $M=100$, $N\in\{1000,1100,\ldots,1600\}$,
  and $F\in\{2,4,\ldots,16\}$}\label{fig-DCTheat}
\end{figure}

{\color{red}Moreover, to provide a clear comparison
of recovery performance in different environments,
Figure \ref{fig-DCTheat} evaluates the performance
over different environment parameters.
We fix $M=100$, $k=24$, and $a=5$, and vary
the dimension $N \in \{1000, 1100, \dots, 1600\}$ and the factor $F \in \{2, 4, \dots, 16\}$.
The parameter $p$ of IRLSTLp is selected from $\{0.8,0.9,0.95,0.99\}$
(the same as above), and only the highest success rate is recorded.
The phase image shows that the IRLSTLp attains success rates
almost in $[0.7,1]$ over different environment parameters.
This further supports the claim that the two-parameter
structure translates into the adaptability across  environments
and into concrete performance gains over the other three methods
in their respective weak-performance regimes.}

\begin{figure}[H]
  \centering
  \includegraphics[width=10cm]{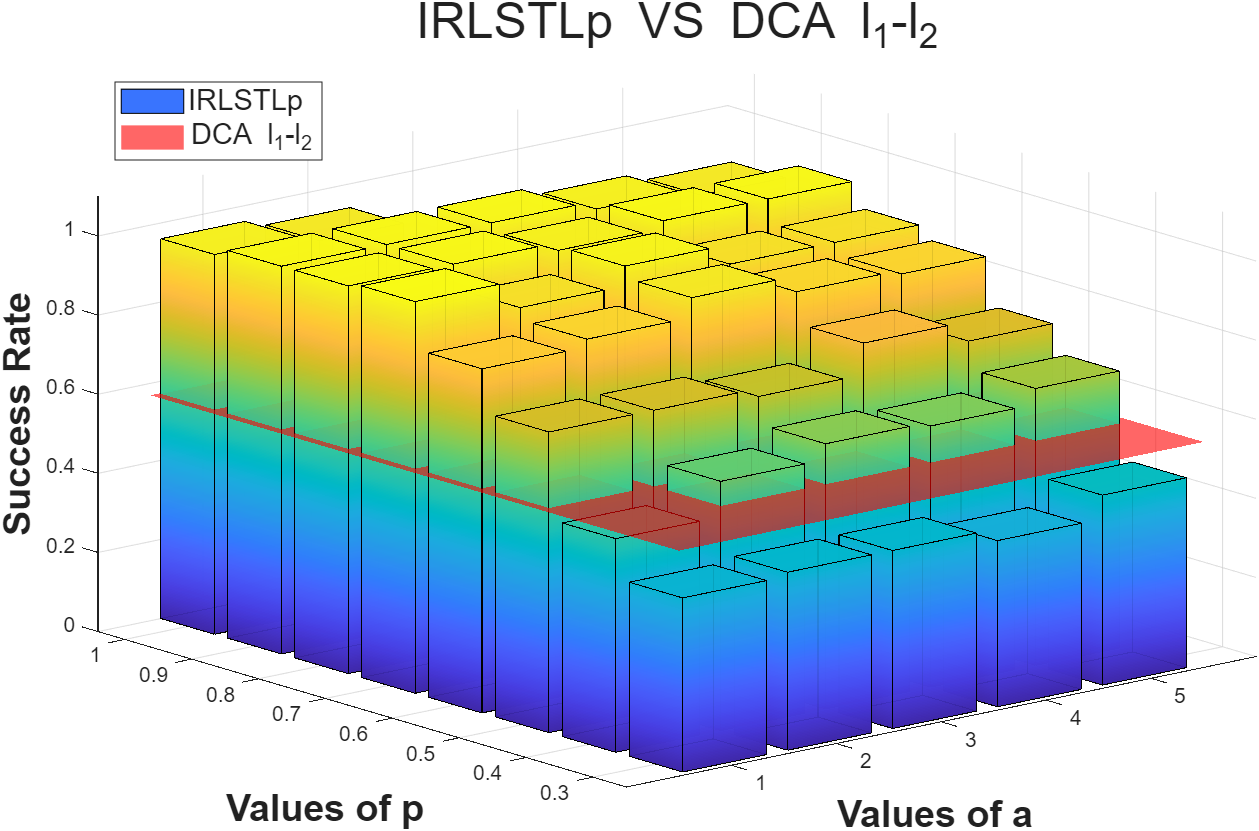}
 \caption{Comparison of success rates between IRLSTLp across a grid of
  $a \in \{1, 2, \dots, 5\}$ and $p \in \{0.3, 0.4, \dots, 1\}$
  and the baseline DCA $\ell_1-\ell_2$, using $100 \times 1000$
  over-sampled DCT matrices at $F=2$.
  The baseline achieves a success rate of $0.57$.}\label{fig-VS3D}
\end{figure}

{\color{red}In Subsection \ref{subsec-exp-a,p},
we have already demonstrated that RD$_{P_{a,p}}$ can help reduce the tuning burden.
Here, we further address the concern regarding this burden through a sensitivity analysis
on $100\times1000$ over-sampled DCT matrices at $F=2$,
with the parameter-free DCA $\ell_1$-$\ell_2$ as the baseline.
As shown in Figure \ref{fig-VS3D},
the DCA $\ell_1-\ell_2$ baseline has a success rate $0.57$,
indicated by the horizontal plane. IRLSTLp exceeds this value over most of the displayed grid.
More importantly for tuning, a broad plateau is visible: when
$p\in[0.7,0.9]$, the success rate is close to one for any tested
$a\in\{1,2,\ldots,5\}$. Thus, in this particular setting,
a coarse $(a,p)$ with $p$ in that interval
can yield improved performance compared to DCA $\ell_1$-$\ell_2$.
}

\subsection{Empirical Convergence Analysis}\label{subsec-emconverg}

{\color{red}
As discussed in Remark \ref{limitations}, the convergence theorem for the idealized
modified IRLS update does not cover Algorithm \ref{alg2}. We therefore examine
the numerical behavior of the implemented outer iteration.
}

\begin{figure}[h]
  \centering
  \includegraphics[width=8cm]{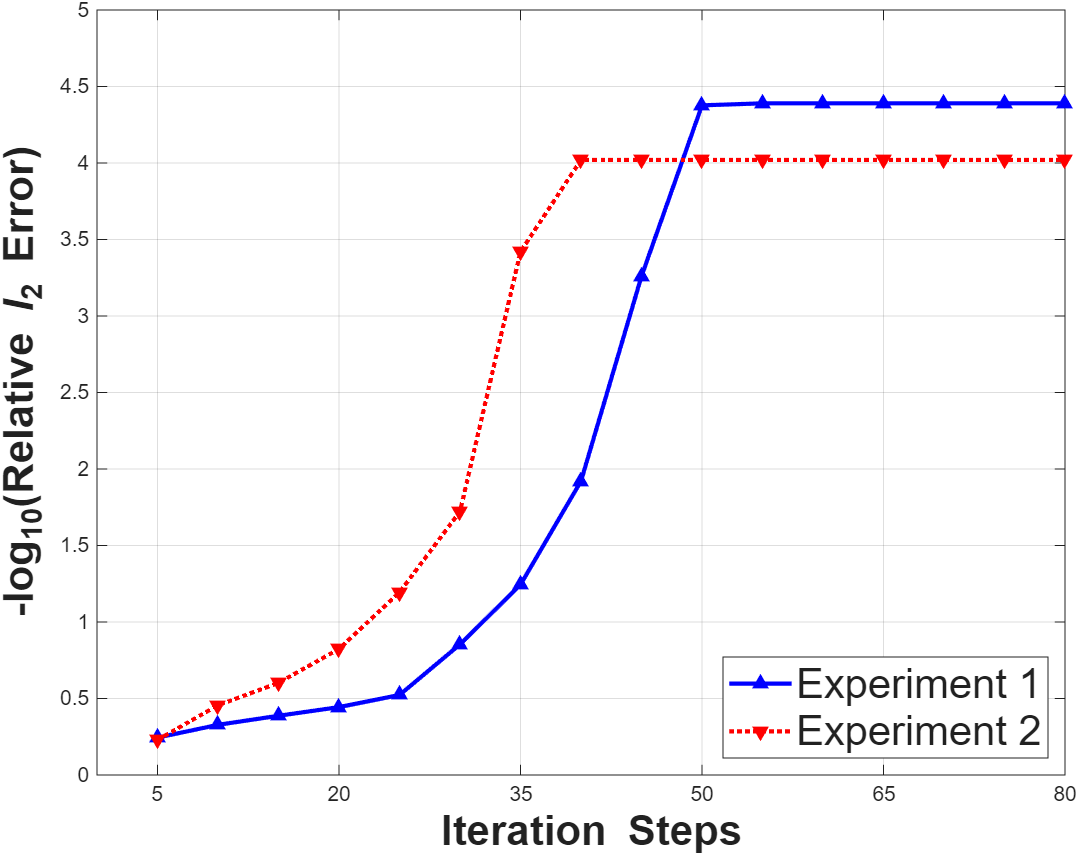}
 \caption{ The negative  base-10 logarithm of the relative $\ell_2$ error of $\{x^n\}_{n\in\mathbb{N}}$
  generated by the IRLSTLp algorithm with varying iterative steps $n$.
Experiment 1 uses a  $64 \times 256$ standard Gaussian matrix
and a randomly generated true signal at the sparsity $k = 28$ while
Experiment 2 uses a  $64 \times 1024$ Gaussian matrix
and a true signal at $k = 20$.}\label{fig-convergence}
\end{figure}

{\color{red}
We record
$-\log_{10}(\frac{\|x^{n}-x_0\|_2}{\|x_0\|_2})$ every 5 outer iterations,
where $x^n$ is recoved by IRLSTLp and $x_0$ is the real signal.
The parameters are set to $a=1$, $p=0.8$, and $\kappa=3$.
In both runs in Figure \ref{fig-convergence}, this quantity increases rapidly after
the initial iterations and then
gradually stabilizes around a certain value  after approximately 40 to 50 iterations,
which indicates  the convergent behavior from the experimental point of view.
They complement the limitation stated in Remark \ref{limitations}.
}

\section{Conclusions}\label{s5}

This article introduces a minimization framework via a nonconvex
two-parameter penalty function $P_{a,p}$ with
$a\in(0,\infty)$ and $p\in (0,1]$,
which reduces to the TL1 minimization model investigated by Zhang and Xin in \cite{ZX17,ZX18}.
We propose the relaxation degree RD$_P$ for a normalized separable penalty $P$.
{\color{red}This relaxation degree provides a dimension-dependent
feature of the unit level set. Moreover, Proposition
\ref{prop-RobinHoodRD} links it to strict concavity and the Robin Hood criterion,
and the numerical results show that it serves as an empirical indicator correlated with the performance
within the $P_{a,p}$ family.} Using the relaxation degree, we show that
the proposed $P_{a,p}$ has a smaller normalized diagonal intercept than $P_p$
and $\ell_a^p$ in the comparisons considered above.
Moreover, applying the sparse convex-combination technique
developed by Cai and Zhang in \cite{CZ14} and, independently, by
G. Xu and Z. Xu in \cite{xx13} for $\ell_1$ and Zhang and
Li in \cite{ZL19} for $\ell_p$ with $p\in (0,1]$, we also establish the exact
and the stable sparse signal recovery based on the restricted isometry
property (RIP), whose upper bound reduces, when $p\in (0,1]$ and
as $a\to\infty$, to the sharp RIP bound obtained
by Zhang and  Li in \cite{ZL19} and, especially, recovers the well-known sharp bound
$\delta_{2s}<\frac{\sqrt{2}}{2}$ when $p=1$.

The IRLSTLp algorithm discussed in this article consists of
a modified iteratively re-weighted least squares
method and the difference of convex functions algorithm (DCA).
{\color{red}Numerical experiments demonstrate that the proposed $P_{a,p}$ minimization,
 to some extent, achieves improved performance over DCATL1,
IRucLq, and DCA $\ell_1-\ell_2$ and offers more
adaptability.}

\bigskip

\noindent\textbf{Author Contributions}\quad All authors have contributed
equally to the writing of this manuscript. All authors
have reviewed the manuscript.

\medskip

\noindent\textbf{Acknowledgements}\quad The authors would like to thank
three referees for their insightful and constructive comments which indeed
greatly improve the quality of this article.

\medskip

\noindent\textbf{Funding}\quad This project is partially supported by the National
Natural Science Foundation of China (Grant Nos. 12431006, 12371093, and 12371094),
the Beijing Natural Science Foundation (Grant No. 1262011), and
the Fundamental Research Funds for the Central Universities (Grant No. 2253200028).

\medskip

\noindent\textbf{Data Availability}\quad No datasets were
generated or analysed during the current study.

\section*{Declarations}

\noindent\textbf{Conflicts of Interest}\quad There is no conflict
of interest for any of the authors of this article.
\medskip

\noindent\textbf{Competing interests}\quad The authors declare no competing interests.

\bigskip

\noindent Ziwei Li

\medskip

\noindent Institute of Applied Physics and Computational Mathematics,
Beijing 100088, The People's Republic of China

\smallskip

\noindent{\it ORCID:} \texttt{https://orcid.org/0009-0009-4208-3856}

\noindent{\it E-mail:} \texttt{zwli@buct.edu.cn}

\bigskip

\noindent Wengu Chen

\noindent Institute of Applied Physics and Computational Mathematics,
Beijing 100088, The People's Republic of China

\smallskip

\noindent{\it ORCID:} \texttt{https://orcid.org/0000-0002-1751-0379}

\noindent{\it E-mail:} \texttt{chenwg@iapcm.ac.cn}
\bigskip

\noindent Huanmin Ge

\medskip
\noindent School of Sports Engineering, Beijing Sport University, Beijing, 100084, China
 The People's Republic of China

\smallskip

\noindent{\it ORCID:} \texttt{https://orcid.org/0009-0005-8843-5221}

\noindent{\it E-mail:} \texttt{gehuanmin@bsu.edu.cn}
\bigskip

\noindent Dachun Yang (Corresponding author)

\medskip

\noindent Laboratory of Mathematics and Complex Systems (Ministry of Education of China),
School of Mathematical Sciences, Institute for Advanced Study,
Beijing Normal University, Beijing 100875,
The People's Republic of China

\smallskip

\noindent{\it ORCID:} \texttt{https://orcid.org/0000-0001-9024-3345}

\noindent{\it E-mail:} \texttt{dcyang@bnu.edu.cn}


\begin{thebibliography}{99}

\bibitem{BDDW08} R. Baraniuk, M. Davenport, R. DeVore and M. Wakin, A simple proof
of the restricted isometry property for random matrices, Construct.
Approx.  28 (2008), 253--263.

\vspace{-0.3cm}

\bibitem{CWX10} T. Cai, L. Wang and G. Xu, New bounds for restricted isometry constants,
IEEE Trans. Inform. Theory 56 (2010), 4388--4394.

\vspace{-0.3cm}

\bibitem{CZ13} T. Cai and A. Zhang, Sharp RIP bound for sparse signal and
low-rank matrix recovery,
 Appl. Comput. Harmon. Anal. 35 (2013), 74--93.

\vspace{-0.3cm}

\bibitem{CZ14} T. Cai and A. Zhang, Sparse representation of a polytope
and recovery in sparse signals and low-rank matrices,
 IEEE Trans. Inform. Theory 60 (2014), 122--132.

\vspace{-0.3cm}

\bibitem{CRT06a} E. J. Cand\`{e}s, J. K.  Romberg and T. Tao,
Stable signal recovery from incomplete and inaccurate measurements,
Comm. Pure Appl. Math. 59 (2006), 1207--1223.

\vspace{-0.3cm}

\bibitem{CRT06b} E. J. Cand\`{e}s, J. K.  Romberg and T. Tao,
Robust uncertainty principles: exact signal reconstruction from highly incomplete frequency information, IEEE Trans. Inform. Theory 52 (2006), 489--509.

\vspace{-0.3cm}

\bibitem{CT05} E. J. Cand\`{e}s and T. Tao,  Decoding by linear programming,
IEEE Trans. Inform. Theory 51 (2005), 4203--4215.

\vspace{-0.3cm}

\bibitem{CS08} R. Chartrand and V. Staneva,  Restricted isometry properties
and nonconvex compressive sensing,
Inverse Problems 24 (2008), Paper No. 035020, 14 pp.

\vspace{-0.3cm}

\bibitem{CSS16}
{\color{red}F. Chen, L. Shen and B. W. Suter, Computing the proximity operator of the $\ell_p$
norm with $0<p<1$, IET Signal Process. 10 (2016), 557--565.}

\vspace{-0.3cm}

\bibitem{CL19} W. Chen and Y. Li,  Recovery of signals under the
condition on RIC and ROC via prior support information,
Appl. Comput. Harmon. Anal. 46 (2019), 417--430.

\vspace{-0.3cm}

\bibitem{CLW18} W. Chen, Y. Li and G. Wu, Recovery of signals under the high order
RIP condition via prior support information, Signal Processing 153 (2018),83--94.

\vspace{-0.3cm}

\bibitem{CDD08} A. Cohen, W. Dahmen and R. DeVore, Compressed
sensing and best $k$-term approximation,
J. Amer. Math. Soc. 22 (2009),  211--231.

\vspace{-0.3cm}

\bibitem{DDFG10} I. Daubechies, R. DeVore, M. Fornasier and C. S. G\"{u}nt\"{u}rk,
   Iteratively reweighted least squares minimization for sparse recovery,
   Comm. Pure Appl. Math. 63 (2010), 1--38.

\vspace{-0.3cm}

\bibitem{D06} D. Donoho, Compressed sensing, IEEE Trans. Inform. Theory 52 (2006),
1289--1306.

\vspace{-0.3cm}

\bibitem{FL12} A. Fannjiang and W. Liao,  Coherence pattern-guided
compressive sensing with unresolved grids, SIAM J. Imaging Sci. 5 (2012),
179--202.

\vspace{-0.3cm}

\bibitem{FR13} S. Foucart and H. Rauhut, A Mathematical Introduction to Compressed Sensing, Birkh\"{a}user, Boston, 2013.

\vspace{-0.3cm}

\bibitem{GCN20} H. Ge, W. Chen and M. K.  Ng, New RIP bounds for
 recovery of sparse signals with partial support information via
 weighted $\ell_p$-minimization, IEEE Trans. Inform. Theory 66 (2020),
 3914--3928.

\vspace{-0.3cm}

\bibitem{GCN21b} H. Ge, W. Chen and M. K. Ng,  On recovery of
sparse signals with prior support information via weighted
$\ell_p$-minimization, IEEE Trans. Inform. Theory 67 (2021),
7579--7595.

\vspace{-0.3cm}

\bibitem{HL25} G. Huang and S. Li,  Low-rank Toeplitz matrix restoration: descent cone analysis and structured random matrix,
IEEE Trans. Inform. Theory 71 (2025), 3950--3956.

\vspace{-0.3cm}

\bibitem{HLZLT25} J. Huang, X. Liu, F, Zhang, G. Luo and  R. Tang, Performance analysis of unconstrained
 minimization for sparse recovery, Signal Proessing 233 (2025), 109937.

\vspace{-0.3cm}

\bibitem{HLSVS} X. Huang, Y. Liu, L. Shi, S. Van Huffel and J. Suykens,
Two-level $\ell_1$ minimization for compressed sensing,
Signal Processing 108 (2015), 459--475.

\vspace{-0.3cm}

\bibitem{HR09}{\color{red}N. Hurley and S. Rickard, Comparing measures of sparsity,
IEEE Trans. Inform. Theory 55 (2009),  4723--4741.}

\vspace{-0.3cm}

\bibitem{LXY13} M. Lai, Y. Xu and W.Yin, Improved iteratively reweighted
least squares for unconstrained smoothed $\ell_q$ minimization,
SIAM J. Numer. Anal. 51 (2013),  927--957.

\vspace{-0.3cm}

\bibitem{LG26} H. Li and X. Geng, Improved RIP-based bounds performance guarantee for sparse signal
recovery via Lorentzian iterative hard thresholding, Signal Processing 241 (2026),  110381.

\vspace{-0.3cm}

\bibitem{LYHX15}  Y. Lou, P. Yin, Q. He and J. Xin, Computing
sparse representation in a highly coherent dictionary based on difference of
$L_1 $ and $L_2$, J. Sci. Comput. 64 (2015),  178--196.

\vspace{-0.3cm}

\bibitem{LF09} J. Lv and Y. Fan, A unified approach to model selection
and sparse recovery using regularized least squares,
Ann. Statist. 37 (2009), 3498--3528.

\vspace{-0.3cm}

\bibitem{DL98} T. Pham Dinh and H. A.  Le Thi, A d.c. optimization
algorithm for solving the trust-region subproblem. SIAM J. Optim.
8 (1998), 476--505.

\vspace{-0.3cm}

\bibitem{PSY} {\color{red}A. Prater-Bennette, L. Shen and J. Yu, Explicit hypergeometric representations and fast
Chebyshev evaluation for the scalar $\ell_p$ proximity operator, arXiv:2606.31951.}

\vspace{-0.3cm}

\bibitem{SY} {\color{red}L. Shen and J. Yu, Explicit series and
a certified hybrid evaluator for the $\ell_p$ proximity operator for $0<p<1$, arXiv:2606.31951.}

\vspace{-0.3cm}

\bibitem{S11} Q. Sun,  Sparse approximation property and
stable recovery of sparse signals from noisy measurements,
IEEE Trans. Signal Process. 59 (2011), 5086--5090.

\vspace{-0.3cm}

\bibitem{S12} Q. Sun,  Recovery of sparsest signals
via $\ell_q$-minimization, Appl. Comput. Harmon. Anal. 32 (2012),
329--341.

\vspace{-0.3cm}

\bibitem{WDZ15}  J. Wen, D. Li and F. Zhu,   Stable recovery of
sparse signals via $\ell_p$-minimization, Appl. Comput. Harmon. Anal.
38 (2015), 161--176.

\vspace{-0.3cm}

\bibitem{WC13}  R. Wu and D.-R. Chen, The improved bounds of
restricted isometry constant for recovery via $\ell_p$-minimization,
IEEE Trans. Inform. Theory 59 (2013), 6142--6147.

\vspace{-0.3cm}

\bibitem{XW13} F. Xu and S. Wang, A hybrid simulated annealing
thresholding algorithm for compressed sensing, Signal Processing
93 (2013), 1577--1585.

\vspace{-0.3cm}

\bibitem{xx13} G. Xu and Z. Xu, On the $\ell_1$-norm invariant convex
$k$-sparse decomposition of signals, J. Oper. Res. Soc. China
1 (2013), 537--541.

\vspace{-0.3cm}

\bibitem{XZWCL10}  Z. Xu, H. Zhang, Y. Wang, X. Chang and Y. Liang,
$L_{1/2}$ regularization, Sci. China Inf. Sci. 53 (2010), 1159--1169.

\vspace{-0.3cm}

\bibitem{YLHX15} P. Yin, Y. Lou, Q. He and J. Xin, Minimization of
$\ell_{1-2}$  for compressed sensing,
SIAM J. Sci. Comput. 37 (2015),  A536--A563.

\vspace{-0.3cm}

\bibitem{YY23} Q. Yu and M. Yang, Low-rank tensor recovery via non-convex regularization,
structured factorization and apatio-temporal characteristics,
Pattern Recognition 137 (2023), Article No. 109343, 14 pp.

\vspace{-0.3cm}

\bibitem{ZW25} K. Zhan and A. Wan, Sparse representation for $\ell_p -\alpha\ell_q$
 minimization and uniform condition for the recovery of approximately $k$-sparse signals with prior support information,
 Signal Processing 235 (2025), 110019.

\vspace{-0.3cm}

\bibitem{ZL18} R. Zhang and S. Li, A proof of conjecture on restricted
isometry property constants $\delta_{tk}$ $(0<t<\frac43)$,
IEEE Trans. Inform. Theory 64 (2018),  1699--1705.

\vspace{-0.3cm}

\bibitem{ZL19} R. Zhang and S. Li, Optimal RIP bounds for sparse
signals recovery via $\ell_p$ minimization,
Appl. Comput. Harmon. Anal. 47 (2019), 566--584.

\vspace{-0.3cm}

\bibitem{ZX17} S. Zhang and J. Xin,  Minimization of transformed $L_1$ penalty:
closed form representation and iterative thresholding algorithms,
Commun. Math. Sci. 15 (2017), 511--537.

\vspace{-0.3cm}

\bibitem{ZX18} S. Zhang and J. Xin,   Minimization of transformed $L_1$
penalty: theory, difference of convex function algorithm, and robust
application in compressed sensing, Math. Program. 169 (2018),
no. 1, Ser. B, 307--336.

\vspace{-0.3cm}

\bibitem{ZLWX14} J. Zeng, S. Lin, Y. Wang and Z. Xu, $L_{1/2}$
regularization: convergence of iterative half thresholding algorithm,
IEEE Trans. Signal Process. 62 (2014), 2317--2329.

\end{thebibliography}
\end{document}